\documentclass[10pt,twoside]{amsart}
\usepackage{amsmath}
\usepackage{amssymb}
\usepackage{graphicx,epsf,amsmath}  
\usepackage{epsf,graphicx}
\usepackage{comment}

\newtheorem{thm}{Theorem}[section]
\newtheorem{pro}[thm]{Proposition}
\newtheorem{lem}[thm]{Lemma}
\newtheorem{lemma}[thm]{Lemma}

\newtheorem{corollary}[thm]{Corollary}

\newtheorem{defi}[thm]{Definition}

\newtheorem{preremark}[thm]{Remark}
\newenvironment{remark}{\begin{preremark}\rm}{\end{preremark}}
\def\NN{{\mathbb N}}
\def\nat{{\mathbb N}}
\def\ZZ{{\mathbb Z}}
\def\integer{{\mathbb Z}}
\def\RR{{\mathbb R}}
\def\real{{\mathbb R}}
\def\TT{{\mathbb T}}
\def\torus{{\mathbb T}}
\def\sphere{{\mathbb S}}
\def\CC{{\mathbb C}}
\def\complex{{\mathbb C}}
\def\B{{\mathcal B}}
\def\E{{\mathcal E}}
\def\tE{{\tilde \E}}
\def\G{{\mathcal G}}
\def\K{{\mathcal K}}
\def\L{{\mathcal L}}
\def\M{{\mathcal M}}
\def\N{{\mathcal N}}
\def\S{{\mathcal S}}
\def\U{{\mathcal U}}
\def\Id{{\rm Id}}

\def\Im{{\rm Im\,}}
\def\Re{{\rm Re\,}}
\def\Ker{{\rm Ker\,}}
\def\avg{{\rm avg\,}}
\def\ep{\varepsilon}
\def\A{{\mathcal A}}
\def\th{\theta} 
\def\om{\omega} 
\def\la{\lambda} 
\def\hth{ {\hat \theta}}
\def\tmu{ {\tilde \mu}}
\def\vp{\varphi} 
\def\eps{\varepsilon} 
\def\dist{{\rm dist}}
\def\range{{\rm range\,}}
\def\Lip{{\rm Lip}}
\def\Tau{{\mathcal T}}

\title[Whiskered tori]{Construction of invariant whiskered tori by a
  parameterization method. Part I: Maps and flows in finite dimensions. }
\author[E. Fontich]{Ernest Fontich}
\address{
Departament de Matem\`atica Aplicada i An\`alisi \\
Universitat de Barcelona \\
Gran Via, 585, 08007 Barcelona, Spain
}
\email{fontich@mat.ub.es} 

\author[R. de la Llave]{Rafael de la Llave}
\address{
Department of Mathematics \\
University of Texas at Austin \\
1 University station  Austin, USA
}
\email{llave@math.utexas.edu}
\author[Y. Sire]{Yannick Sire}
\address{
Univ. Aix Marseille 3 Paul C\'ezanne\\
Avenue Escadrille Normandie-Niemen\\ 
Laboratoire LATP, UMR 6632 Marseille
France
}
\email{sire@cmi.univ-mrs.fr}

\begin{document}
\maketitle
\begin{abstract}

We present theorems which provide the existence of
invariant whiskered tori 
in  finite-dimensional exact symplectic maps and flows.
The method is based on the study  of a functional equation 
expressing that there is an invariant torus.

We show that, given an approximate 
solution of the invariance
equation which satisfies some non-degeneracy conditions, 
there  is a true solution nearby. We call this an {\sl a posteriori} approach.

The proof of the main theorems is based on an iterative method to 
solve the functional equation. 

The theorems do  not assume 
that the system is close to integrable 
nor that it  is written in action-angle variables 
(hence we can deal in a unified way with primary and secondary 
tori). It also does not assume 
that the hyperbolic bundles are trivial and much less that the hyperbolic 
motion can be reduced to constant.

The \textit{a posteriori} formulation allows us to justify
approximate solutions produced by many non-rigorous methods
(e.g. formal series expansions, numerical methods). The iterative method is 
not based on transformation theory, but rather on succesive corrections. 
This makes it possible to adapt the method almost verbatim to 
several infinite-dimensional situations, which we will discuss in a 
forthcoming paper. We also note that the method 
leads to fast and efficient algorithms. We plan to develop these 
improvements in forthcoming papers.

\end{abstract}


\section{Introduction}

The goal of this paper is to prove some results on 
persistence of invariant tori for symplectic and exact symplectic maps 
and flows.

We will assume that the motion on the torus 
is  a Diophantine rotation and that the remaining directions are 
as hyperbolic as allowed by the symplectic structure (if the remaining directions are not void such tori are 
commonly called \emph{whiskered tori}).

More precisely,
as it is well-known, the preservation of the symplectic  
structure, together with the fact that the motion on the torus is a
rotation, implies that the symplectic conjugate direction to the tangent of 
the torus is not hyperbolic. We will assume that the remaining directions 
in the tangent bundle of the phase space at the torus are
spanned by a basis of vectors which contract exponentially in the
future or in the past.

To make the previous statements more precise, we discuss first the 
case of maps. As we will show, results for flows can 
be readily deduced from the ones for maps. 
Given an  exact symplectic map $F$ 
from an exact  symplectic manifold $(\mathcal{M},\Omega=d\alpha)$
into itself
(for the purposes of this preliminary exposition,  we will take 
$\mathcal{M}$ to be an Euclidean manifold, even if
we will indicate how to eliminate this restriction later), and a frequency vector
$\omega \in   \RR^l$,
we seek  an embedding $K:\mathbb{T}^l\rightarrow \mathcal{M}$
satisfying 
\begin{equation}\label{preliminary}
(F \circ K)(\th)=K(\th+\omega),\qquad  \th \in \mathbb{T}^l = \RR^l/\ZZ^l .
\end{equation}

Equation \eqref{preliminary} implies 
that the range of $K$ is invariant under 
$F$. If $K$ is an embedding, we obtain that 
$K(\TT^l)$ is a torus
contained in $\M$, invariant by $F$
and that the dynamics on it is, up to a change of 
coordinates, just a rotation  of rotation vector $\omega$.

The main result of this paper will show that if 
we can find a function $K$ which satisfies some 
non-degeneracy assumptions and which satisfies
\eqref{preliminary} up to a sufficiently small 
error, then there is a true solution nearby.

Differentiating the functional equation \eqref{preliminary} 
with respect to $\th$ one gets
\begin{equation*}
DF(K(\th))DK(\th)=DK(\th+\omega). 
\end{equation*}    
Geometrically,  this shows that the tangent vector-field 
$DK(\th)$ is invariant and does not grow or contract 
under iteration of the action by the map.

As we will see in more detail in 
Section \ref{sect:normal}, 
if the map preserves the symplectic form $\Omega$ and $K$ is 
a solution of the invariance equation, there 
exists an analytic matrix valued function  $ A(\th)$, such that 
\begin{equation} \label{derivatives} 
DF(K(\th))[J(K)^{-1}DK\,N](\th)= DK(\th+\omega) A(\th) 
+[J(K)^{-1}\,DK\,N ](\th+\omega)) ,
\end{equation}
where $J$ is the matrix representation of the symplectic form and
$$
N(\th)=[DK(\th)^\top DK(\th)]^{-1}.
$$

As a consequence, 
$[J(K)^{-1}DK\,N](\th)$ cannot grow more that polynomially. 
Hence we obtain that the center subspace of $T_{K(\th)} \M$ is at least a $2l$-dimensional space 
spanned by $DK(\th)$ and $[J(K)^{-1}DK\,N](\th)$ (we will show that range $DK(\th) \cap $ range $[J(K)^{-1}DK\,N](\th) = \{0\}$ because 
the image of the torus is a co-isotropic manifold.)

For approximately invariant systems, the previous identities are just approximate  
and this implies that the center direction is at least $2 l$. 
We will assume that indeed the dimension of the center subspace is exactly
$2 l$. That is, we will assume that the tori we consider are
as hyperbolic as allowed by the fact that the motion on them are
rotations and that the system preserves the symplectic structure.

The main non-degeneracy assumptions on the approximate solution 
are a) that  the other directions in $T_{K(\th)} \mathcal{M}$ are 
hyperbolic. That is, they are spanned by vectors which contract 
exponentially fast in the future or in the past. 
b) that there is some twist condition, 
that is,  that the matrix $A$ in \eqref{derivatives} 
is invertible.

We will use a KAM 
iterative method to  show that, if we are given a function $K$ 
which 
solves  \eqref{preliminary} up to an error which is 
sufficiently small with respect to the properties of 
the non-degeneracy conditions a), b) above, 
then there is a true solution close to this 
approximate solution.

These results based on validating an approximate
solution --- which we call {\sl a posteriori} ---
imply the usual persistence results (one can 
take as approximate solution of the modified system the 
exact solution for the original one). 
Nevertheless, the {\sl a posteriori} results can be used for 
other purposes. 
 For example {\sl a posteriori} results  can be used to 
validate solutions obtained through any method 
such as numerical approximations or asymptotic methods. 
The validation of Lindstedt series leads to estimates on 
their domain of analyticity. 
The paper \cite{Masdemont05} for instance considers 
Lindstedt series of whiskered tori. 

{\sl A posteriori}  results also lead automatically to Lipschitz dependence 
on paramaters and, with a bit more of work, to
differentiable dependence on parameters. 
The {\sl a posteriori} approach to KAM theorem was emphasized in 
\cite{Moser66a,Moser66b,Zehnder75,Zehnder76}. There, it was pointed 
out that this {\sl a posteriori} approach automatically allows 
to deduce results for finitely differentiable systems.
We refer the reader to \cite{Llave01c} for 
a comparison of different KAM methods. 

In the present paper we deal with 
finite-dimensional maps and flows.
In the forthcoming second part of it we consider coupled map lattices
\cite{FontichLS08b}. The
case of Partial Differential Equations, which can be treated in a
similar way but involves technical difficulties, is postponed to a
forthcoming paper \cite{LlaveS07}.

Results  on whiskered tori  similar to the  finite-dimensional ones 
of this paper
have been considered several times in the literature. 
The first ones are \cite{Graff74,Zehnder76}.

The approach in \cite{Zehnder76} --- which also 
takes the {\sl a posteriori} format --- is based on \cite{Zehnder75}
which consists of finding a change of variables which reduces the system to a normal form 
which obviously possesses an invariant torus. This change of variables
is accomplished by applying a sequence of canonical transformations.  
The method of proof introduced here is not based on successive transformations 
but rather on successive corrections introduced additively. This makes
the estimates easier to establish and it leads to efficient numerical
implementations. 
In order to  be able to solve the equations, we take advantage of
some cancellations due to the preservation of the symplectic structure 
that were also pointed out in \cite{JorbaLZ00, LGJV05, Llave01c}. 

The method of \cite{Zehnder76} proves the result for periodic 
Hamiltonian flows. The result for diffeomorphisms 
is proved in \cite{Zehnder76} by interpolating diffeomorphisms
by  periodic flows and then applying the results 
for periodic flows. The proof we present here
proceeds along the opposite route. We prove first the result for
diffeomorphisms and, then, deduce the result for 
flows  taking
time-one maps. Giving a direct proof of the result for 
 persistence of whiskered tori for maps has been 
suggested  as somewhat desirable in 
J. Moser's Mathematical review for \cite{Zehnder76}. We also provide such a direct proof.
Of course, if one uses normal forms --- as in \cite{Zehnder76} --- it 
is natural to consider flows since the normal forms 
require only the study of the Hamiltonian function, which 
transforms very well. In the method presented here, the 
geometric cancellations are much more transparent in the case of 
diffeomorphisms.

Among other results for finite-dimensional systems, we call 
attention to \cite{Valdinoci00}, which uses a method similar to
that of \cite{Arnold63a}. The paper \cite{Valdinoci00} 
has the advantage that it is 
a first order method (i.e. that each step of the Newton iteration
requires to solve only one small divisors equation). As a consequence,
the size of the gaps among tori in near 
integrable systems, the loss of regularity as a function of 
the Diophantine exponent and the required minimum regularity are smaller than these of the second order methods. A  comparison 
between first and second order methods to prove KAM results can be found in 
\cite{Llave01c}. The paper \cite{Sorrentino03} (see also the sketch in
\cite{Moeckel96}) uses a reduction to a
normally hyperbolic manifold and then applies the standard KAM theorem
for Lagrangian tori. Of course, since normally hyperbolic manifolds
are in general only $C^r$, the above method cannot produce $C^{\infty}$ or
analytic tori. On the other hand, we note that the 
method of \cite{Sorrentino03, Moeckel96} 
leads  to very good regularity conclusions for 
finite differentiable systems and also to 
good estimates on the measure occupied by the tori. 
We also call attention to 
\cite{ZhuLL08,HanY06,LiY06,LiY05,Sevryuk06,Sevryuk99,Eliasson01,Eliasson88} which consider also tori
with hyperbolic and elliptic directions and relax the 
twist conditions and the differentiability 
requirements.  The paper \cite{GallavottiG02} 
considers analytic perturbations which depend only on the angles
of reducible tori satisfying a twist 
condition  and uses a direct resummation method.

As compared with previous finite-dimensional 
results, the method presented here has 
the advantage that one does not need to assume that the hyperbolic bundles are 
trivial (and much less that the motion in the 
hyperbolic directions is reducible to 
a constant linear map). 
Tori with non-trivial invariant bundles appear naturally in
one parameter families after crossing a resonance, see 
\cite{HLlex}.

Also, we do not need to assume that the system is given in action-angle coordinates, something which is convenient if we are working in 
situations when the action-angle coordinates are singular. For instance, in 
the study of diffusion   one is lead naturally
to the study of 
whiskered tori near resonances (see
\cite{DelshamsLS03, DelshamsLS06a}). In this case, the action-angle 
variables are singular and avoiding its use leads to better estimates.

For symplectic ODE's we will also prove a translated tori theorem.
{From} this general version 
we will deduce the results for 
exact symplectic ODE's using a vanishing lemma.
We note that the approach of proving a translated torus
theorem was introduced in \cite{Russmann76} in the one
degree of freedom case. 
 
The method presented here lends itself 
to a very efficient numerical implementation (see \cite{HLS08}). 
The only functions to be considered are functions with a 
number of variables equal to the dimension of the torus itself
(independently of the number of variables of the 
ambient space). Of course, when studying infinite-dimensional 
systems --- PDE's or coupled map lattices or chains of oscillators ---
studying functions with the number of variables of the phase 
space is prohibitive. 
When implementing our method, if we discretize
the tori by $N$ Fourier coefficients, the algorithm presented here 
only requires storage of order of $N$ and a Newton step takes only 
order of $N \log(N)$ operations using the Fast Fourier Transform. This 
seems to be significantly faster than other algorithms. 
Actual implementations are now
being pursued and will
be the subject of a forthcoming paper (see \cite{HLS08}).
We refer the reader to \cite{HLlth,HLlnum, HLlex} for analysis 
and implementation of related algorithms. 

\section{Definitions and notations}

Before presenting the basic ideas 
 and the results of our  method, we introduce some notations and
definitions which are useful for our purposes.  All 
definitions are rather standard and we collect them here mainly to 
set the notation. 


\subsection{Diophantine vectors}  

In the study of invariant tori one needs an arithmetic condition over the frequency vector.
In the case of maps the notion of
Diophantine vector is the following. 

\begin{defi}\label{rotVect}
Given $\kappa>0$ and $\nu \geq l$, we define $D(\kappa,\nu)$ as the set of frequency vectors $\omega \in \real^l$ satisfying the Diophantine condition:
\begin{equation*}
|\omega \, \cdot \,k-n|^{-1} \leq \kappa |k|^{\nu}, \qquad 
\mbox{for all $k\in\integer^l \setminus \{ 0 \}\quad $ and $n\in \integer$},
\end{equation*}
where $ \cdot $  means scalar product, 
$|k|=|k_1|+\dots +|k_l|$ and $k_i$ are the coordinates of $k$. 
We will say that $\omega\in D(\kappa,\nu)$ is Diophantine.
\end{defi}
For vector-fields the corresponding notion is the following.
\begin{defi}\label{rotVF}
Given $\kappa>0$ and $\nu \ge l-1$, we define $D_h(\kappa,\nu)$ 
as the set of frequency vectors $\omega \in \real^l$ satisfying the 
Diophantine  condition:
\begin{equation*}
|\omega \, \cdot \,k|^{-1} \leq \kappa |k|^{\nu},\qquad
\mbox{for all $k\in\integer^l \setminus \{ 0\}$}
\end{equation*}
with the same notation as in Definition \ref{rotVect}.
\end{defi}

The two conditions are closely related since $\omega \in D_h(\kappa, \nu)$ 
with $\omega_1\ne 0$ 
if and only if $(\omega_2/\omega_1, \ldots \omega_l/\omega_1) 
\in D(\kappa', \nu)$ for some $\kappa'$. The  geometric and measure 
properties of the sets of Diophantine vectors have 
been extensively studied. These 
results translate immediately into statements about 
the abundance of KAM tori. 

\begin{defi}
Let $\TT^l = \RR^l/\ZZ^l$ and
$f \in L^1(\torus^l)$. We denote $\avg(f)$ its average on the $l$-dimensional torus, i.e. 
\begin{equation*}
\avg (f)  =\int_{\torus^l}f(\th)\,d\th.
\end{equation*}
\end{defi}

\begin{defi}
Given $\omega \in \RR^l$ we 
introduce the rotation over $\mathbb{T}^l$ of rotation vector $\omega$:
\begin{equation*}
T_\omega(\th)=\th+\omega .
\end{equation*}
\end{defi}

\subsection{Functional spaces, functions and operators}
\label{sec:spaces}

We will denote $D_{\rho}$ the complex 
extension of the torus  of width $\rho$, i.e.  
\begin{equation}\label{Drho}
D_{\rho}=\left \{ z\in \complex^l/\integer^l\mid \, |\mbox{Im}\,z_i| \leq \rho,\,\,i=1,\dots ,l\right \}.
\end{equation}

We denote by $|\cdot |$ the supremum norm on $\mathbb{R}^N$ or
$\mathbb{C}^N$. The sup norm makes several estimates independent of the dimension of
the manifold, which are useful when considering 
infinite-dimensional problems. However on $\ZZ^l $ we will use the norm $|k|=|k_1|+...+|k_l|.$
Furthermore, for finite differentiability purposes, we
consider the following norms: given $g$ analytic, with bounded derivatives in a complex domain
$\mathcal{B}$, and $m \in \mathbb{N}$ we introduce the following $C^m$-norm
for $g$
\begin{equation*}
|g|_{C^m(\mathcal{B})}=\displaystyle{\sup_{0 \leq |k| \leq m } }\;
\displaystyle{\sup_{ z \in \mathcal{B}}} \,|D^kg(z)|. 
\end{equation*}

Let $\A_{\rho}$ be the set of continuous functions on
$ {D_{\rho}}$, analytic in the interior of $D_{\rho}$ with values on a
manifold $\mathcal{M}$, which is assumed to be Euclidean. We endow the space 
$\mathcal{A}_{\rho}$ with the usual supremum norm 
\begin{equation*}
\|u\|_{\rho}=\displaystyle{\sup_{z \in D_{\rho}}}|u(z)|.
\end{equation*}
We have that $(\mathcal{A}_{\rho},\|\cdot\|_{\rho})$ is a Banach
space. 
In particular, $\|u\|_0 = \| u \|_{L^\infty(\torus^l)}$. 

We also recall the following 
convexity property (see \cite[Lemma 12.8]{Rudin74}). 

\begin{pro} \label{prop:interpolation}
Let $0 \le \rho_1 \le \rho_2$ and
assume that $f \in \mathcal{A}_{\rho_2}$. Then, for every 
$\th \in [0,1] $ we have: 

\begin{equation} \label{interpolation} 
\|f \|_{\th \rho_1 + (1 - \th) \rho_2}
\le \|f \|_{\rho_1}^\th \| f\|_{\rho_2}^{1 - \th}. 
\end{equation} 

In particular, taking $\rho_1=0$ and $ \th=1/2$, 
\begin{equation} \label{interpolation2} 
\| f\|_{\rho_2/2} \le \| f \|_{L^\infty(\torus^l)}^{1/2} 
\|f\|_{\rho_2}^{1/2} .
\end{equation} 
\end{pro}

We will also consider spaces of continuous functions 
on $D_\rho$ analytic in its interior and 
taking values on finite-dimensional vector spaces, for instance in spaces of matrices. When endowed with the supremum norm, these 
function spaces are also Banach spaces.

In particular, we will also need some norms of linear maps on the tangent space
$T_{K(\th)}\mathcal{M}$ with $K(\th) \in \mathcal{M}$, where $K$ is an embedding. 
More concretely, let $A(\th)$ be a
continuous linear operator from $T_{K(\th)}\mathcal{M}$ into
itself depending on the variable $\th \in D_\rho$. 
Then we define $\|A \|_{\rho}$ by
\begin{equation*}
\|A\|_{\rho}=\displaystyle{\sup_{\th \in D_\rho} \, \sup_{v\in T_{K(\th)}
    \mathcal{M},\,\,|v|=1}}
    \|A(\th)v\|_\rho. 
\end{equation*}


\section{Setting of the problem and results}

\subsection{Geometric setup}
 
We will consider the 
Euclidean manifolds $\mathcal{M}=\real^{2d}$ 
and $\mathcal{M}=\real^{2d-2l}\times  \real^l \times \torus^l$.
In the second case we can consider the universal covering $\RR^{2d}$ of $\M$ 
and lift the maps defined on $\M$ to maps $\bar F$, defined on $\RR^{2d}$, 
such that $\pi \bar F = \bar F \pi$, where 
$\pi : \RR^{2d} \rightarrow \M$ is the canonical projection.
Even if we pass to the covering we will use the symbol $\M$ to refer to the manifold.
These manifolds obviously admit complex extensions
by considering $\real \subset \complex$ and  
$\torus \equiv \real/\mathbb{Z} \subset \complex/ \mathbb{Z}$. 
As we will see, these different possibilities are
convenient when we consider tori whose embeddings are topologically
different. For example, tori which are contractible 
to tori with different dimensions. 
We will use the same symbol $\M$ for the complex extension of the manifold
or its covering. 

For convenience of notation, we will endow these manifolds 
with the standard Riemannian metric, even if this may not be 
natural for the problem at hand. For us, the metric will 
only play the role to measure sizes and therefore any 
equivalent metric will give a similar result. 
The standard metric
will have the advantage that it will allow us to use matrix 
notation for adjoints.
In matrix notation, thinking of vectors 
as column vectors, we can write
$a^\top b = \langle a, b\rangle$.
On the other hand, we note 
that the length of  vectors will always be the supremum norm
and the norm of matrices will be the operator norm associated to the supremum 
norm on vectors. Of course, for finite-dimensional problems
the supremum norm is equivalent to the Euclidean norm. 

We will assume that the  Euclidean manifold $\mathcal{M}$ has 
an analytic exact symplectic form $\Omega$ with primitive $\alpha$, i.e. 
$\Omega=d\alpha$.  
For each $z \in \mathcal{M}$, let $J(z): T_z \mathcal{M} \rightarrow T_z \mathcal{M}$ be the isomorphism such that 
\begin{equation*}
\Omega(\xi,\eta)=<\xi, J(z) \eta>,
\end{equation*}  
where $<,>$ is the Euclidean product on $T_z\mathcal{M}$.

We will not assume that $J(z)$
has the standard form. We do not assume either that $J$ induces an
almost-complex structure on $T\mathcal{M}$.    
This generality is useful in some applications, (celestial mechanics, numerics, ...)
when we use some system of coordinates --- e.g. polar coordinates --- 
which lead to non-standard symplectic matrices.

\begin{remark}
As we will see in the proof, we are not using much 
the Euclidean structure of the manifolds. In Section~\ref{sec:nontrivial}, we
will present the modifications needed to work on 
other manifolds. 

More precisely, we will show that  it is 
possible to work  out the proof  in a neighborhood $U$ of 
the zero section of a bundle $E^c \oplus E^s \oplus E^u$. 
The bundle $E^c$ will be shown  to be trivial (as a consequence of 
the fact that the motion on the torus is a rotation, the preservation of 
the symplectic structure and the fact that the dimension of 
the center space is $2l$, see Section~\ref{sec:identification}),
 but the others 
bundles --- which correspond to the hyperbolic directions 
need not be trivial. 

We note that equation \eqref{preliminary} is geometrically 
natural since it can be formulated in any manifold.

In the following write up the 
Euclidean structure enters in two 
ways: one is a purely notational one and can be 
eliminated at the price of  a typographical nightmare. 
When we only have approximate solutions,
we will denote the error just as 
$F\circ K(\th) - K(\th + \omega)$ rather 
than  $\exp^{-1}_{K(\th + \omega)}( F\circ K(\th))$. 
We will also compare 
vectors in $T_{F\circ K(\th)} \mathcal M$ with vectors in $T_{K(\th + \omega)}\mathcal M$. 
This can be done by introducing \emph{connectors} as in 
\cite{HirschPPS69}, 
so that what we denote 
$DF\circ K(\th + \omega) DF\circ K(\th)$ 
is really $S_{K(\th + 2 \omega)}^{F\circ K(\th + \omega)}
DF\circ K(\th + \omega) S_{K(\th + \omega)}^{F \circ K(\th)}
DF\circ K (\th)$. See Definition \ref{connector} and the discussion after (particularly equation \eqref{ndeg1bis}).

A second, and more serious way in that the Euclidean space enters
is that, to implement the iterative step in KAM theory,
we will use Fourier series. This certainly requires that the 
functions take values in a vector space. Fortunately, this 
happens only in the center directions. In the hyperbolic 
directions there are geometrically natural ways to solve 
the iterative equation.  This is why we are requiring that 
the center bundle is trivial, but we do not need the triviality
of the hyperbolic bundles. 

Of course, the fact that we work  in a set $U$ as above is 
no loss of generality because, if there is a whiskered torus, 
by the tubular neighborhood theorem, we
can identify a neighborhood of the torus with a neighborhood of
the zero section of the normal bundle.  
\end{remark}

\subsection{Setting of the problem and results for maps}

The main purpose of the theory we are going to develop is to construct invariant tori for exact symplectic maps. 
We recall the following 
\begin{defi} \label{exSymp} Let $(\mathcal{M},\Omega=d\alpha)$ be an exact symplectic manifold. A map $F$ from $\mathcal{M}$ into itself is exact symplectic if there exists a smooth function $W$ on $\mathcal{M}$ such that 
\begin{equation*}
F^*\alpha =\alpha +dW.
\end{equation*}
\end{defi} 
In particular, every exact symplectic map is symplectic, i.e. $F^*\Omega=\Omega$.

Heuristically, our problem is the following: let $F$ be an exact symplectic map and $\omega \in D(\kappa,\nu)$. We want to construct an invariant torus for $F$ such that the dynamics of $F$ on it is conjugated to 
$T_\omega$. To this end, we search for an embedding 
$K: D_\rho \supset \torus^l \rightarrow \mathcal{M}$ in $\mathcal{A}_{\rho}$ such that for all $\th \in D_\rho$, 
$K$ satisfies the functional equation      
\begin{equation}\label{embed}
F ( K (\th))=K(T_{\omega}(\th)). 
\end{equation}
Notice that if \eqref{embed} is satisfied, 
the image under $F$ of a point in the range of 
$K$ will be also in the same range. Hence, since $K$ is an embedding, the range of $K$ will 
be an invariant torus. 

The assumptions of our results will be that we are given a mapping $K$ that satisfies 
\eqref{embed} up to a very small error and which satisfies some 
non-degeneracy and hyperbolicity assumptions. 
We will prove that then, there is a true solution 
of \eqref{embed} close to $K$.
We will also prove that the solution of \eqref{embed} is
unique up to composition on the right with translations.

The exactness of the map $F$ is important for the existence
of a solution to \eqref{embed}.
It is easy to construct examples of symplectic non-exact symplectic maps 
without invariant tori. For instance, consider $\mathcal{M}=\mathbb{T}
\times \mathbb{R}$ with the standard symplectic structure. The
translation in the $\mathbb{R}$-direction is a
symplectic non-exact symplectic map without any invariant torus.     

To construct the desired invariant torus, we consider a parameter $\lambda \in \mathbb{R}^l$ and introduce a
translation term in equation \eqref{embed} depending on $\th$.

We then consider the following functional equation, 
where $G$ is a suitably chosen function of $\th$ 
taking values in $2d\times l$ matrices 
and 
whose unknowns are both $K$ and $\lambda$ 
\begin{equation}\label{translated}
F ( K(\th)) + G(\th)\lambda 
= K ( T_{\omega}(\th)). 
\end{equation}
The introduction of this parameter $\lambda$ will allow us to sidestep
several technical complications and then we will show that, since $F$ is
exact symplectic, the geometry implies that $\lambda=0$. 
The fact
that the dimension of the parameter $\lambda$ is $l$ is important for
our purpose. We also mention that it is possible to use the parameter $\lambda$ to weaken non-degeneracy conditions by taking $\lambda \in \mathbb{R}^{2l}$ instead of 
$\mathbb{R}^l$.  In such a case, $G$ is a $2d \times 2l$ matrix.

\begin{remark}
The introduction of the parameter 
$\lambda$ is also motivated by numerical calculations (see \cite{HLS08}).
It leads to more stable computations. More importantly, 
it is useful in the numerical 
computation of \emph{secondary} tori (i.e. tori generated 
by resonances, which have some contractible directions).

\end{remark}

We go through a KAM technique to prove the existence of such a pair $(\lambda,K)$. To this end, we introduce the operator $\mathcal{F}_{\omega}$ 
\begin{equation}\label{operatordefined}
\mathcal{F}_{\omega}(\lambda,K)=F\circ K + G \,\lambda -K\circ T_{\omega},
\end{equation}
where 
\begin{equation}\label{gdefined}
G = [J(K_0)^{-1} DK_0]\circ T_\omega
\end{equation} 
is a function defined on $\TT^l$ and where $K_0$ stands for an approximate whiskered torus. We will write $G$ instead of its explicit form in many of the following 
results. As we will see
later, the important property of $G$ is that translations 
along the direction of $G$ can change the cohomology of 
the pushforward in the center directions. 

The method is based on a careful study of the linearization (around a given pair $(\lambda,K)$) of the operator $\mathcal{F}_{\omega}$. We will show that this linear operator is approximately invertible in a suitable sense. 

For that,  we have to introduce several non-degeneracy conditions. 

\begin{defi}\label{ND}
Given $\lambda \in \RR^l$ and an embedding  $K: D_\rho\supset \TT^l \to \M$ we say that 
the pair $(\lambda,K)$  is non-degenerate 
for the functional equation \eqref{translated} (and we denote $(\lambda,K) \in ND(\rho)$) 
if it satisfies the following conditions:
\begin{itemize}
\item {\sl Spectral condition:} the tangent space $T_{K(\th)}\mathcal{M}$ has an
  invariant splitting for all $\th \in \torus^l$,
\begin{equation}\label{splitting}
T_{K(\th)}\mathcal{M}=\mathcal{E}^s_{{K(\th)}}\oplus \mathcal{E}^c_{{K(\th)}}
\oplus \mathcal{E}^u_{{K(\th)}},
\end{equation} 
where $\mathcal{E}^s_{{K(\th)}}$, $\mathcal{E}^c_{{K(\th)}}$ and $\mathcal{E}^u_{{K(\th)}}$
are the stable, center and unstable invariant spaces
respectively, i.e.
$$DF(K(\th))\mathcal E^{s,c,u}_{K(\th)}=\mathcal E^{s,c,u}_{K(\th+\omega)}.$$

This splitting is analytic in $\th$. 
To this splitting we associate the projections 
$\Pi^s_{K(\th)}$, $\Pi^c_{K(\th)}$ and $\Pi^u_{K(\th)}$
respectively, which are analytic with respect to $\th$.  

Moreover, the splitting \eqref{splitting} is characterized by
asymptotic growth conditions (co-cycles over
$T_{\omega}$): there exist $0<\mu_1,\mu_2 <1$ , $\mu_3 >1$ such that
$\mu_1 \mu_3 <1$, $\mu_2 \mu_3 <1$ and $C_h>0$
such that for all $n \geq 1$ and $\th \in D_\rho$ 
\begin{equation}\label{ndeg1}
\begin{split}
|(DF)_{}&\circ K\circ T^{n-1}_{\omega}(\th)\times \dots \times
(DF)_{ }\circ K(\th) v| \leq C_h\mu_1^n |v|\\
&\iff v \in \mathcal{E}^s_{{K(\th)}} 
\end{split}
\end{equation} 

\begin{equation}\label{ndeg2}
\begin{split}
|(DF)^{-1}&\circ K\circ T^{-(n-1)}_{\omega}(\th)\times \dots \times
(DF)^{-1}\circ K(\th) v| \leq C_h\mu_2^n |v|\\
&\iff v \in
\mathcal{E}^u_{{K(\th)}}
\end{split} 
\end{equation} 
and
\begin{equation}\label{ndeg3}
\begin{split}
|(DF)_{ }&\circ K\circ T^{n-1}_{\omega}(\th)\times \dots \times
(DF)_ {}\circ K(\th) v| \leq C_h\mu_3^n |v| , \\
|(DF)_{ }^{-1}&\circ K\circ T^{-(n-1)}_{\omega}(\th)\times\dots \times
(DF)_{ }^{-1}\circ K(\th) v| \leq C_h\mu_3^n |v|\\
&\iff v \in
\mathcal{E}^c_{{K(\th)}}.
\end{split} 
\end{equation}

\item
Furthermore, we assume that the dimension of the center subspace is $2l$.

That is, the torus is as hyperbolic as allowed by the 
symplectic structure and there are no elliptic directions in 
the normal direction.

\item {\sl Twist condition:} We introduce the notation
\begin{equation}
\label{eq:definitions} 
\begin{split}
& N(\th)=[DK(\th)^\top DK(\th)]^{-1}, \\
& P(\th)=DK(\th)N(\th).
\end{split}
\end{equation}

Assume that the  averages on $\torus^l$ of the matrices  
\begin{equation} \label{Q}
Q_{ }(\th)= DK(\th+\omega)^\top J(K(\th+\omega)) G(\th)
\end{equation}
and 
\begin{equation}\label{A}
A (\th)
=P(\th+\omega)^\top\Big[[DF(K) J(K)^{-1}\ P](\th)-[ J(K)^{-1}P](\th+\omega)\Big]
\end{equation}
are non-singular. 
\end{itemize}
\end{defi}

\begin{remark}
With a view to applications, we note that in Proposition
\ref{PropDEG}, we will show that we can deduce the existence of an
invariant splitting from the existence of an approximately invariant
one which satisfies the hyperbolicity Conditions
\eqref{ndeg1}--\eqref{ndeg3}. 
Consequently,
Definition \ref{ND} can be verified with a finite precision
calculation on a given numerical approximation.   
We anticipate that the basic idea is that, if we can 
verify that for some operator $B$ we have $\|B^N\| \le \mu^N < 1$ for some $N > 0$, it follows that 
$\| B^n \| \le C \mu^n $ for all $n > 0$. This gives a way to 
obtain all inequalities from finite computations. 
\end{remark}

\begin{remark} \label{remDK}
Note that since $K$ is an embedding  --- hence $DK(\th)$ is one to one for all $\th$ ---
and $d\ge l$ we have that $DK(\th)^\top DK(\th) $ is invertible  for all $\th$.
\end{remark}

\begin{remark} 
If we take $G(\th) = J(K_0(\th+\omega))^{-1} DK_0(\th+\omega) $, 
$Q$ becomes 
$ DK(\th + \omega)^\top J(K_0(\th+\omega)) J(K(\th+\omega))^{-1} DK_0(\th+\omega)
\approx N(\th+\omega)^{-1} $  
and hence one of the twist conditions becomes automatic
because, under the smallness assumptions,  $\avg(Q) := \int_{\TT^l} Q(\th) \, d \th
\approx \avg(N^{-1})
$ and $\avg(N^{-1})$ is invertible.
Indeed,
assume that $v\in \Ker (\avg (N^{-1}))$. Then $v^\top \avg(N^{-1}) v =0$. This last expression is approximately 
$$
0=\int_{\TT^l} v^\top (DK^\top DK)(\th ) v \, d \th = 
\int_{\TT^l} |DK(\th ) v |^2\, d \th 
$$
which implies $DK(\th) v = 0$ for all $\th\in \TT^l$. Since $DK(\th)$ is one to one 
for all $\th$ we obtain $v=0$. Hence the 
condition on the invertibility of $\avg(Q)$ is just a quantitative
statement of the fact that $K$ is indeed an embedding. 
The condition on $A$ is a twist condition.

\end{remark}

\begin{remark}
Note that if the torus $K$ was exactly invariant (i.e. $F \circ K= K
\circ T_{\omega}$) then 
\begin{equation*}
DF \circ K\circ T^{n-1}_{\omega}(\th)\times\dots \times
DF \circ K(\th)=DF^n \circ K (\th),
\end{equation*} 
so that Conditions \eqref{ndeg1}--\eqref{ndeg3} are
the usual growth conditions in the theory of normally hyperbolic
manifolds (see
\cite{Fenichel74, HirschPS77,Pesin04}). Of course, for our applications, we
only assume that the tori are approximately invariant.

When the manifolds are Euclidean, the conditions 
\eqref{ndeg1}--\eqref{ndeg3} make perfect sense. 
Nevertheless, if the phase space is a general manifold $\M$, 
we have $DF (K(\th)) : T_{K(\th)} \M \rightarrow T_{F\circ K(\th)}\M$. 
If $F\circ K(\th) \ne K(\th + \omega)$, then, we 
should write the conditions \eqref{ndeg1}-\eqref{ndeg3}
using connectors (see \cite{HirschPPS69}).

We recall that
\begin{defi}\label{connector} 
A connector $S_{x}^y$ is an isomorphism 
from $T_y \M$  to $T_x \M$, defined when $d(x,y)$ is 
small enough, such that $S_x^x = \Id$ and $S_x^y S_y^z = S_x^z$, 
when both make sense.
\end{defi}

A concrete way of implementing the connectors
is to take parallel transport along the shortest geodesic 
joining $x,y$ (equivalently, the differential of the exponential map).

In the case that we formulate the result in a general manifold,
\eqref{ndeg1} should be written
\begin{equation}\label{ndeg1bis}
\begin{split}
|(DF)_{}&\circ K\circ T^{n-1}_{\omega}(\th)
S_{K\circ T^{n-1}_{\omega}(\th)}^
{F\circ K\circ T^{n-2}_{\omega}(\th)} 
 \times \dots \times
 S_{K(\th+\omega)}^{F(K(\th))}
 ( DF)_{ }\circ K(\th) v| \leq C_h\mu_1^n |v|\\
&\iff v \in \mathcal{E}^s_{{K(\th)}} 
\end{split}
\end{equation} 
and analogously the others. 
\end{remark}

\begin{remark}
The technical reason why we introduced the extra parameter $\lambda$ in
\eqref{translated} is the following: in the iteration of the KAM
scheme, one has to prove that some equations are approximately solved up to
a quadratic error. To this end, we have to show that some averages are
quadratic in the error. To avoid these
technicalities, we introduce this parameter $\lambda$ which allows us to
cancel some terms in the equation so that we can reach the suitable approximate solution  
(See Propositions \ref{approximateSol} and \ref{solCenter}). Then we use the exact symplecticness  of the map to keep the parameter $\lambda$ under control. 
\end{remark} 

We can now state our main theorem, which provides the existence of a
solution $K$ to the functional equation \eqref{embed}
with $F$ exact symplectic, provided we are
given a sufficiently approximate one.   

\begin{thm}\label{existence}
Let $\omega \in D(\kappa,\nu)$ for some $\kappa>0, \nu \geq l$. Assume that 
\begin{enumerate}
\item $F:\U \subset \mathcal{M} \rightarrow \mathcal{M}$ is an exact symplectic
  map and $\U$ is an open connected set, which we will assume without 
loss of generality has a smooth boundary. 
\item $K_0\in ND(\rho_0)$ (the embedding $K_0$ is non-degenerate) in
  the sense that it satisfies the spectral condition in Definition
  \ref{ND} and the average on $\torus^l$ of the matrices 
  $Q_{0 }(\th)$ and $A_{0 }(\th)$ 
are non-singular, where $Q_{0 }$ and $A_{0 }$ are as $Q$ and $A$ in Definition \ref{ND} with $K=K_0$.
\item The map $F$ is real analytic and it can be extended holomorphically to some complex neighborhood of the image under $K_0$ of $D_{\rho_0}$: 
\begin{equation*}
B_r=\left \{ z \in \complex^{2d} |\; \exists \th \in \{|\Im \th| < \rho_0\} 
\mbox{ s.t. } |z-K_0(\th)| <r \right \}, 
\end{equation*}   
for some $r>0$ and such that $|F|_{C^2(B_r)}$ is finite. 
\end{enumerate}   
Denote $E_0=F\circ K_0-K_0\circ T_\omega$ the initial error. Then there exists a constant $C>0$ depending on $l$,  $\nu$,
$|F|_{C^2(B_r)}$, $\|DK_0\|_{\rho_0}$, $\|N_0\|_{\rho_0}$,
$\|A_0\|_{\rho_0}$, $|(\avg (A_0))^{-1}|$, $|(\avg (Q_0))^{-1}|$, $|J|_{C^1(B_r)}$ and the norms of the
projections $\|\Pi^{c,s,u}_{K_0(\th)}\|_{\rho_0}$ such that, if $E_0$ satisfies the estimates
\begin{equation} \label{constantIter}
C\kappa^4 \delta^{-4\nu} \|E_0\|_{\rho_0} <1
\end{equation}
and 
\begin{equation*}
C\kappa^2 \delta^{-2\nu} \|E_0\|_{\rho_0} <r,
\end{equation*}
where $0 < \delta \le \min(1,\rho_0/12)$ is fixed, then there exists an embedding $K_{\infty} \in ND(\rho_{\infty}:=\rho_0-6\delta)$ such that  
\begin{equation*}
F\circ K_{\infty}=K_{\infty}\circ T_{\omega}. 
\end{equation*}
Furthermore, we have the following estimate   
\begin{equation}\label{estimation}
\|K_{\infty}-K_0\|_{\rho_{\infty}} \leq C \kappa^2 \delta^{-2\nu} \|E_0\|_{\rho_0}. 
\end{equation}    
\end{thm}

\begin{remark}
The previous theorem provides a construction of whiskered tori without
assuming the existence of action-angle variables for the original
system. Moreover, the method of proof does not involve the sequence of transformations by symplectomorphisms, which is often used to prove this kind of results, but hard to implement numerically.   
\end{remark}

\begin{remark} 
It is important to remark that the non-degeneracy conditions we use 
in Theorem \ref{existence} depend only on the approximate solution under consideration. 
As one can see, Definition \ref{ND}  only depends 
on averages of the approximately computed solutions. This latter fact is useful in
the validation of numerical computations. Indeed, numerical computations provide an approximate solution
and this is the only information that is available. 
The non-degeneracy conditions needed to apply Theorem~\ref{existence} 
can be verified by straightforward computations on the numerical approximation.

This leads directly to the so-called \emph{small twist theorems}. 
See Section \ref{sec:smalltwist} and in particular 
Proposition~\ref{pro:constants} and the subsequent comments for  more details on the dependence 
of the constants on the non-degeneracy assumptions.
\end{remark}

After introducing an additional term in the functional equation \eqref{embed}, namely  
\begin{equation*}
F \circ K +  (J(K_0)^{-1} DK_0)\circ T_\omega \,\lambda 
= K \circ T_{\omega} 
\end{equation*}
and performing a KAM iteration on $(K,\lambda)$, the final task consists of 
proving that $\lambda_{\infty}=0$ using the geometry. This is done by using the exact symplecticness
of $F$ and a suitable representation of the center
subspace. Indeed, the center subspace
in $T_{K(\th)} \mathcal{M}$, which will be shown to be 
non-trivial, will be very close to  the vector space spanned 
by  $DK(\th)$ and its
symplectically conjugate $J(K(\th))^{-1}DK(\th)$.

\subsection{Uniqueness}

A natural question to ask is whether the embedding
$K$ provided by Theorem \ref{existence} 
is unique.  
Notice that if $K$ is a solution of \eqref{embed}, for any $\sigma\in \RR^l$, $K\circ T_\sigma$ is also 
a solution, hence one can only hope for uniqueness 
up to a composition with a translation on the right. 

The following theorem
 provides a local uniqueness result. We
will see in the next section that there is a simple general 
argument that shows that uniqueness results 
allow us to deduce results for flows from results for 
diffeomorphisms. 

\begin{thm}\label{uniqueness}
Let $F$ be exact symplectic and analytic in $B_r \subset \M$. 
Let $\omega \in D(\kappa,\nu)$ for some $\kappa>0, \nu \geq l$. Assume $K_1, K_2 \in ND(\rho)$ 
 with $\rho>0$ are two solutions of equation
\eqref{embed} such that $K_1(D_{\rho}) \subset B_r,\,K_2(D_{\rho})
\subset B_r$. Then there exists a constant $C>0$ depending on $l$,
 $\nu$, $|F|_{C^2(B_r)}$,
$\|DK_1\|_{\rho}$, $\|N_1\|_{\rho }$, $|J|_{C^1(B_r)}$, $\|A_1\|_{\rho }$, $\|\Pi^{c,s,u}_{K(\th)}\|_\rho$, 
$|(\avg(A_1))^{-1}|$ such that if for some $\tau \in \RR^l$  the norm $\|K_1\circ T_\tau-K_2\|_{\rho}$ 
satisfies 
\begin{equation} \label{cond-unicitat}
C \kappa^2 \rho^{-2\nu} \|K_1\circ T_\tau -K_2\|_{\rho} \leq 1
\end{equation}
with $\delta=\rho/4$, there exists a phase $\tilde \tau \in \real^l$ such that 
$K_1 \circ T_{\tilde\tau}=K_2$ in $D_{\rho}$.  
Moreover $|\tilde \tau -\tau| \le  C \kappa^2 \rho^{-2 \nu} \| K_1 - K_2\|_\rho$.
\end{thm}
The proof of this theorem is   postponed to Section \ref{sec:uniqueness}.

\subsection{Result for flows}

As a by-product of the previous uniqueness theorem, we get a result on the existence of
invariant whiskered tori for flows. This follows from a time-one map
argument (see \cite{douady82}). The argument we present here comes from
\cite{LGJV05,CabreFL03a}.   
\begin{thm}\label{flots}
Let $\omega \in D(\kappa,\nu)$ for some $\kappa>0, \nu \geq l$. Let $(S_t)_{t \in \mathbb{R}}$ be the
flow generated by a finite-dimensional analytic exact symplectic vector-field
\begin{equation*}
\frac{d u}{dt}=f(u),
\end{equation*} 
where $u:I\subset \mathbb{R} \rightarrow \mathcal{M}$. Assume that
there exists a time $t=1$ and an embedding $K \in ND(\rho)$ for some
$\rho >0$ such that $S_1 \circ
K(\th)=K(\th+\omega)$ for all $\th \in \mathbb{T}^l$. Then
for all time $t \in \mathbb{R}$, we have 
\begin{equation*}
S_t \circ K(\th)=K(\th+\omega t).
\end{equation*}
\end{thm}

\begin{proof}
If we have $S_1 \circ
K(\th)=K(\th+\omega)$, then for all $t$ this yields 
\begin{equation*}
S_1 \circ S_t \circ K (\th)=S_t \circ (S_1 \circ K)(\th)=S_t
\circ K (\th+\omega).
\end{equation*}
By Theorem \ref{uniqueness}, if $\| S_t \circ K-K\|_\rho$ is sufficiently small, 
which is achieved if $t$ is sufficiently small, this implies that there exists a phase
$\phi(t)$ such that $S_t \circ K(\th)=K(\th+\phi(t))$. From the
flow property $S_{t+s}=S_t \circ S_s$ and the fact that $K$ is one to one, we have
$\phi(t+s)=\phi(t)+\phi(s)$. We now prove that the function $\phi$ is
continuous. The map $K$ from $\mathbb{T}^l$ into its image is
 one to one and continuous over a compact (for the topology of $\TT^l$). Then
its inverse is continuous. This leads to the continuity of the
function $\phi$. Using this fact and the additivity condition we
deduce, that for $t$ small enough, $\phi(t)=\beta t$ for some $\beta\in \RR^l$. 
Then in this case we have
\begin{equation} \label{invarianciaflow}
S_t \circ K(\th)=K(\th+\beta t). 
\end{equation}  
Since both sides of \eqref{invarianciaflow} are analytic with respect to $t\in [0,1]$ we obtain 
the result for all $t\in [0,1]$.
Putting $t=1$ we get $\beta=\omega$.
Expression \eqref{invarianciaflow} shows that the torus $K(\TT^l) $ is invariant by the flow.
Since the torus is compact, the flow on it is defined for all $t\in \RR$ and hence
\eqref{invarianciaflow} holds for all $t\in \RR$.
This ends the proof. 
\end{proof}
In Section \ref{secHamilFD}, we will give a more precise version of
this result and a direct proof (i.e. a proof which does not pass
through a reduction to a time-$1$ map).  
This is useful since the method of proof leads to 
numerical algorithms for differential equations. The direct 
proof can also be used as a model for results for some  ill-posed 
partial differential equations (see \cite{LlaveS07}).

\section{The linearized operator
$D_{\lambda,K}\mathcal{F}_{\omega}(\lambda,K)$}

In this section, we describe the inductive step of the procedure. 
As most of  the KAM proofs, it will be a modification of 
the classical Newton method. 

Using the Taylor theorem, given an approximate solution, we write 
$$
\mathcal{F}_{\omega}(\lambda+\Lambda,K+\Delta) = 
\mathcal{F}_{\omega}(\lambda,K)  + 
D_{\lambda,K}\mathcal{F}_{\omega}(\lambda,K)(\Lambda,\Delta)+ O(|(\Lambda,\Delta)|^2)
$$
and, following the idea of Newton's method, 
 we look for $(\Lambda,\Delta)$ such that 
$\mathcal{F}_{\omega}(\lambda+\Lambda,K+\Delta) $ is quadratically small
so we are lead to consider the following equation
\begin{equation}\label{linear}
D_{\lambda,K}\mathcal{F}_{\omega}(\lambda,K)(\Lambda,\Delta)=-E,
\end{equation}
where $(\lambda,K ) $ is a pair satisfying approximately equation \eqref{translated} with an error $E(\th)=\mathcal{F}_{\omega}(\lambda,K)(\th)$ with 
$\th \in \torus^l$. Using the definition of the operator $\mathcal{F}_\omega$ in \eqref{operatordefined}, we
see that the derivative of the operator can be written more 
explicitly as:
\begin{equation*}
\begin{split}
D_{\lambda,K}&\mathcal{F}_{\omega}(\lambda,K)(\Lambda,\Delta)(\th)
=G(\th) \Lambda
+
DF_{}(K(\th))\Delta(\th)-\Delta(\th+\omega) .
\end{split}
\end{equation*}
The study of the Newton equation \eqref{linear} is mainly done in three steps:
\begin{itemize}
\item 
One projects equation \eqref{linear} on the hyperbolic space and the center space,
by using the invariant splitting (see Definition \ref{ND}). 
\item  
One reduces the equation of the projection on
  the center subspace to two classical small divisors equations.
  Thanks to a suitable change of coordinates on the tangent space
  (which does not use action-angle variables) these
  equations are then solved approximately (i.e. up to quadratic error) by using
  the extra variable $\Lambda \in \mathbb{R}^l$.    
\item 
One solves (with ``tame'' estimates) the equations 
corresponding to the 
projections onto the stable and unstable invariant subspaces,
by using the conditions on the co-cycles over $T_{\omega}$.
\end{itemize}

\begin{remark}
We note that the equation on the center subspace will not be solved
exactly. We will just solve it up to quadratic errors.  The reason 
is that the change of variables mentioned in the above discussion 
will be constructed taking advantage of approximate identities 
obtained by differentiating with respect to $\th$ the equation for the initial error and applying geometric
identities.  The procedure of comparing the linearized Newton 
equation with the equations that appear taking derivatives 
is very common in KAM theory. It is certainly used 
systematically in \cite{Moser66a, Moser66b, Zehnder76}.  
See \cite[Section 5]{Zehnder75} for some remarks on the relation 
of these identities with a group structure of conjugacy problems. 
We note that some of these remarks in the above references work also for 
some semi-conjugacy problems. 
\end{remark} 

Of course, the above-mentioned strategy uses the non-degeneracy assumptions. 
In subsequent sections, we will show that these assumptions are 
changed only by a small amount, so that the procedure can be iterated.

The main goal of this section is to prove the following result.
\begin{lemma}\label{main}
Consider the linearized equation
\begin{equation}\label{lin-eq}
D_{\lambda,K}\mathcal{F}_\omega(\lambda,K)(\Lambda,\Delta)=-E. 
\end{equation}
Then there exists a  constant $C$ that depends on $\nu$, $l$,
$\|DK\|_{\rho}$,  $\|N\|_\rho$, $\|\Pi^{s,c,u}_{K(\th)}\|_{\rho}$,
$ \|G\|_{\rho}$, $|(\avg(A))^{-1}|$, $|(\avg(Q))^{-1}|$
and the hyperbolicity
constants such that assuming 
that  $\delta \in (0, \rho/2)$ satisfies
\begin{equation} \label{smallness1} 
C \kappa \delta^{-(\nu + 1)} ( \| E\|_\rho + \|G\|_\rho|\lambda|) < 1
\end{equation} 
we have
\begin{enumerate}
\item There exists an approximate solution $(\Lambda,\Delta)$
of \eqref{lin-eq},  in the
following sense: there exits a function $\tilde{E}(\th)$ such that
$(\Lambda,\Delta)$ solves exactly
\begin{equation*}
D_{\lambda,K}\mathcal{F}_\omega(\lambda,K)(\Lambda,\Delta)=-E+\tilde{E},
\end{equation*}
with the following estimates: for all $\delta \in (0,\rho/2)$  
\begin{equation} \label{improve1}
\|\Delta\|_{\rho-2\delta} \leq C \kappa ^2 \delta ^{-2\nu}
\|E\|_{\rho},  
\end{equation}
\begin{equation} \label{improve2}
\|D \Delta\|_{\rho-2\delta} \leq C \kappa ^2 \delta ^{-2\nu -1}
\|E\|_{\rho},  
\end{equation}
\begin{equation} \label{improveLam}
|\Lambda| \leq C \|E\|_{\rho},
\end{equation}
\begin{equation}
\|\tilde{E}\|_{\rho-\delta} \leq C \kappa ^2 \delta ^{-(2\nu+1)} \|E\|_{\rho} \|\mathcal{F}_\omega(\lambda,K)\|_\rho.
\end{equation}

\item If $\Delta_1$ and $\Delta_2$ solve the linearized equation
  in the previous approximate sense, then there exists $\alpha \in
  \mathbb{R}^l$ such that for all $\delta \in (0,\rho)$
\begin{equation}\label{estimAvg}
\|\Delta_1- \Delta_2-DK(\th)\alpha\|_{\rho-\delta} \leq C \kappa
^2 \delta ^{-(2\nu+1)} \|E\|_{\rho}
\|\mathcal{F}_\omega(\lambda,K)\|_\rho. 
\end{equation}
\end{enumerate}

\end{lemma}
\begin{remark}
The form of the previous inductive lemma corresponds very 
closely to Zehnder's
implicit function theorem in \cite{Zehnder75}. Once Lemma~\ref{main} is
proved, we then follow the strategy in \cite{Zehnder75}. 
The most crucial step is the verification of how the 
hypothesis of hyperbolicity are changed when the embedding changes 
in the iterative step.
\end{remark}
More precise information on the dependence of the constants 
$C$ on the non-degeneracy conditions will be provided in 
Proposition \ref{pro:constants}.  We anticipate that, roughly speaking , the constants $C$ can be bounded by universal  powers of 
the non-degeneracy constants. We  postpone the precise formulation since 
it will involve some notations  that will be developed along the proof. 
This power dependence on the constants has some applications 
to the study of tori close to resonance and to small twist theorems.
 
We will need the following classical proposition (see
\cite{Russmann76}, \cite{Russmann76a},
\cite{Russmann75}, \cite{Llave01c}) 
which provides existence of a solution together with estimates for small divisors equations.  
\begin{pro}\label{sdrussman}
Let $\omega \in D(\kappa,\nu)$ and assume the mapping $h:\torus^l \rightarrow \mathcal{M}$ is analytic on $D_{\rho}$ and has zero average. Then for any $0 < \sigma <\rho$ the difference equation 
\begin{equation*}
v(\th+\omega)-v(\th)=h(\th)
\end{equation*}
has a unique zero average solution $v:\torus^l \rightarrow \mathcal{M}$, real analytic on $D_{\rho-\sigma}$ for any $0 < \sigma <\rho$. Moreover, we have the estimate 
\begin{equation}\label{estimsd}
\|v\|_{\rho-\sigma} \leq C \kappa\sigma^{-\nu}\|h\|_{\rho},
\end{equation} 
where $C$ only depends on $\nu $ and the dimension of the torus $l$. 
\end{pro}

\begin{remark} It is important for our purposes to have estimates
independent of the dimension of the manifold $\mathcal{M}$  since
in a followup paper \cite{FontichLS08b} we 
apply the  procedure of   this paper in  an infinite-dimensional 
context. 

The independence of the estimates on the number of 
dimensions comes from the fact that we consider the 
supremum norm and the equation is solved component-wise. 
\end{remark}

\subsection{Geometric considerations}

\subsubsection{Isotropic character of the torus} 
\label{sec:isotropic}

We start by recalling  the definition 
of isotropy.
\begin{defi}
Let $(\mathcal{M},\Omega)$ be a symplectic manifold. A submanifold $\N$
of $\mathcal{M}$ is isotropic if $\N \subset \N^{\perp}$, where
$\N^{\perp}$ is the orthogonal space of $\N$ with respect to the $2$-form $\Omega$. 
\end{defi}
We formulate in our framework the well-known fact that a torus supporting
an irrational rotation is isotropic. The manifold $K(\mathbb{T}^l)$ is isotropic
if the pull-back $K^* \Omega_{}(\th)$ vanishes for all $\th \in \mathbb{T}^l$. In
other words, noting 
\begin{equation*}
K^*\Omega_{}(\th)(\xi,\eta)=<\xi,L(\th)\eta>
\end{equation*} 
for all $\xi,\eta \in \mathbb{R}^{l}$, the isotropic character is
equivalent to $$L(\th)=DK(\th)^\top 
J(K(\th))DK(\th)=0$$ for all $\th \in \mathbb{T}^l$. We
first deal with the case of an exact solution of
\eqref{embed} (see Lemma \ref{lagExact}). The approximate case is the purpose of Lemma~\ref{lagApp}. We note that the fact that exactly 
invariant tori are isotropic manifolds remains true for all irrational 
rotations and is well known \cite{Zehnder76}. 
The fact that approximately invariant tori carrying an irrational rotation
are approximately isotropic seems to require that the 
rotation is Diophantine, see \cite{LGJV05}. For 
the sake of completeness, we present the simple proofs of 
both results. 

\begin{lemma}\label{lagExact}
Assume that $\M$ is exact symplectic, $K$ satisfies \eqref{embed}
and $\omega$ is rationally independent.
Then $L(\th)$ is identically
zero. 
\end{lemma}  
\begin{proof}
Since  $F_{}$ is symplectic  we have 
\begin{equation*}
F^* \Omega =\Omega. 
\end{equation*}
Consequently, this yields 
$$ K^*\Omega_{ }=K^*F_{}^*\Omega_{ }=(K\circ T_{\omega})^*\Omega_{ }.$$
Since $\omega$ is rationally independent, $T_{\omega}$ is ergodic and this implies that $K^*\Omega_{}$ is constant and so is $L(\th)$.
Using that  $\M$ is exact symplectic, 
we have that $K^* \Omega = d K^* \alpha$ and, the 
only constant form which is exact is zero. 

Similarly, a computation shows that
$L(\th)$ has the form $DL_1(\th)^\top-DL_1(\th)$ for some matrix
$L_1(\th)$. Since the average on $\mathbb{T}^l$ of $DL_1(\th)$ is zero,
we get the result. 
\end{proof}

\begin{lemma}\label{lagApp}
Assume that $\mathcal M$ is an exact symplectic manifold, $F: B_r \to \mathcal M$ is analytic and symplectic. 
Let $K$
be real analytic on the complex strip $D_{\rho}$ for some $\rho
>0$ and such that $K(D_\rho) \subset B_r$. Assume also that $\omega \in D(\kappa, \nu)$ and denote 
$$E=F \circ K +G\lambda -K\circ T_\omega.$$

 Then there exists a constant $C$ depending on $l$, $\nu$, 
$\|DK\|_{\rho}$, $|F|_{C^1(B_r)}$, $|J|_{C^1(B_r)}$ such that for all $\delta \in
(0,\rho/2)$ we have 
\begin{equation}\label{lag}
\|L\|_{\rho-2\delta} \leq C \kappa \delta^{-(\nu+1)} (\|E\|_{\rho}+ \|G\|_\rho |\lambda|).
\end{equation}   
\end{lemma}
\begin{proof}
We want to estimate the norm of the matrix $L$. Recalling $F^* \Omega = \Omega$, one gets 
$$K^* \Omega -(K\circ T_\omega)^* \Omega=E^*\Omega -(G\lambda)^* \Omega. $$
Performing the same computations as in \cite{LGJV05}, this leads to the following equation 
\begin{equation*}
L-L\circ T_{\omega}=g,
\end{equation*}
where $g$ is a function on $\TT^l$ such that (here we just use Cauchy estimates)
$$\|g\|_{\rho-\delta} \leq C \delta^{-1} (\|E\|_{\rho}+\|G\|_{\rho}|\lambda|).$$

We now make use of Proposition \ref{sdrussman} to complete the proof.  
\end{proof}

Recall that we are assuming that $K$ is an embedding. Hence 
the range of $DK$ is $l$-dimensional. 
 
\subsubsection{Vanishing lemma}

This section is devoted to an estimate which allows to control the
extra parameter $\lambda$ through the iterative step. 
We consider the functional equation 
\begin{equation*}
F\circ K + G \lambda= K \circ T_{\omega} +E,
\end{equation*}
where $F$ is {\sl exact} symplectic (see Definition \ref{exSymp}) and 
$G=[J(K_0)^{-1}DK_0] \circ T_{\omega}$. 

Recall that $\lambda \in \mathbb{R}^{l}$ and $K_0\in ND(\rho_0)$. Note that the
term  $(J(K_0)^{-1}DK_0)\circ T_{\omega}$ is very close to $(J(K)^{-1}DK)\circ T_{\omega}$
and hence close to the center subspace associated to the torus $K(\TT^l)$.

The following lemma provides the desired vanishing result. 
\begin{lemma}\label{vanishing} 
Assume  $F $  maps $\mathcal{M}$ into itself and
  $\omega \in D(\kappa,\nu)$. Let $K \in ND(\rho)$ be a solution of 
\begin{equation} \label{embed3} 
F \circ K+ G \lambda =K \circ T_{\omega}+E,
\end{equation}    
with $G=[J(K_0)^{-1} DK_0] \circ T_{\omega}$ and $\lambda$ is such that 
\begin{equation} \label{smallnessvanishing}
\begin{split} 
& \|E\|_\rho + \|G\|_\rho |\lambda| \leq r, \\
& \|K-K_0\|_{\rho} \leq r, \qquad
 \|DK - DK_0\|_{\rho} \leq r,
\end{split} 
\end{equation} 
where $r > 0$ is sufficiently small (precise conditions will 
be given along the proof).

Assume furthermore that 
\begin{enumerate}
\item
$F$ is exact symplectic.
\item 
$F$ extends analytically to a neighborhood of $K(\TT^l)$.
\end{enumerate}
Then,  there exists a constant $C$ such that 
\[
|\lambda|  \leq C \|E\|_\rho. 
\]
\end{lemma} 
\begin{proof}
We follow a method used in 
\cite{JorbaLZ00}. We refer the reader to the Figure \ref{van} for an illustration of the method.

We denote by
\begin{equation}\label{thetahat}
\hth_i = (\th_1, \ldots, \th_{i-1}, \th_{i+1},\ldots,\th_l)\in \TT^{l-1}
\end{equation}
and similarly
 $\hat \omega_i = (\omega_1, \ldots, \omega_{i-1},
\omega_{i+1},\ldots, \omega_l)\in \RR^{l-1}$.

We also denote $\sigma_{i, \hth_i}:\torus \to \torus^l$ the path
given by 
\begin{equation}\label{sigmai}
\sigma_{i, \hth_i}(\eta) = (\th_1, \ldots, \th_{i-1}, \eta, 
\th_{i+1},\ldots, \th_l).
\end{equation}
We will compute 
$ \int_{\mathbb{T}^{l-1}} \int_{K \circ \sigma_{i, \hth_i}}  F ^* \alpha$
in two different ways.
On one hand, using the fact that $F $ is exact symplectic, we have
\begin{equation}\label{calculation1}
\int_{K \circ \sigma_{i, \hth_i+\hat \omega_i}} F^* \alpha  = 
\int_{K \circ \sigma_{i, \hth_i +\hat \omega_i}} \alpha  + d W\\ 
  = \int_{K \circ \sigma_{i, \hth_i +\hat \omega_i}} \alpha .
\end{equation}

On the other hand, using \eqref{embed3} 
\begin{equation}\label{calculation2}
\int_{K \circ \sigma_{i, \hth_i }} F^* \alpha = 
\int_{F \circ K \circ \sigma_{i, \hth_i}}  \alpha 
 = \int_{(K \circ T_\omega -G  \lambda )\circ \sigma_{i, \hth_i} }
 \alpha + R_i,
\end{equation}
where $|R_i| \le C \| E\|_\rho.$

Since we want to compare the last integrals in 
\eqref{calculation1} and \eqref{calculation2}
it is natural to introduce a
two-cell whose boundary is the difference between the two paths
$K \circ \sigma_{i, \hth_i +\hat \omega_i}$
and 
$(K \circ T_\omega -G  \lambda )\circ \sigma_{i, \hth_i}$. 
We denote $B_{i, \hth_i,\lambda}$ this two-cell, which we 
parametrize by $(\xi,\eta) \in (0,1) \times (0,1)$ as follows:
\[
B_{i, \hth_i, \lambda}(\xi, \eta) = 
K \circ \sigma_{i, \hth_i+ \hat \omega_i}(\eta )   
- G \circ \sigma_{i, \hth_i+ \hat \omega_i}(\eta ) \lambda \xi.  
\]
By Stokes's theorem, since $d \alpha=\Omega$, we have
\begin{equation}\label{tempvani}
\int_{(K \circ T_\omega -G\lambda )\circ \sigma_{i, \hth_i}} \alpha
=\int_{K \circ \sigma_{i, \hth_i +\hat \omega_i}}\alpha+\int_{B_{i, \hth_i, \lambda}} \Omega.
\end{equation}
We have 
\[
\int_{B_{i, \hth_i, \lambda}} \Omega=\int_0^1 \int_0^1 \Omega_{B_{i, \hth_i, \lambda}(\xi,\eta)}(\partial_\xi B_{i, \hth_i, \lambda}(\xi,\eta),\partial_\eta B_{i, \hth_i, \lambda}(\xi,\eta))\,d\xi \,d\eta.
\]
Note that 
\[
\begin{split}
&\partial_\eta B_{i, \hth_i, \lambda} = 
\partial_{\th_i} K\circ \sigma_{i, \hth_i+ { \hat \omega}_i }
(\eta)-\partial_{\th_i} 
G\circ \sigma_{i,\hth_i+\hat \omega_i}(\eta)\lambda \xi
\end{split}
\]
and
\[
\begin{split}
&\partial_\xi B_{i, \hth_i, \lambda} = -G\circ \sigma_{i, \hth_i+\hat\omega_i} (\eta) \lambda.
\end{split}
\]

Using the previous expressions
\begin{align*}
\Omega &_{B_{i, \hth_i, \lambda}(\xi,\eta)}(\partial_\xi B_{i, \hth_i, \lambda}(\xi,\eta),\partial_\eta B_{i, \hth_i, \lambda}(\xi,\eta)) \\
=& - \lambda^\top 
G\circ \sigma_{i,\hth_i+ { \hat \omega}_i }(\eta)^\top
J(B_{i,\hth_i, \lambda}(\xi,\eta))
(\partial_{\th_i} K\circ \sigma_{i,\hth_i+ { \hat \omega}_i }(\eta)-\partial_{\th_i} G\circ \sigma_{i,\hat\th_i+\hat \omega_i}(\eta)\lambda \xi).
\end{align*}

Using \eqref{smallnessvanishing}, we have 
\begin{align*}
J(B_{i,\hth_i, \lambda}(\xi,\eta)) & 
= J(K\circ \sigma_{i,\hth_i+ { \hat \omega}_i }(\eta)) +O(|\lambda|) 
\\
& =  J(K_0\circ \sigma_{i,\hth_i+ { \hat \omega}_i }(\eta)) + O(r) +O(|\lambda|). 
\end{align*}
Therefore, 
we end up with (using the expression of $G(\th)=J(K_0(\th))^{-1}DK_0(\th)$)
\begin{align*}
\int_{B_{i, \hth_i, \lambda}} \Omega  = &-\int_0^1 \int_0^1
\lambda^\top  [DK_0^\top J(K_0)^{-\top} J(K_0) \partial_{\th_i} K_0]\circ \sigma_{i,\hth_i+ { \hat \omega}_i }(\eta)
\, d\xi d\eta \\ & + O( r |\lambda|) +
O(|\lambda|^2)+O(\|E\|_\rho).  
\end{align*}
Joining these expressions for all values of $i$ and integrating over $\TT^{l-1}$ 
we get 
\begin{equation}\label{tempvani2}
\int_{\TT^{l-1}}\int_{B_{i, \hth_i, \lambda}} \Omega = \lambda^\top
\left[  \int_{\TT^l} \tilde Q
\right]
 + O( r |\lambda|) + O(|\lambda|^2)+O(\|E\|_\rho)
\end{equation}
where $\tilde Q=DK_0 ^\top DK_0$.
 Since $ DK_0$ has rank $l$ then the matrix $\tilde Q$ has 
rank $l$.  See Remark \ref{remDK}.


We now integrate with respect to $\hth_i$ both 
\eqref{calculation1} and \eqref{calculation2}. By a simple change of variables we have that the following integrals are equal 
\[
\int_{\torus^{l-1}}\, d\hth_i 
 \int_{K \circ \sigma_{i, \hth_i + \hat \omega_i } }F^* \alpha 
= \int_{\torus^{l-1}}\, d\hth_i 
 \int_{K \circ \sigma_{i, \hth_i  }} F^* \alpha. 
\]
Therefore, {from} \eqref{tempvani} we obtain 
\[
\int_{\TT^{l-1}} 
\int_{B_{i, \hth_i, \lambda}} \Omega = -\int_{\TT^{l-1}} R_i.
\]
Now, equation \eqref{tempvani2}, the fact that $\tilde Q$ is invertible, the assumption 
$$
\|E\|_{\rho}+\|G\|_\rho |\lambda|\leq C, 
$$ 
and $r$ sufficiently small (this is the condition we imposed in 
 \eqref{smallnessvanishing}),
leads to the desired result invoking the implicit function theorem.
\end{proof}
\begin{remark}
The assumption in Lemma \ref{vanishing} that 
$$\|E\|_\rho + \|G\|_\rho |\lambda | \leq C$$
will be an inductive assumption in the iteration of the KAM
method that we will deal with later. 
\end{remark}
\begin{remark}
In the KAM iteration, we will generate a sequence $\left \{ \lambda_n,K_n \right \}_{ n \in \NN}$ of approximations of the solution $(\lambda_\infty,K_\infty)$ of the equation 
$$F \circ K+G \lambda =K \circ T_\omega. $$  
As a corollary of Lemma \ref{vanishing}, the sequence $\left
  \{ \lambda_n \right \}_{n \in \NN}$ in the KAM iteration converges
to $0$ since $\|E_n\|_{\rho_n}$ converges to $0$.
\end{remark}

\begin{figure}[htbp]
\begin{center}
\includegraphics[scale=0.4]{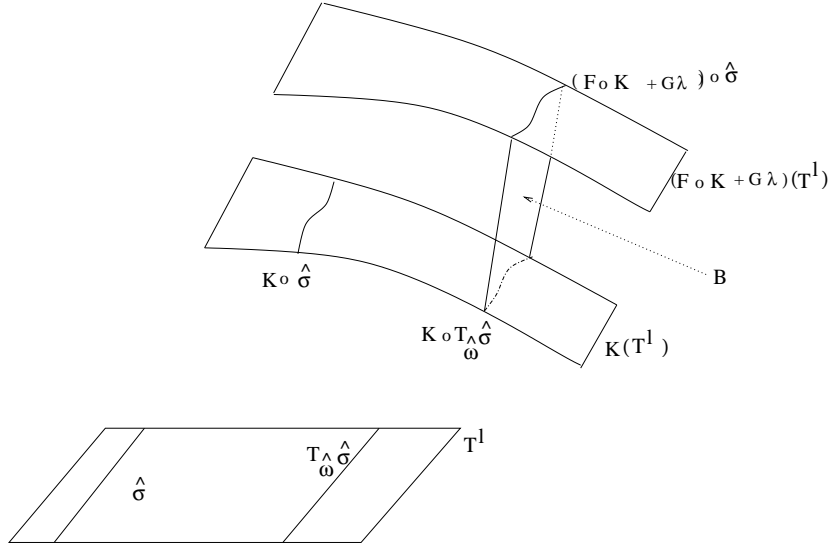}
\end{center}
\caption{\label{van} Illustration of the vanishing lemma}
\end{figure}

\subsubsection{Basis of $\E^c_{K(\th)}$ when $K$ is an exact parameterization}
\label{sec:basis}
To avoid the use of action-angle variables
we are going to perform a change of variables, using the
geometric structure on the tangent bundle.

For that we will first find a useful basis of the center space $\E^c_{K(\th)}$
in the case that $K:\TT^l\to \M$ is a solution of $F\circ K= K \circ T_\om$.

Here we are assuming that the dimension of the center subspace is $2l$ and
hence 
$\mathcal{E}^c_{K(\th)} \sim \mathbb{R}^{2l}$. In \cite{LGJV05},
the authors studied the case when $d=l$, i.e. the dimension of the
range of $K(\mathbb{T}^l)$ is half the dimension of the space
and the tori are Lagrangian submanifolds. 

In \cite{LGJV05} it is shown that, in the Lagrangian case, 
the perturbative equations can be studied very 
conveniently applying the change of variables
given by the following matrix
\begin{equation}\label{change}
[DK(\th), J(K(\th))^{-1}DK(\th)N(\th)].
\end{equation}

In the Lagrangian case, the range of  
\eqref{change} is the tangent space of 
the manifold at $K(\th)$. 
In our case, however, the range of \eqref{change} 
is not the whole space, but it will be a very good approximation of the center space. 
Then, we can apply a method very similar to the method in 
 \cite{LGJV05} for the equations in the center directions. 
The hyperbolic directions will be solved by other methods.

By the symplecticness of $F$ and the dynamical properties
we have that the matrix of the symplectic structure
with respect to the splitting 
$\E^s_{K(\th)} \oplus\E^c_{K(\th)} \oplus\E^u_{K(\th)} $
has the form
\begin{equation} \label{matriuJ}
J(K(\th))  
= 
\begin{pmatrix} 
0 & 0 & J^{su} \\
0 & J^{cc} & 0 \\
J^{us}  & 0 & 0
\end{pmatrix}.
\end{equation}
where $J^{cc}$ is an antisymmetric form and 
$J^{su}(e_s, e_u) = - J^{us}(e_u, e_s)$. 

Indeed from 
$$
u^\top J(K(\th)) v = \Omega_{K(\th)}(u,v) 
= \Omega_{F^n(K(\th))}(DF^n(K(\th)) u,DF^n(K(\th))v), \qquad n\in \ZZ,  
$$
we deduce, sending $n \rightarrow +\infty$ and using the hyperbolic conditions (expansion/contraction properties), that $u^\top J(K(\th)) v=0$ in the following cases
\begin{itemize}
\item $u,v\in \E^s_{K(\th)} $,
\item $u,v\in \E^u_{K(\th)} $,
\item $u\in \E^s_{K(\th)}\cup\E^u_{K(\th)}  $ and $v\in \E^c_{K(\th)} $,
\item $v\in \E^c_{K(\th)} $ and $v\in \E^s_{K(\th)}\cup\E^u_{K(\th)}  $
\end{itemize}
which implies the form \eqref{matriuJ}. See also \cite{dls08}.
The form \eqref{matriuJ} proves that 
$J(K(\th))^{-1} $ sends the center subspace into itself.

Since $\range DK(\th) $ is the tangent space of the torus 
$K(\TT^l)$ and the dynamics on the torus is conjugated to 
a rotation, $DK(\th)\RR^l$ is contained in $\E^c_{K(\th)}$.
Moreover the previous property of $J(K)^{-1}$ implies that 
$J(K(\th))^{-1}DK(\th)\RR^l$ also is contained in  $\E^c_{K(\th)}$.
Instead of $J(K(\th))^{-1}DK(\th)$ 
we will consider the matrix $J(K(\th))^{-1}DK(\th) N(\th)$
where $N(\th)$ is the normalization $l \times l$-matrix 
$N(\th)=[DK(\th)^\top DK(\th)]^{-1}$ introduced in 
\eqref{eq:definitions}. Both have the same range 
because $N(\th)$ is non-singular. 
The  role of $N$ is to provide some normalization for the 
symplectic conjugate.

Now we check that the range of $[DK(\th), J(K(\th))^{-1} DK(\th)N(\th)]$
is $2l$-dimensional. Indeed, assume that there is a linear combination
$$
f= \sum_{j=1}^l\alpha_j DK(\th) e_j + 
\sum_{j=1}^l \beta_j J(K(\th))^{-1} DK(\th)N(\th) e_j  = 0.
$$
Then, for $1\le k\le l$, using the isotropic character of $T_{K(\th)} K(\TT^l)$
\begin{align*}
0 & = \Omega (DK(\th) e_k , f ) 
= \sum_{j=1}^l \beta_j e_k^\top DK(\th)^\top J(K(\th))  J(K(\th))^{-1} DK(\th)N(\th) e_j  \\
& = \sum_{j=1}^l \beta_j \langle e_k, e_j \rangle = \beta_k.
\end{align*}
This calculation shows that $f$ reduces to
$ \sum_{j=1}^l\alpha_j DK(\th) e_j $. Moreover, for  $1\le k\le l$
\begin{align*}
0 & = \Omega (J(K(\th))^{-1} DK(\th) N(\th)e_k , f ) \\
& = \sum_{j=1}^l \alpha_j e_k^\top N(\th)^\top DK(\th)^\top J(K(\th))^{-\top}
J(K(\th))
DK(\th) e_j  \\
& = -\sum_{j=1}^l\alpha_j \langle e_k, e_j \rangle = \alpha_k .
\end{align*}
Hence $\alpha_j=\beta_j=0$ for all $j=1,...,l$. 
We conclude that 
$$
\range [DK(\th), J(K(\th))^{-1} DK(\th)N(\th)] = \E^c_{K(\th)}.
$$


Finally we define 
\begin{equation} \label{definicioMtilde}
\tilde{M}(\th)=[DK(\th), J(K(\th))^{-1} DK(\th)N(\th)].
\end{equation}

\subsection{Solving the linearized equation on the center subspace}

This section is devoted to the study of equation
\eqref{linear} projected on the center subspace. We denote
\begin{equation*}
\Delta^c(\th)=\Pi^c_{K(\th)} \Delta(\th).
\end{equation*}
Projecting the linearized equation \eqref{lin-eq} into the center space, 
we end up with the following equation  
\begin{equation}\label{eqCenter}
\Pi^c_{K(\th+\omega)}G(\th)\Lambda+DF_{}(K (\th))\Delta^c(\th)-\Delta^c(\th+\omega)=-E^c(\th),
\end{equation}
where $E^c(\th) = \Pi^c_{K(\th+\omega)} E(\th)$

In Section~\ref{sect:normal}, we will develop several 
identities and approximate identities that have a 
geometric nature. These identities will be used to
reduce the equation on the center to constant 
coefficients equations of the form considered in
Proposition~\ref{sdrussman}. One important step 
is accomplished in Section~\ref{sec:identification} where
we use the geometric identities and the theory of 
hyperbolic systems to  
obtain  an approximate representation of the center space. 
Once this material is developed, we can establish the main 
result of this Section, Proposition~\ref{solCenter}.

\subsubsection{Normalization procedure} \label{sect:normal}

In the following, we construct a suitable representation for the
matrix $DF_{ }(K(\th))\tilde{M}(\th)$.  Recall that the $2d \times 2l$-matrix $\tilde M$ is given by 
\begin{equation}
\tilde M= [ DK, J(K)^{-1}DK\,N].
\label{eq:formtildeM}
\end{equation}

As a motivation, we first consider the case when $K$ is a solution of \eqref{embed}. We search for a matrix $\mathcal{S}(\th)$ satisfying
\begin{equation}
\label{conjS}
DF_{ }(K(\th))\tilde{M}(\th)=\tilde{M}(\th+\omega)\mathcal{S}(\th),
\end{equation}     
where $\mathcal{S}(\th)$ is upper triangular with identity matrices
on the diagonal. Explicit expressions for $\mathcal{S}$ will be 
given later. 

Differentiating equation \eqref{embed} with respect to $\th$, we get 
\begin{equation*}
DF_{}(K(\th))DK(\th)=DK(\th+\omega). 
\end{equation*} 
This shows that $\mathcal{S}(\th)$ has the form
\begin{equation}
\begin{pmatrix} \Id_l & A (\th)\\ 0_l & B(\th)
\end{pmatrix}, 
\label{formS}
\end{equation}
where $A(\th)$ and $B(\th)$ are $l \times l$ matrices. 
We will see that the choice of the second column
of $\tilde M$  and the symplectic structure forces that 
$B(\th)=\Id_l$. Then, it will be easy to compute an 
expression for $A$. 

Indeed, from \eqref{eq:formtildeM}, \eqref{conjS} and \eqref{formS}
we should have
\begin{equation}
\label{formula-AB}
[DF(K) J(K)^{-1}DK \,N ](\th)=DK(\th+\omega)\,A(\th)
+[J(K)^{-1}DK \,N](\th+\omega)B(\th).
\end{equation}
By the isotropic character of $K(\TT^l)$ we have 
$DK^\top J(K) DK =0$. Hence 
\begin{equation}
[DK ^{\top} J(K)](\th+\om)
[DF(K) J(K)^{-1}DK \,N ](\th)=[DK ^{\top} DK \,N](\th+\omega)B(\th).
\label{eq:formulaB}
\end{equation}

Also by the symplecticness of $F$
$$
J(K(\th+\om))DF (K(\th)) = 
J(F(K(\th))) DF (K(\th)) =  [DF (K)^{-\top} J(K)] (\th).
$$
Then the left-hand side of \eqref{eq:formulaB} becomes
$$
DK ^{\top} (\th+\om)
[DF (K)^{-\top} DK \,N ](\th)
= 
[DK ^{\top}  DK \,N ](\th)=\Id_l. 
$$
With this we conclude that $B(\th) = \Id_l$.

To obtain the expression of $A(\th)$ we multiply  
\eqref{formula-AB} by $(DK\,N)(\th+\om)^\top$. Using $N^\top= N$ and 
$NDK^\top DK = \Id_l$ we get
\begin{equation}\label{defA}
A(\th)=P(\th+\omega)^\top \Big[[DF(K) J(K)^{-1}  P](\th)-[J(K)^{-1}P](\th+\omega)\Big]. 
\end{equation}
We sum up the previous computations in the following lemma. 
\begin{lemma}\label{repEx}
Let $K$ be a solution of equation \eqref{embed}. Then we can write
\begin{equation*}
DF(K(\th))\tilde{M}(\th)=\tilde{M}(\th+\omega)\mathcal{S}(\th),
\end{equation*}
with
\begin{equation}\label{smatrix}
\mathcal{S}(\th)=\begin{pmatrix} \Id_l & A(\th)\\ 0_l &
  \Id_l \end{pmatrix}
\end{equation}
and 
\begin{equation*}
A(\th)=P(\th+\omega)^\top 
\Big[[DF(K) J(K)^{-1} P](\th)-[J(K)^{-1}P](\th+\omega)\Big],
\end{equation*} 
where the notation $P(\th)=DK(\th)N(\th)$ was introduced in 
\eqref{eq:definitions}.
\end{lemma}
The matrix $\tilde{M}(\th)$ is not invertible since it is not
square. However we can derive a generalized inverse for
$\tilde{M}(\th)$. 
As a motivation for subsequent developments, we 
first present Lemma~\ref{isinvariant} which deals with  
the geometric cancellations in the case of an exactly invariant torus. The case of interest for 
a KAM algorithm --- when the torus is only approximately 
invariant --- will be studied in Lemma~\ref{repApp} as a perturbation of 
Lemma~\ref{isinvariant}. 

A straighforward calculation shows that 
\begin{equation}
\tilde{M}^\top J(K) \tilde{M}
=\begin{pmatrix} L &   \Id_l \\ 
-\Id_l & N^{\top}DK^\top J(K)^{-\top} DK \,N  
\end{pmatrix}.
\label{eq:formJM}
\end{equation}

\begin{remark}
When $J^2 = -\Id_{2d}$ we have $J^{-\top} = - J^{\top} = J $
and then 
\begin{equation}
\label{eq:matriuJJJ}
\tilde{M}^\top J(K) \tilde{M}
=\begin{pmatrix} L &   \Id_l \\ 
-\Id_l & N^{\top}L N  
\end{pmatrix}.
\end{equation}   

If moreover $K$ is a solution of $F\circ K  = K \circ T_\om$ the right-hand side
matrix of \eqref{eq:matriuJJJ} reduces to the standard symplectic matrix 
$J_0= \begin{pmatrix}  0&   \Id_l \\ 
-\Id_l & 0
\end{pmatrix}$.
\end{remark}

\begin{lemma} \label{isinvariant} 
Let $K$ be a solution of \eqref{embed}. Then the matrix 
$\tilde{M}^\top J(K) \tilde{M} $ is invertible
and
\begin{equation*}
(\tilde{M}^\top J(K) \tilde{M})^{-1}
=\begin{pmatrix} N^{\top}DK^\top J(K)^{-\top} DK \,N &
  -\Id_l \\ \Id_l & 0 
\end{pmatrix}.
\end{equation*}
\end{lemma}
\begin{proof} 
It follows from \eqref{eq:formJM} and 
the isotropic character of the
invariant torus, i.e. $L=0$.
\end{proof}

We now establish a similar result for approximate solutions, i.e. solutions of \eqref{translated} up to error $E(\th)=\mathcal{F}_{\omega}(\lambda,K)(\th)$. We can expect this type of normalization to be true if the error and $\lambda$ are small enough. 
Following the calculations in Lemma~\ref{repEx},
we obtain:
\begin{equation*}
DF(K (\th))\tilde{M}(\th)=\tilde{M}(\th+\omega)\begin{pmatrix} \Id_l & A(\th)\\ 0_l & \Id_l \end{pmatrix} + O(E^c,DE^c).
\end{equation*}
More precisely, we introduce 
\begin{equation}\label{e}
e(\th)=DF(K(\th))\tilde{M}(\th)-\tilde{M}(\th+\omega)\mathcal{S}(\th),
\end{equation}
where $\mathcal{S}$ is given by \eqref{smatrix}.
If we denote $e(\th)=(e_1(\th), e_2(\th))$, a simple algebraic
computation yields
\begin{equation*}
\begin{split}
& e_1(\th)= DE^c(\th) - D_\th G^c(\th) \lambda,\\
& e_2(\th)=[(DF J^{-1})(K)DK\,N](\th)-
DK(\th+\omega)A(\th)-[J^{-1}DK\,N](\th+\omega)= O(E,DE)
\end{split}
\end{equation*}
by the choice of $A$, where $G^c(\th)=\Pi^c_{K(\th+\omega)}G(\th).$

The next step is to ensure the invertibility of the $2l \times
2l$-matrix $\tilde{M}^\top  J(K)\tilde{M}$. According to expression
\eqref{eq:formJM}, we can write  
\begin{equation*}
\tilde{M}(\th)^\top J(K(\th)) \tilde{M}(\th)=V(\th)+R(\th),
\end{equation*}
where
\begin{equation*}
V =\begin{pmatrix} 0 &
  \Id_l\\ -\Id_l  & N^\top DK^\top  J(K)^{-\top}DK\, N \end{pmatrix}
\end{equation*}
and 
\begin{equation*}
R =\begin{pmatrix} L &   0\\0 & 0\end{pmatrix}.
\end{equation*}
We have the following lemma, providing the desired invertibility
result  under 
a smallness assumption on $E$, namely \eqref{smallnessdeltaE} 
in the next lemma. 
Note that \eqref{smallnessdeltaE} has the same form as 
\eqref{smallness1}, but the constants could be 
slighly different since \eqref{smallness1} should also 
accomodate \eqref{smallnessvanishing}, which is  implied by 
conditions of the same form.

\begin{lemma}\label{repApp}
There exists a constant $C>0$ such that if 
\begin{equation} \label{smallnessdeltaE}
C \kappa \delta^{-(\nu+1)} \|E\|_{\rho} \leq 1/2 
\end{equation}
for some $0<\delta< \rho/2$ then the matrix $\tilde{M}^\top(\th) 
J(K(\th)) \tilde{M}(\th)$ is
invertible for $\th \in D_{\rho-2\delta}$ and there exists a matrix $\tilde{V}(\th)$ such that 
\begin{equation*}
(\tilde{M}(\th)^\top J(K(\th)) \tilde{M}(\th))^{-1}=V(\th)^{-1}+\tilde{V}(\th)
\end{equation*}
with 
\begin{equation*}
\tilde{V}(\th)=\Big(\sum_{k=1}^{\infty}(V(\th)^{-1}R(\th))^k\Big)V(\th)^{-1},
\end{equation*}
where the series is absolutely convergent. Furthermore, we have the estimate 
\begin{equation}\label{invEst}
\|\tilde{V}\|_{\rho-2\delta} \leq C' \kappa \delta^{-(\nu+1)} \|E\|_{\rho}, 
\end{equation} 
where the constant $C'>0$ depends on $l$, $\nu$, 
$|F|_{C^1(B_r)}$, $|J|_{C^1(B_r)}$,
$\|DK\|_{\rho}$,
$\|N\|_{\rho}$ and $\|\Pi^c_{K(\th)}\|_{\rho}$. 
\end{lemma}
\begin{proof}  The matrix
$V(\th)$ is invertible with 
\begin{equation*}
V^{-1}=\begin{pmatrix} N^\top DK^\top  J(K)^{-\top}DK\, N  & -\Id_l \\ 
\Id_l & 0
\end{pmatrix}. 
\end{equation*}

We can write
\begin{equation*}
\tilde{M}(\th)^\top J(K(\th))\tilde{M}(\th) =V(\th) (\Id_{2l}+V(\th)^{-1}R(\th)). 
\end{equation*}
To apply the Neumann series
(and consequently justify the existence of
the inverse of $\Id_{2l}+V^{-1}R$ as well as the estimates
for its size), we have to estimate the
term $V^{-1}R$. According to Lemma \ref{lagApp}, we have the
estimate for $L$
\begin{equation*}
\|L\|_{\rho-2\delta} \leq C \kappa \delta^{-(\nu+1)} (\|E\|_{\rho}
+ \|G\|_{\rho} |\lambda|)
\end{equation*}
for all $\delta \in (0,\rho/2)$. 
Using Lemma \ref{vanishing} 
this leads to the estimate 
\begin{equation*}
\|V^{-1}R\|_{\rho-2\delta} \leq C \kappa \delta^{-(\nu+1)} \|E\|_{\rho} 
\end{equation*}
for $0<\delta <\rho/2$, 
where $C>0$ depends on $l$,  $\nu$, 
$|F|_{C^1(B_r)}$, $|J|_{C^1(B_r)}$, $\|DK\|_{\rho}$,
$\|N\|_{\rho}$ and $\|\Pi_{K(\th)}^c\|_\rho$. 
Because of assumption \eqref{smallnessdeltaE}, we have
that  the right-hand side of the last equation is less than $1/2$.

Then the matrix $\Id_{2l}+V(\th)^{-1}R(\th)$ is invertible with 
\begin{equation*}
\|(\Id_{2l}+V^{-1}R)^{-1}\|_{\rho-2\delta} \leq \frac{1}{1-\|V^{-1}R\|_{\rho-2\delta}} \leq 2.
\end{equation*}
This ends the proof of Lemma \ref{repApp}.
\end{proof} 
\subsubsection{Identification of the center space}\label{sec:identification} 
In this section, we identify the center space as being very 
close (up to terms that can be bounded by the error) 
to the range of the matrix $\tilde{M}$ introduced in 
\eqref{definicioMtilde}, see Proposition~\ref{prop:distance}.
This will allow us to use the range of $\tilde{M}$ in 
place of $\E_{K(\th)}^c$ without changing the quadratic 
character of the method.

\begin{pro}\label{prop:distance}
Denote by $\Gamma_{K(\th)}$ the 
range of $\tilde{M}(\th)$ and by 
$\Pi^\Gamma_{K(\th)}$ the projection onto
$\Gamma_{K(\th)}$ according to the splitting 
$\E^s_{K(\th)} \oplus \Gamma_{K(\th)} \oplus \E^u_{K(\th)}$.

Then there exists a constant $C>0$ such that if  
$$
\delta^{-1} \| E\|_\rho \leq C
$$
we have the estimates (here dist$_\rho$ stands for the distance 
between subspaces
at the Grassmanian level)

\begin{equation} \label{eq:distancebound} 
\begin{split}
& \dist_{\rho-2\delta}( \Gamma_{K(\th)}, \E^c_{K(\th)}) \le C\delta^{-1} \| E \|_\rho  \\
& \| \Pi_{K(\th)}^c - \Pi_{K(\th)}^\Gamma \|_{\rho - 2 \delta}
\le  C\delta^{-1} \| E \|_\rho 
\end{split}
\end{equation}
for every $\delta \in (0,\rho/2)$ and where $C$, as usual, depends on the non-degeneracy constants of 
the problem. 
\end{pro} 

\begin{proof} 
Of course, the two inequalities in \eqref{eq:distancebound} are
equivalent. 

{From} \eqref{e} and Cauchy estimates, we 
have:
\[
\dist_{\rho-\delta} ( (DF\circ K) \Gamma_{K(\th)}, \Gamma_{K(\th)} \circ T_\omega)  
\le C \delta^{-1} \| E \|_\rho .
\]

Using again equation \eqref{e} and iterating it, we obtain for $n \geq 1$
\begin{equation*}
DF(K(\th + n \omega )) \times \cdots \times 
DF(K(\th))\tilde{M}(\th) = 
\tilde{M}(\th + n \omega) \S(\th + (n-1) \omega) \times \cdots \times \S(\th) 
+ R_n,
\end{equation*} 
where 
$$
\| R_n \|_{\rho - \delta} \le C_n \delta^{-1} \| E\|_\rho
$$
and $C_n$ depends on $n$.

Since $\S(\th)$ is upper triangular with Id$_l$ on the diagonal, 
we have:
\[
\S(\th + (n-1) \omega)\times  \cdots \times \S(\th)  = 
\begin{pmatrix} 
& \Id_l  & A(\th + (n-1) \omega) + \cdots +A(\th) \\
& 0 & \Id_l 
\end{pmatrix} .
\]

Therefore, by induction, we have for every $n \in  \nat$
\[
\| DF(K(\th + n \omega )) \cdots 
DF(K(\th))\tilde{M}(\th)  \|_{\rho - \delta} 
\le 
C n + C_n \delta^{-1} \| E\|_{\rho}.  
\]
Identical calculations give that 
\[
\| DF^{-1}(K(\th - n \omega )) \cdots 
DF^{-1} (K(\th))\tilde{M}(\th)  \|_{\rho - \delta} 
\le 
C n + C_n \delta^{-1} \| E\|_{\rho}.
\]

Note that, given any $\mu_3 > 1 $ (as in Definition \ref{ND}),  there exists an integer $n_{\mu_3} \geq 0$ such that for all $n \geq n_{\mu_3}$, 
we have $C n <   \mu_3 ^n $. Consequently, choosing such $n_{\mu_3}$ there exists a constant $C$ such that if the error satisfies
$$\delta^{-1} \| E\|_\rho \leq C,$$
 we have 
$C n + C_n \delta^{-1} \| E\|_\rho <   \mu_3^n $. 
In other words, the above estimates hold for all sufficiently large 
$n$, provided that we impose a suitable smallness condition on
$\delta^{-1} \| E\|_\rho$.

As a consequence, $\Gamma_{K(\th)}$ is an approximately
invariant bundle, and we also have bounds on the rate of 
growth of the co-cycle both in positive and negative times. 
Using standard tools in the theory of hyperbolic systems (see Proposition~\ref{PropDEG} below where we prove 
the result for all the bundles), this shows that indeed 
one can find a true invariant subspace $\tilde \E_{K(\th)}$ close to $\Gamma_{K(\th)}$. Since 
this invariant subspace should be of the same dimension of 
the center space $\E^c_{K(\th)}$, we deduce that 
$$\tilde \E_{K(\th)}=\E^c_{K(\th)}.$$ 
See also 
Remark~\ref{rem:onespace} below. 
\end{proof}

\subsubsection{Final estimates of the solution on the center subspace}

We can now finish the solution of  equation \eqref{lin-eq}
on the center
subspace. We recall the
linearized equation around $(\lambda,K)$ projected on the center subspace:
\begin{equation}\label{linTr}
G^c(\th) \Lambda+DF (K(\th))
\Delta^c(\th)
-\Delta^c(\th+\omega)=-E^c(\th).
\end{equation}
We make the change the unknowns in \eqref{linTr}
\begin{equation}\label{changeVar}
\Delta^c(\th)=\tilde{M}(\th)W(\th) + \hat e(\th)  W(\th), 
\end{equation}
where 
\begin{equation}\label{hate} 
\hat e = \Pi_{K(\th+\omega)}^c - \Pi_{K(\th+\omega)}^\Gamma
\end{equation}  which 
was estimated in Proposition~\ref{prop:distance}.

Substituting \eqref{changeVar} into equation \eqref{linTr} we get 
\begin{align}\label{change1-1}
DF(K(\th)) & \tilde{M}(\th)W(\th)  -
\tilde{M}(\th+\omega)W(\th+\omega) \\
& =-E^c(\th)-G^c(\th) \Lambda 
+\hat e(\th+\om) W(\th+\om) - DF(K(\th)) \hat e(\th) W(\th) . 
\end{align}

We anticipate that the term $\hat e W$ will be quadratic in the error. Similarly, writing 
$$
G^c=\Pi^\Gamma_{K(\th+\omega)}G+\hat e G,
$$ 
we also anticipate that the term $\hat e G \Lambda$ will be quadratic 
in the error. Since the function $G$ will be chosen to be $J(K_0)^{-1}DK_0\circ T_\om$,
namely in $\Gamma_{K_0(\th+\om)}$, 
we drop the index from $G^c$, writing $G$ directly. 
As a consequence, we 
will ignore these two terms and consider instead the equation 
\begin{equation}\label{change2}
DF(K(\th)) \tilde{M}(\th)W(\th)-
\tilde{M}(\th+\omega)W(\th+\omega) =-E^c(\th)-G(\th) \Lambda 
\end{equation}
which differs from the linearized equation in the term 
$(\hat e W)\circ T_\om-DF(K) (\hat e W)-\hat e G\Lambda$. 
Note that, ignoring this term we obtain an 
equation where all the terms are in the range of $\tilde{M}$.

We multiply equation \eqref{change2} by 
$\tilde{M}(\th+\omega)^\top J(K(\th+\om))$ 
\begin{align*}
[\tilde{M}^\top J(K)](\th+\omega) & DF (K(\th))\tilde{M}(\th)W(\th)-[\tilde{M}^\top J(K)](\th+\omega) \tilde{M}(\th+\omega)W(\th+\omega)\\   
=&-[\tilde{M}^\top J(K)](\th+\omega) 
[E^c(\th)+ G(\th) \Lambda].
\end{align*}

Using Lemma \ref{repApp} (invertibility of $\tilde{M}^\top J(K)\tilde{M}$) and
equation \eqref{e}, we can write
\begin{eqnarray}\label{eqSD1}
\left[ \begin{pmatrix} \Id_l & A(\th)\\ 0_l & \Id_l
\end{pmatrix}+B(\th)\right]W(\th)-W(\th+\omega)=p_1(\th)+p_2(\th)\\
-[\tilde{M}^\top J(K)
\tilde{M}](\th+\omega)^{-1}[\tilde{M}^\top J(K)](\th+\omega)G(\th) \Lambda,\nonumber 
\end{eqnarray}
where 
\begin{equation}\label{b}
B(\th)=[\tilde{M}^\top J(K)
\tilde{M}](\th+\omega)^{-1}[\tilde{M}^\top J(K)](\th+\omega)
e(\th),
\end{equation}

\begin{equation}\label{p1}
p_1(\th)=-V(\th+\omega)^{-1}[\tilde{M}^\top J(K)](\th+\omega)
E^c(\th) 
\end{equation}
and
\begin{equation}\label{p2}
p_2(\th)=-\tilde{V}(\th+\omega)[\tilde{M}^\top J(K)](\th+\omega) E^c(\th).
\end{equation}

In the following lemma, we sum up the previous computations and estimate the terms in equation \eqref{eqSD1}.
\begin{lemma}\label{repres}
Assume $\omega \in D(\kappa,\nu)$ and $\delta$ and $\|E\|_{\rho}$ satisfy \eqref{smallnessdeltaE}. 
Equation \eqref{change2} can be written in the form
\begin{eqnarray}\label{eqSD}
\left[ \begin{pmatrix} \Id_l & A(\th)\\ 0_l & \Id_l
\end{pmatrix}+B(\th)  \right] W(\th)
- W(\th+\omega)
=p_1(\th)+p_2(\th)\\
-[\tilde{M}^\top J(K)
\tilde{M}]^{-1}(\th+\omega)[\tilde{M}^\top J(K)](\th+\omega) G(\th) \Lambda,\nonumber 
\end{eqnarray}
where the matrix $B$ and the vectors  $p_1$ and $p_2$ are given by expressions \eqref{b}, \eqref{p1} and \eqref{p2} respectively. 

The following estimates hold: 
\begin{equation}\label{estimp1}
\|p_1\|_{\rho} \leq C \|E\|_{\rho},
\end{equation}
where $C$ only depends on $|J|_{C^1(B_r)}$, $\|N\|_{\rho}$, $\|DK\|_{\rho}$ and
$\|\Pi^c_{K(\th)}\|_{\rho}$. For $p_2$ and $B$ we have 
\begin{equation}\label{estimp2}
\|p_2\|_{\rho-2\delta} \leq C \kappa \delta^{-(\nu+1)} 
\|E\|^2_{\rho}  
 \end{equation}
 and 
\begin{equation}\label{estimb}
\|B\|_{\rho-2\delta} \leq C   \delta^{- 1 }( \|E\|_{\rho}+ |\lambda|),
 \end{equation}
where $C$ depends $l$, $\nu$, $\|N\|_{\rho}$, 
$\|DK\|_{\rho}$, $|F|_{C^1(B_r)}$, $|J|_{C^1(B_r)}$ and
$\|\Pi^c_{K(\th)}\|_{\rho}$. 
\end{lemma}
\begin{proof} 
Since the matrix $V^{-1}$ does not depend on $L$ the estimate (\ref{estimp1}) is 
obvious from the formula \eqref{p1} for $p_1(\th)$. 

According to the proof of Lemma \ref{repApp} 
the estimate (\ref{estimp2}) then comes from estimate \eqref{invEst}. We turn to the
estimate on $B$. We have 
\begin{equation*}
B(\th)=(V(\th+\omega)^{-1}+\tilde{V}(\th+\omega))\tilde{M}(\th+\omega)^\top 
J(K(\th+\om)) e(\th).
\end{equation*}
This leads to 
\begin{align*}
\|B\|_{\rho-2\delta} \leq & \|V(\th+\omega)^{-1}\|_{\rho-2\delta} \|\tilde{M}(\th+\omega)^\top J(K(\th+\om)) e(\th)\|_{\rho-2\delta}\\
& +\|\tilde{V}(\th+\omega)\tilde{M}(\th+\omega)^\top 
J(K(\th+\om)) e(\th)\|_{\rho-2\delta}.
\end{align*}
Therefore, using estimate
\eqref{invEst} and Cauchy estimates, we end up with 
\begin{equation*}
\|B\|_{\rho-2\delta} \leq C(  \delta^{-1} \|E\|_{\rho}
+ \delta^{-1}|\lambda| +\kappa \delta^{-(\nu+1)}\|E\|_{\rho}(\delta^{-1} \|E\|_{\rho} + \delta^{-1} |\lambda|)).
\end{equation*}
This leads to the desired result thanks to 
the smallness assumption on $\|E\|_{\rho}$.
\end{proof} 
\subsubsection{Approximate solvability of the equations on the
center subspace}
\label{sec:temp}

This section is devoted to solving approximately (up to quadratic
error) the linearized equation \eqref{eqSD},
as is usual in KAM theory.

To this end, we introduce the following operator 
\begin{equation*}
\mathcal{L}W(\th)=\begin{pmatrix} \Id_l & A(\th)\\ 0_l & \Id_l \end{pmatrix}W(\th)-W(\th+\omega).
\end{equation*}
Equation \eqref{eqSD} can be written as
\begin{align}\label{sd4}
\mathcal{L}W(\th)&+B(\th)W(\th)=p_1(\th)+p_2(\th)\\&-[\tilde{M}^\top J(K)
\tilde{M}](\th+\omega)^{-1} [\tilde{M}^\top J(K)](\th+\omega) G(\th)\Lambda. \nonumber
\end{align}
According to estimates in Lemma \ref{repres}, we have
$p_2=O(\|E\|_{\rho}^2)$, $p_1=O(\|E\|_{\rho})$ and
$B=O(\|E\|_{\rho}+|\lambda|)$.
 Solving approximately equation \eqref{sd4} with an error
``quadratic'' in $E$  does not affect the convergence of the Newton scheme.
See \cite{Zehnder75} for an abstract discussion and 
\cite{Zehnder76} for several concrete applications.

Equation \eqref{sd4} does not fit into the framework of Proposition
\ref{sdrussman} since the average of the right-hand side is
generically non-zero. However, by using the increment parameter
$\Lambda$, we can make this average equal to zero. Furthermore, equation \eqref{sd4} has two unknowns (the two symplectic coordinates). Thanks to Lemma \ref{repApp}, one can write the term 
\begin{equation*}
[\tilde{M}^\top J(K)
\tilde{M}]^{-1}(\th+\omega) [\tilde{M}^\top J(K)](\th+\omega)
 G(\th)\Lambda=q_1(\th)\Lambda+q_2(\th)\Lambda,
\end{equation*}
where the matrix $q_1$ (which is $2l \times l$) is 
\begin{equation*}
q_1(\th)=V(\th+\omega)^{-1}\tilde{M}(\th+\omega)^\top J(K(\th+\omega)) G(\th)
\end{equation*}
and $q_2$ satisfies for all $\delta \in (0,\rho/2)$
\begin{equation*}
\|q_2\|_{\rho-2\delta} \leq C \kappa \delta^{-(\nu+1)} \|G\|_{\rho}  \, \|E\|_{\rho},
\end{equation*}
where the constant $C$ depends on $l$, $\nu$, $\|N\|_{\rho}$, 
$\|DK\|_{\rho}$, $|F |_{C^1(B_r)}$, $|J|_{C^1(B_r)}$ and
$\|\Pi^c_{K(\th)}\|_{\rho}$.

We define an approximate solution of \eqref{sd4} as a solution of the 
following  
equation \eqref{sd2Approx}, obtained by 
removing the terms containing $B$ and $q_2$ from the 
equation \eqref{eqSD}, which was equivalent to \eqref{change2}. 
We recall that \eqref{change2} was obtained from 
the Newton step by removing the terms that contained $\hat e$.
As we will see, all these eliminations do not change the 
quadratic convergence of the method. Consider now
\begin{equation}\label{sd2Approx}
\mathcal{L}v(\th)=p_1(\th)-q_1(\th)\Lambda.
\end{equation}

Thanks to the non-degeneracy conditions (see Definition \ref{ND}), we obtain the following result. 

\begin{pro}\label{approximateSol}
Assume $\omega \in D(\kappa,\nu)$ and $(\lambda,K)$ is a
non-degenerate pair (i.e. $(\lambda,K) \in ND(\rho)$). If the error
$\|E\|_{\rho}$ satisfies \eqref{smallnessdeltaE}
 and the smallness assumptions in proposition \ref{prop:distance}, 
there exist a mapping $v$, analytic on $D_{\rho-2\delta}$ 
and a vector $\Lambda \in \mathbb{R}^{l}$ solving equation \eqref{sd2Approx}. 

Moreover there exists a constant $C>0$ depending on $\nu$, $l$,
$\|K\|_{\rho}$, $|(\avg(Q))^{-1}|$, $|(\avg(A))^{-1}|$,
$\|N\|_{\rho}$ and $\|\Pi^c_{K(\th)}\|_{\rho}$ such that
\begin{equation*}
\|v\|_{\rho-2\delta} <C \kappa^2 \delta^{-2\nu} \|E\|_{\rho} 
\end{equation*}
and
\begin{equation*}
|\Lambda| <C \|E\|_{\rho}. 
\end{equation*}
\end{pro}
\begin{proof}  We denote $R(\th)$ the right-hand side of equation
\eqref{sd2Approx}, i.e. we solve 
\begin{equation}\label{eqTemp}
\mathcal{L}v(\th)=R(\th), 
\end{equation}
with 
\begin{equation*}
R=p_1-q_1\Lambda. 
\end{equation*}
We now decompose equation \eqref{eqTemp} into symplectically conjugate coordinates,
i.e. $v=(v_1,v_2)^\top $,
$R(\th)=(R_1(\th),R_2(\th))^\top $. Therefore, equation \eqref{eqTemp} is equivalent to 
\begin{eqnarray*}
v_1(\th)+A(\th)v_2(\th)=v_1(\th+\omega)+R_1(\th),\\
v_2(\th)=v_2(\th+\omega)+R_2(\th). 
\end{eqnarray*}
A simple computation shows that 
\begin{equation*}
R_2(\th) = - [DK^\top J(K)] \circ T_\om G(\th)
\end{equation*}
We choose $\Lambda \in \mathbb{R}^l$ such that 
\begin{equation*}
\avg (R_2)=0. 
\end{equation*}
According to Proposition \ref{sdrussman}, 
if $\avg (R_2) =0$ 
the equation in $v_2$ admits
an analytic solution with arbitrary average on $D_{\rho-\delta}$
and we have the estimate 
\begin{equation}\label{estimv2}
\|v_2\|_{\rho-\delta}  \leq C \kappa \delta^{-\nu} \|R_2\|_{\rho}+|\avg (v_2)|.
\end{equation}   
Then we choose $\avg(v_2)$ such that $\avg (R_1-Av_2)=0$, which allows us 
to
solve uniquely the equation in $v_1$, the function $v_1$ being of zero
average. Furthermore, we have the estimate 
 \begin{equation*}
\|v_1\|_{\rho-2\delta}  \leq C \kappa \delta^{-\nu} \|R_1-A v_2\|_{\rho-\delta}.
\end{equation*} 

We now turn to the estimates. First we estimate $\Lambda$. The vector $\Lambda \in \mathbb{R}^l$ is such that 
\begin{equation*}
\avg \Big(DK^\top(\om+\th) J(K(\om+\th))
(E^c(\th)+G(\th) \Lambda )\Big)=0.
\end{equation*}
This leads to 
\begin{align*}
\avg\Big((DK^\top(\om+\th) & J(K(\om+\th)) G(\th) \Big)\Lambda\\ = &
    -\avg \Big((DK^\top(\om+\th) J(K(\om+\th)) E^c(\th)\Big).
\end{align*}
Note that by the definition of $P$ and the fact that $N$ is symmetric, the matrix which applies to $\Lambda$ is the average of $Q$ 
which, by hypothesis, is invertible. This leads to the desired estimate for $\Lambda$.  

We now estimate the solution $v$. {From} the expression of $R$ and the value of $\Lambda$ obtained above, we
have that there exists a constant $C$ such that
\begin{equation*}
\|R_i\|_{\rho} \leq C \|E\|_{\rho},
\end{equation*}
 for $i=1,2$. Furthermore, we choose $\avg (v_2)$ such that
 $\avg (R_1-Av_2)=0$, i.e. 
\begin{equation*}
\avg (v_2)=\avg(A)^{-1}(\avg (R_1)-\avg (Av^{\perp}_2)), 
\end{equation*}
where $v_2=v_2^{\perp}+\avg (v_2)$. 
This is possible since by the twist condition $\avg(A) $ is invertible.
Thanks to estimate \eqref{estimv2},
this leads to the desired result. 
\end{proof}  

We now come back to the solutions of \eqref{eqCenter}. The above
procedure allows us to prove the following proposition, providing an approximate solution of the projection of $D_{\lambda,K}\mathcal{F}_{\omega}(\lambda,K)(\Lambda,\Delta)=-E$ on the center subspace.  

\begin{pro}\label{solCenter}
Let $(\Lambda,W)$ be as in Proposition \ref{approximateSol}
and assume the hypotheses of that proposition hold. 
Define 
$\Delta^c(\th)=\tilde{M}(\th)W(\th) + \hat e(\th) W(\th)$
and obtain $W$ and $\lambda$ as indicated above. 

Then, $(\lambda, \Delta^c)$ 
is an approximate solution of 
 \eqref{eqCenter} and we have the following estimates 
\begin{equation*}
\|\Delta^c\|_{\rho-2\delta} \leq C \kappa^2 \delta^{-2\nu} \|E\|_{\rho},
\end{equation*} 
\begin{equation*}
|\Lambda| \leq C \|E\|_{\rho},
\end{equation*}
where the constant $C$ depends on $\nu$, $l$, $|(\avg(Q))^{-1}|$, 
$|(\avg(A))^{-1}|$,
$\|N\|_{\rho}$, $\|G\|_{\rho}$ and $\|\Pi^c_{K(\th)}\|_{\rho}$. Moreover
\begin{equation}\label{estimApprox}
\begin{split} 
\|D_{\lambda,K}\mathcal{F}_{\omega}(\lambda,K)(\Lambda,\Delta^c)+E^c\|_{\rho-2\delta}
&\leq C \kappa^3 \delta^{-(3\nu+1)} (\|E\|^2_{\rho} + \|E\|_{\rho}
|\lambda|) + C \delta^{-1 + \nu} \| E \|_\rho^2 \\
&\leq C \kappa^3 \delta^{-(3\nu+1)} \|E\|^2_{\rho},
\end{split}
\end{equation}
where the constant $C$ depends on $l$, $\kappa$, $\nu$,
$|F|_{C^1(B_r)}$, $\|DK\|_{\rho}$, $\|N\|_{\rho}$, $|(\avg (A))^{-1}|$, 
$|(\avg (Q))^{-1}|$ and $\|G\|_{\rho}$. 
\end{pro}
\begin{proof}

  The first estimate comes from the previous
Proposition \ref{approximateSol}.

For the second one \eqref{estimApprox}, we recall that we have: 
\begin{equation}\label{identity} 
\begin{split}
D_{\lambda,K}& \mathcal{F}_{\omega}(\Lambda,\Delta^c)(\th)+ E^c(\th)\\
 = & -[\tilde{M}^\top J(K)\tilde{M}](\th+\omega)[B(\th)v(\th)-p_2(\th)] -q_2(\th)\Lambda \\
& +\hat e(\th+\om) W(\th+\om) - DF(K(\th)) \hat e(\th) W(\th). 
\end{split}
\end{equation}

The first term in the right-hand side of 
\eqref{identity}
is estimated in 
Proposition \ref{approximateSol}, see estimates
\eqref{estimp2}-\eqref{estimb}. The second one
comes from the vanishing Lemma \ref{vanishing}. 
The third term is estimated in Proposition \ref{prop:distance}.
\end{proof} 

\subsection{Solving the linearized equations on the hyperbolic subspaces}\label{sec:hyper}

According to the splitting \eqref{splitting}, there exist projections
on the linear spaces $\mathcal{E}^s_{{K(\th)}}$ and $\mathcal{E}^u_{{K(\th)}}$. The
analytic regularity of the splitting implies that the dependence of
these projections in $\th$ is analytic in the same domain 
as the spaces. We denote
$\Pi^s_{K(\th+\omega)}$ (resp. $\Pi^u_{K(\th+\omega)}$) the
projections (of base $K(\th+\omega)$) on the stable (resp. unstable)
invariant subspace. 

We project equation \eqref{linear} on the stable and unstable spaces to obtain 
\begin{equation}\label{eqStable}
\Pi^s_{K(\th+\omega)}\Big(G(\th)\Lambda +DF_{}(K (\th))\Delta(\th)-\Delta(\th+\omega)\Big)=
-\Pi^s_{K(\th+\omega)}E(\th),
\end{equation}  
\begin{equation}\label{eqUnstable}
\Pi^u_{K(\th+\omega)}\Big(G(\th) \Lambda +DF_{}(K (\th))\Delta(\th)-\Delta(\th+\omega)\Big)=
-\Pi^u_{K(\th+\omega)}E(\th).
\end{equation}
Furthermore,
thanks to the invariance of the splitting, we can write 
\begin{equation*}
\Pi^s_{K(\th+\omega)}DF_{}(K (\th))\Delta(\th)=
DF_{ }(K(\th))\Pi^s_{K(\th)}\Delta(\th)
\end{equation*}
for the stable part and 
\begin{equation*}
\Pi^u_{K(\th+\omega)}DF_{}(K(\th))\Delta(\th)=DF_{ }(K (\th))\Pi^u_{K(\th)}\Delta(\th)
\end{equation*}
for the unstable one. Introducing the change of variables $\th'=T_{\omega}(\th)$ and the notation $\Delta^{s,u}(\th')=\Pi^{s,u}_{K(\th')}\Delta(\th')$, equations \eqref{eqStable}-\eqref{eqUnstable} can be written in the following form
\begin{equation}\label{eqStable2}
DF_{}(K)\circ T_{-\omega}(\th')\Delta^s(T_{-\omega}(\th'))-\Delta^s(\th')
=-\tilde{E}^s(\th',\Lambda) ,
\end{equation} 
where 
\begin{equation*}
\tilde{E}^s(\th',\Lambda)=\Pi^s_{K(\th')}\Big(G(T_{-\omega}(\th')) \Lambda\Big)+\Pi^s_{K(\th')}E \circ
T_{-\omega}(\th') 
\end{equation*}
and  
\begin{equation}\label{eqUnstable2}
DF_{}(K)\circ T_{-\omega}(\th')\Delta^u(T_{-\omega}(\th'))-\Delta^u(\th')
=-\tilde{E}^u(\th',\Lambda),
\end{equation}
where 
\begin{equation*}
\tilde{E}^u(\th',\Lambda)=\Pi^u_{K(\th')}\Big(G(T_{-\omega}(\th'))\Lambda\Big)+\Pi^u_{K(\th')}E\circ
T_{-\omega}(\th').
\end{equation*}

The following proposition provides an existence result together with
estimates for equations \eqref{eqStable2}-\eqref{eqUnstable2}. 

\begin{pro}\label{hyperb}
Fix $\rho>0$. Then  equation \eqref{eqStable2}
(resp. \eqref{eqUnstable2}) admits a unique analytic solution
$\Delta^s:D_{\rho} \rightarrow \mathcal{E}^s_{K(\th)}$
(resp. $\Delta^u:D_{\rho} \rightarrow
\mathcal{E}^u_{K(\th)}$). Furthermore there exists a constant $C$
such that 
\begin{equation}\label{estimHyperb}
\|\Delta^{s,u}\|_{\rho} \leq C (\|E\|_{\rho} +|\Lambda|),
\end{equation} 
where the constant $C$ depends on the hyperbolicity constant $\mu_1$
(resp. $\mu_2$), the norm of the projection
$\|\Pi^s_{K(\th)}\|_{\rho}$ (resp. $\|\Pi^u_{K(\th)}\|_{\rho}$)
 $\|G(\th)\|_{\rho}$ and the constant $C_h$ involved in (\ref{ndeg1})
 (resp.  (\ref{ndeg2})). 
\end{pro}
\begin{proof}  We give the proof for the stable case, the unstable one being similar and left to the reader. Using equation \eqref{eqStable2} iteratively, we claim that its solution is given by 
\begin{equation}\label{hyptemp}
\Delta^s(\th')=\displaystyle{\sum_{k=0}^{\infty}}
\Big(DF_{}(K)\circ T_{-\omega}(\th')\times \dots  \times DF_{}(K)\circ T_{-\omega}(\th')\Big)
\tilde{E}^s(T_{-(k+1)\omega}(\th'),\Lambda).
\end{equation}
Using the condition on the co-cycles over $T_{-\omega}$ (see equation \eqref{ndeg1}), the series converges uniformly on $D_{\rho}$ and one can estimate
\begin{equation} \label{hyperbolicformula}
\|\Delta^s\|_{\rho} \leq C_h \|\tilde{E}^s\|_{\rho}
\displaystyle{\sum_{k=0}^{\infty}} \mu_1^k \leq C (\|E\|_{\rho}+|\Lambda|)
\end{equation}
since $\mu_1 <1$. 
Once we know that the series
converges uniformly, we can rearrange the terms and get that \eqref{hyptemp} is indeed a solution.
The proof in the case of the
unstable space follows in the same way, multiplying equation
\eqref{eqUnstable2} by $(DF_{}(K)\circ T_{-\omega})^{-1}$ and
using the Condition \eqref{ndeg2} on the co-cycles. 
\end{proof} 

\section{Iteration of the Newton step and convergence}
\label{sec:iteration}

In the following we describe precisely the iteration 
of the Newton method. As it is standard in 
KAM theory, we show that if the initial error 
$\|E_0\|_{\rho_0}$ is small enough,
one can choose the domain loss, 
so that the  iterative scheme converges 
to a solution of \eqref{translated} which moreover is close to the initial
approximate solution. 
As a consequence of the vanishing lemma 
(i.e. Lemma~\ref{vanishing})  one gets $\lambda=0$ and then a solution of 
$$F \circ K=K \circ T_\omega. $$

In the rest of this section, we are under the assumptions of Theorem \ref{existence}.

\subsection{Estimates for one  step of the Newton method}

Recall that we have implemented a step showing that, given an approximate 
solution, $(\lambda_{m-1}, K_{m-1})$ of \eqref{translated}, 
 which is  non-degenerate in the sense 
of Definition~\ref{ND}  and satisfies 
the conditions \eqref{smallnessvanishing} of
Lemma~\ref{vanishing} and \eqref{smallness1} of 
Lemma~\ref{main} , then we  find an approximate
solution $(\Lambda_{m-1}, \Delta_{m-1})$
of the Newton equation. That is, we can find 
\begin{equation*}
D_{\lambda,K}\mathcal{F}_{\omega}(\lambda_{m-1},K_{m-1})(\Lambda_{m-1},\Delta_{m-1})=-E_{m-1} + R_{m}
\end{equation*}
with $E_{m-1}(\th)=\mathcal{F}_{\omega}(\lambda_{m-1},K_{m-1})(\th)$
and $R_m$ ``quadratically'' small. 
If $E$ is defined in $D_{\rho_{m-1}}$, the Newton correction $\Delta_{m-1}$
is defined in a smaller domain $D_{\rho_m}$, $\rho_m = \rho_{m-1} - \delta_m$.
The precise results on the step are collected in Lemma~\ref{main} and the 
description of the step is given along the proof.

The next result  Proposition~\ref{improvement}, makes precise 
the observation that, if we can define 
$F\circ K_{m}$ 
then it is possible to show 
that  the new  remainder is quadratic. 
Furthermore, we will show that the change in the 
non-degeneracy assumptions can be estimated by the size of 
the error.  

The assumption that 
$F \circ K_{m}$ can be defined, requires only
that the range of 
$K_m = K_{m-1} + \Delta_{m-1}$ does not get 
out the domain of $F$. This will be implied by smallness 
assumptions on $\Delta$ that, using the conclusions of 
Proposition~\ref{main}, are implied 
by assumption \eqref{compositiondefined}. 
As it will turn out, the assumption
\eqref{compositiondefined} is stronger than 
\eqref{smallness1} so that \eqref{compositiondefined} is 
enough to  ensure that we can carry out a Newton step as indicated.

In subsequent sections, we will show 
that  if we choose the sequence of domain losses 
$\delta_m = \frac{1}{4} \delta_0 2^{-m}$,
and  the error is small enough, the process can 
be iterated infinitely often and converges to 
a solution of the equation. The argument also shows that the 
hyperbolic splitting  converges.

\begin{pro}\label{improvement}
Choose an initial approximation $\lambda_0=0, K_0$, where
$K_0 \in ND(\rho_0)$. Assume that $K_0( D_{\rho_0}) $, 
the range of $K_0$ is at a distance $r > 0$ from complement of 
the domain of 
definition of $F$. 

Assume $(\lambda_{m-1},K_{m-1}) \in ND(\rho_{m-1})$ is an approximate solution of equation \eqref{translated}  
and that the following holds 
\begin{equation}\label{closetoK0}
\|K_{m-1}-K_0\|_{\rho_{m-1}} < r/2,
\end{equation} 
where $r$ is chosen 
sufficiently small so that
we can apply Lemma~\ref{vanishing},  the constants
in $ND(\rho_{m-1})$ are chosen uniformly
and that  the range of $K_{m-1}$ is  inside
the domain of definition of $F$. 
Assume furthermore that \eqref{smallness1} holds so that we can 
apply Lemma~\ref{main}.

Denote by 
 $C$ expressions that depend only on $\nu$, $l$,
$|F|_{C^1(B_r)}$, $\|DK_{m-1}\|_{\rho_{m-1}}$,
$\|\Pi^{s,c,u}_{K_{m-1}(\th)}\|_{\rho_{m-1}}$,
$|(\avg(Q_{m-1}))^{-1}|$ and $|(\avg(A_{m-1}))^{-1}|$
and, hence, can be chosen uniformly if $K_{m-1}$ is in a sufficiently
small neighborhood of $K_0$ as indicated in 
\eqref{closetoK0}.

Let $\Lambda_{m-1}, \Delta_{m-1}$ be the corrections
produced in Lemma~\ref{main}. 

If $E_{m-1}$ is small enough such that 
\begin{equation}\label{compositiondefined}
C \kappa\delta_{m-1}^{-2 \nu-1} \|E_{m-1}\|_{\rho_{m-1}} < r/2
\end{equation}
then, the set $(K_{m-1} + \Delta_{m-1})( D_{\rho_{m-1} - \delta_{m-1}})$ is 
well inside the domain of definition of $F$ and 
$E_m(\th)=\mathcal{F}_{\omega}(\lambda_m,K_m)(\th)$ satisfies 
(defining $\rho_m=\rho_{m-1}-3\delta_{m-1}$) 
\begin{equation}\label{estimate}
\|E_m\|_{\rho_m} \leq C\kappa^4 \delta_{m-1}^{-4\nu} \|E_{m-1}\|^2_{\rho_{m-1}}. 
\end{equation}
\end{pro}  

\begin{proof} 

We have $\Delta_{m-1}(\th)=
\Pi^h_{K_{m-1}(\th)} \Delta_{m-1}(\th) +\Pi^c_{K_{m-1}(\th)} \Delta_{m-1}(\th)$, where $\Pi^h_{K_{m-1}(\th)}$ is the projection on the
hyperbolic subspace.
Proposition \ref{solCenter} and Proposition \ref{hyperb}
respectively, particularized to 
$\delta_{m-1}$  give us 
that 
\[
|| \Delta_{m-1}||_{\rho_m} \le C \kappa^2 \delta_{m-1}^{- 2 \nu}
|| E_{m-1}||_{\rho_{m-1}}
\]
and using Cauchy inequalities. 
\[
|| D \Delta_{m-1}||_{\rho_m} \le C \kappa^2 \delta_{m-1}^{- 2 \nu - 1}
|| E_{m-1}||_{\rho_{m-1}}
\]

Using \eqref{compositiondefined}, and the previous estimates on 
$\Delta_{m-1}$, we see that the range of 
$K_m \equiv K_{m-1} + \Delta_{m-1}$ is well inside the domain of 
definition of $F$ so that we 
can define $F\circ K_m$. 

Define the remainder of the Taylor expansion
\begin{align*}
\mathcal{R}(\lambda,\lambda',K,K') =& \mathcal{F}_{\omega}(\lambda,K)-\mathcal{F}_{\omega}(\lambda',K')\\
&-D_{\lambda,K}\mathcal{F}_{\omega}(\lambda,K)(\lambda-\lambda',K-K').
\end{align*}

Then we have
\begin{align*}
E_m(\th)=E_{m-1}(\th)&+D_{\lambda,K}\mathcal{F}_{\omega}(\lambda_{m-1},K_{m-1}(\th))(\Lambda_{m-1},\Delta_{m-1}(\th))\\&+ \mathcal{R}(\lambda_{m-1},\lambda_m,K_{m-1},K_m)(\th).
\end{align*}
Using  estimate \eqref{estimApprox},
for the error in solving the center equation 
and recalling that 
 the equations on the hyperbolic subspace are 
{\sl exactly} solved, we have 
\begin{eqnarray*}
\|E_{m-1}+D_{\lambda,K}\mathcal{F}_{\omega}(\lambda_{m-1},K_{m-1})(\Lambda_{m-1},\Delta_{m-1})\|_{\rho_m} \\
\leq c_{m-1} \kappa^3 \delta_{m-1}^{-(3\nu+1)}\|E_{m-1}\|^2_{\rho_{m-1}}. 
\end{eqnarray*}

Estimate \eqref{estimate} then follows from  Taylor's remainder
bound
\[
\begin{split} 
 | F\circ(K_{m-1} &+ \Delta_{m-1})(\th) -  
F \circ K_{m-1}(\th) - DF \circ K_{m-1}(\th) \Delta_{m-1}(\th) | \\
& \le
C  \| D^2 F  \|_{\B} | \Delta_{m-1}(\th)|^2.
\end{split} 
\]
Note that, since $\delta_n$ go to zero, we can assume 
that the estimates from the Taylor remainder are 
larger than those from the error of the solution.
\end{proof}   

\subsection{Change of the hyperbolicity and 
the non-degeneracy conditions in the iterative step}


The main goal of this section is to estimate 
the change of the non-degeneracy conditions in terms of 
the size of the  error at the beginning of the iterative 
step. 

We begin by estimating the change in the invariant splitting. 
Later, we will estimate the change in the twist conditions. 

The first result Proposition~\ref{PropDEG} is 
a standard result in the theory of normally hyperbolic 
sets that  allows us to conclude that if we are given an 
approximately invariant splitting, which has
some hyperbolicity, then there is a truly invariant splitting
nearby. The proof is a reformulation in 
an {\em a posteriori} format of standard arguments on 
the stability of hyperbolic splittings 
\cite{SackerS74, HirschP68, Fenichel71,PlissS99,Pesin04}. 
Since this will be part of an iterative procedure, 
we also need to obtain rather detailed estimates.

As a corollary, we will obtain that, when we change the embeddings
$K$ in the iterative step, the change of the invariant subspaces
will be controlled by the change in the embedding. Of course, 
since the twist conditions are just properties of the 
restriction of the derivative to an appropriate 
subspace, we will obtain that the size of 
the change in the twist conditions
is controlled by the size of the change of 
the embedding.

Notice also that Proposition~\ref{PropDEG} provides a  way to verify the hyperbolicity out of 
a finite calculation and in particular, out of the results of 
a numerical calculation. 
We have also used Proposition~\ref{PropDEG} to identify the 
center space in Section~\ref{sec:identification}.

\begin{pro}\label{PropDEG} 
Assume that there is an analytic splitting 
\begin{equation} \label{splittingassumed}
T_{K(\th)}\mathcal{M}=\tE^s_{K(\th)}\oplus
\tE^c_{K(\th)}\oplus \tE^u_{K(\th)}
\end{equation}
which is 
approximately invariant  
under the co-cycle 
$DF\circ  K$
 over $T_\omega$. That is,
\[
\dist_\rho ( DF\circ K(\th) \tE^{c,s,u}_{ K(\th)}, 
\tE^{c,s,u}_{K(\th + \omega)} )\le \delta ,
\]
where $\dist_\rho$ stands for  the supremum of the distance when 
$\th$ belongs to  $D_\rho$, the complex extension of 
the torus defined in \eqref{Drho}. We denote by 
$\Pi^{s,c,u}$ the projections corresponding to the  
above splitting.

Assume, moreover  that, for some $N \in \nat$, 
$0 < \tmu_1, \tmu_2 < 1$, and some $1 \le \tmu_3$, such that 
$\max(\tmu_1, \tmu_2) \cdot \tmu_3 < 1$,  we have

\begin{equation}\label{ndeg1Iter}
\begin{split}
|DF_{}&\circ K\circ T^{N-1}_{\omega}(\th)\times\dots \times
DF_{}\circ K(\th) v| \leq \tmu_1^N |v|\\
& \forall\,  v \in \tE^s_{{K}(\th)}, 
\end{split}
\end{equation} 

\begin{equation}\label{ndeg2Iter}
\begin{split}
|DF_{}^{-1}&\circ {K}\circ T^{-(N-1)}_{\omega}(\th)\times\dots \times
DF_{}^{-1}\circ {K}(\th) v| \leq \tmu_2^N |v|\\
&\forall \, v \in
\tE^u_{{K}(\th)}
\end{split} 
\end{equation} 
and

\begin{equation}\label{ndeg3Iter}
\begin{split}
&|DF\circ {K}\circ T^{N-1}_{\omega}(\th)\times\dots \times
DF\circ {K}(\th) v| \leq \tmu_3^N |v|\\
&|DF_{}^{-1}\circ {K}\circ T^{-(N-1)}_{\omega}(\th)\times\dots \times
DF_{}^{-1}\circ {K}(\th) v| \leq \tmu_3^N |v|\\
&\phantom{DF_{}^{-1}\circ {K}} \forall\, v \in \tE^c_{(\th)}.
\end{split} 
\end{equation} 
Assume that  $\delta < \delta_0$, 
where $\delta_0$ is an expression depending on 
$N$, $\| DF \circ K\|_\rho$, 
$\| DF^{-1} \circ K\|_\rho$, 
$\| \Pi^{c,s,u}\|_\rho$. 

Then, 
 there is an analytic 
splitting $$T_{K(\th)}\M = \E_{K(\th)}^s \oplus \E_{K(\th)}^u \oplus \E_{K(\th)}^c$$ 
invariant under the co-cycle $DF\circ K$ over $T_\omega$,
which satisfies the characterization of 
hyperbolic splittings \eqref{ndeg1}, \eqref{ndeg2},
\eqref{ndeg3}. 

The splitting above is unique among the splittings
in a neighborhood of the original splitting of 
size $\delta_0 $ measured in  $\dist_\rho$.

Furthermore, we have that 
\begin{equation}  \label{changebounds} 
\begin{split}
& \dist_\rho( \E_{K(\th)}^{s,u,c}, \tE_{K(\th)}^{s,u,c} ) \le C \delta,  \\
& | \mu_{1,2,3} - \tilde \mu_{1,2,3} | \le C \delta , \\
\end{split}
\end{equation} 
where $C$ depends  on the same quantities as 
$\delta_0$ does. 
\end{pro} 

The previous result is applicable to all 
co-cycles over $T_\omega$. It is important that the
base is a rotation. As it is well known in the general theory
of hyperbolic systems, 
if the base of the co-cycle had non-zero Lyapunov exponents,
we expect that the invariant splittings are only finitely differentiable
and not analytic even if the co-cycle and the base map are
analytic. Some explicit examples are available in 
\cite{Llave-sobolev}.

In the statement of Proposition \ref{PropDEG}, for typographical simplicity, we are assuming that the phase space is 
an  Euclidean manifold so that we can compute the product
$DF \circ K(\th + \omega) DF \circ K(\th)$ 
and consider $DF\circ K $ as a co-cycle over $T_\omega$. 
In case that the phase space is not an Euclidean manifold, 
the co-cycle is $S_{K(\th + \omega)}^{F \circ K(\th)} 
DF\circ K$, where $S$ is the connector introduced in Definition~\ref{connector}.
This can be done provided that
 $\dist( F\circ K(\th), K(\th + \omega))$ is 
small enough so that the connectors can be defined. 
The proof of Proposition~\ref{PropDEG} does not require any 
change beyond that to work in non-Euclidean manifolds.

Note that Proposition \ref{PropDEG} 
implies immediately the persistence of invariant bundles under
perturbations of the co-cycle.  Given a co-cycle, its 
invariant bundles are approximately invariant under the 
perturbed co-cycle.
The approximately invariant co-cycles
can be obtained in many different  ways, for example through 
numerical computations or through formal expansions. 
For the numerical applications we refer to 
\cite{HLlth}.  We also mention that \cite{Masdemont05}
computes Lindstedt series expansions for quasi-periodic
solutions in center manifolds for problems in celestial
mechanics. These solutions are whiskered solutions in the full
space and can be validated applying the results of this paper.

\begin{remark}\label{musCs}
Notice that the statements of the hyperbolicity conditions in 
Proposition~\ref{PropDEG} do not involve any constant 
$C_h$ as in \eqref{ndeg1}, \eqref{ndeg2}, \eqref{ndeg3}, but 
on the other hand, we include an $N$. {From} the point of 
view of mathematical theorems, both formulations are 
equivalent if we consider \eqref{ndeg1}, \eqref{ndeg2}, \eqref{ndeg3} for fixed $n=N$. Note that if $\tmu > \mu$ and 
$N$ are such that $C_h (\mu/\tmu)^N < 1$, 
the conditions in \eqref{ndeg1} imply those in 
Proposition~\ref{PropDEG}. The converse is trivial.

We note that the constants $C_h$ depend on the norm
used in the space. Indeed, in theoretical applications, 
it is convenient to choose a norm such that $C_h = 1$. 
Equivalently, one can choose a norm such that $N = 1$. This indeed simplifies 
the notation.  We have chosen not to take advantage of 
this simplification since the adapted norm is not commonly 
used in numerical applications. 
\end{remark} 

\begin{remark}
Another  application of 
Proposition \ref{PropDEG}  that we will not develop here, 
is a bootstrap of regularity. If an invariant splitting is 
continuous, smoothing it, we obtain an approximately invariant 
analytic one and, applying Proposition~\ref{PropDEG}, we 
obtain an analytic invariant splitting which has to 
coincide with the original one. See \cite{Johnson80,HLlmamotreto}. 
\end{remark}

\begin{remark}
With a view to the applications in 
\cite{FontichLS08b}, we note that the arguments in 
the proof of Proposition~\ref{PropDEG} 
are rather soft (contraction mapping principle and such). 
Hence, they go through without changes when the 
bundles are Banach bundles. 
\end{remark}

\begin{remark}\label{rem:onespace} 
Notice that the proof of the existence of invariant subbundles
given the approximately invariant ones is done one subbundle at a time.
Hence, if we have two invariant subbundles
(this is the situation considered in Proposition 
\ref{prop:distance}), the argument in the proof of
Proposition \ref{PropDEG} above leaves unchanged the invariant subspaces. 
Hence, the hyperbolicity constants $\mu_1, \mu_2, \mu_3$ and $C_h$ in these spaces are
unaltered. On the other hand, the projections on the invariant subspaces
are altered because the projections depend on the splitting. 
The change of one of the subbundles changes all the projections. 
Of course, the change of the projections can be estimated by the change 
of the spaces, which is in turn estimated by the error in the invariance
equation. 

\end{remark} 

The main application of Proposition 
\ref{PropDEG} 
in this paper is the following result, 
Proposition~\ref{PropDEG2}, which estimates the change
in the hyperbolicity hypotheses in an iterative step.

\begin{pro} \label{PropDEG2}
Assume that $\|K-\tilde{K}\|_{\rho}$ is small enough and hypotheses of 
Proposition \ref{improvement} apply. 
Then there exists an analytic invariant splitting for $DF\circ
\tilde{K}$.  

Furthermore, there exists a constant $C>0$ such that
we have the estimates   
\begin{align}
\label{proj1} \|\Pi_{{\tilde{K}(\th)}}^{s,c,u}-\Pi_{{K(\th)}}^{s,c,u}\|_{\rho
  } & \leq C\|\tilde K-{K}\|_{\rho},\\
\label{mus}|\tilde \mu_i - {\mu}_i|  & \leq C  \|\tilde K- {K}\|_{\rho},\qquad i=1,2,3, \\
\label{Cs}
\tilde C_h & = C_h.
\end{align} 
\end{pro}

\begin{proof}[Proof of Proposition \ref{PropDEG}] The  proof we present is very similar 
to the proof in \cite{HLlth}. The ideas are very similar to 
the standard proof of the persistence of invariant splittings in 
\cite{HirschPS77,PlissS99,Pesin04}
but we present them in an {\em a-posteriori} format,
obtaining very quantitative estimates
and we take advantage of the fact that the motion in the 
base is a rotation.
This requires only some minor rearrangements of the argument in the above references.

We will denote 
\begin{equation}\label{firstsplit}
\begin{split}
&\E_{K(\th)}^1 = \tE_{K(\th)}^s \\
&\E_{K(\th)}^2 = \tE_{K(\th)}^c \oplus \tE_{K(\th)}^u. 
\end{split}
\end{equation} 

We clearly have 
\begin{equation}\label{splitting2}
T_{K(\th)}\M = \E_{K(\th)}^1 \oplus \E_{K(\th)}^2
\end{equation} 
and the splitting
\eqref{splitting2} 
is almost invariant under $DF\circ K$.

We consider 
the matrix of $DF(K(\th))$ with respect to the splitting \eqref{splitting2}:
$$DF_{}(K (\th))=\begin{pmatrix} a_{11}(\th) & a_{12}(\th) \\
a_{21}(\th) & a_{22}(\th) \end{pmatrix} .
$$

The almost invariance of the splitting implies
that $\| a_{12}\|_\rho, \|a_{21}\|_\rho \le C\eta$. 

We will construct the invariant subspaces corresponding to this 
splitting as
graphs of linear functions 
$u^1(\th) : \E^1_\th  \rightarrow \E^2_\th$ and
$u^2(\th) : \E^2_\th  \rightarrow \E^1_\th$.  

Computing the image of the point $(x, u^1(\th) x)$, (resp. 
$(u^2(\th)y, y)$) and imposing that the 
images are in the graph of $u^{1}(\th + \omega)$ (resp. $u^{2}(\th + \omega)$), we 
obtain that 
the graphs of $u^1, u^2$ 
are invariant if and only if $u^1, u^2$ satisfy
\begin{align} \label{invariance-graph-1}
u^1(\th + \omega) 
( a_{11}(\th) + a_{12} (\th)  u^1 (\th)  ) & =
a_{21}(\th) + a_{22}(\th)  u^1( \th),     \\
a_{11}(\th)  u^2(\th)  + a_{12}(\th)  &= 
u^2(\th + \omega) (  a_{21}(\th)  u^2 (\th)+a_{22}(\th)   ).\label{invariance-graph-2}
\end{align}

As can be seen by elementary algebraic 
manipulations,
equations \eqref{invariance-graph-1} and \eqref{invariance-graph-2}
are equivalent to
\begin{align}\label{invariance-graph21}
u^1(\th)   & = a_{22}^{-1}(\th) 
( u^1(\th + \omega) ( a_{11}(\th)  + a_{12}(\th)  u^1(\th)  ) 
 - a_{21}(\th) ),\\
u^2(\th + \omega)   & =
(a_{11}(\th) u^2(\th)  + a_{12}(\th) ) 
( a_{22}(\th)   +  a_{21}(\th)  u^2(\th) )^{-1}.
\label{invariance-graph22}
\end{align} 

We see that $u^1, u^2$ are fixed points of 
the operators $\Tau^1, \Tau^2$ which are defined as the right-hand side of  
equation 
\eqref{invariance-graph21}  
and the rigth-hand side of equation \eqref{invariance-graph22} shifted by $-\omega$,
respectively:

\begin{equation*}
\begin{split} 
\Tau^1[u^1](\th)     = & a_{22}^{-1}(\th) 
( u^1(\th + \omega) ( a_{11}(\th)  + a_{12}(\th)  u^1(\th)  ) 
 - a_{21}(\th) ),\cr
\Tau^2[u^2](\th)   =  &
(a_{11}(\th -\omega) u^2(\th -\omega )  + a_{12}(\th -\omega) ) \times \cr 
&\times
( a_{22}(\th -\omega )   +  a_{21}(\th -\omega)  u^2(\th -\omega) )^{-1} . \cr
\end{split} 
\end{equation*}

Now we concentrate on the operator  $\Tau^1$.
We introduce the space $\S =\A(D_\rho,\L_1)$ of analytic sections from 
$D_\rho$ to the unit bundle of linear operators from $\E^1_{K(\th)}$ into 
$\E^2_{K(\th)}$, i.e. the space of analytic maps $u$ such that 
$u(\th) :\E^1_{K(\th)} \to\E^2_{K(\th)}$ is linear and 
$\|u(\th) \|\le 1$.
Endowed with $\|u \|_\S= \sup_{\th \in D_\rho}\|u(\th) \|$, $\S$ is a Banach space.
Moreover $\S$ satisfies Banach algebra  properties under the natural 
multiplications. 

We note that if $\eta $ is small enough and consequently $\|a_{12}\|$, $\|a_{21}\|$ are small,
a reasonable linear approximation of $\Tau^1$ is 
(obtained by eliminating all the terms that contain 
$a_{12}, a_{21} $) 
\begin{equation*}
\Tau^1_0 [u^1](\th) := a_{22}^{-1}(\th) u^1(\th + \omega) 
a_{11}(\th) .
\end{equation*} 

An elementary computation gives
\begin{align*} 
 (\Tau^1_0&)^N[u^1](\th) \\ 
 &= 
a_{22}^{-1}(\th) \cdots a_{22}^{-1}(\th + (N-1)\omega)
u^1( \th + N \omega) 
a_{11}(\th + (N-1)\omega) \cdots 
a_{11}(\th) .
\end{align*} 

Using the fact that $\Tau^1$ is a quadratic polynomial operator, by performing
algebraic manipulations we obtain
\begin{equation} \label{elementarybounds1} 
\max_{\|u^1\|_\rho \le \eta  } 
\| (\Tau^1)^N[u^1] - (\Tau^1_0)^N[u^1]\|_\rho \le C \eta 
\end{equation} 
and
\begin{equation} 
\label{elementarybounds2}
\Lip_{B_\eta  } \big(
(\Tau^1)^N -  (\Tau^1_0)^N \big) \le C \eta ,
\end{equation} 
where $C$ depends on $N$.

Note that, by assumptions \eqref{ndeg1Iter}, \eqref{ndeg2Iter}, \eqref{ndeg3Iter}
we have that if $\eta$ is small enough, $(\Tau^1)^N$
maps $\S$ into $\S$.

Then using  \eqref{ndeg1Iter}, \eqref{ndeg2Iter}, \eqref{ndeg3Iter}
together with the previous estimates 
we have that $(\Tau^1_0)^N$ is a contraction from $\S$ to $\S$.

This implies that also $(\Tau^1)^N$ is a contraction. It is well known that then 
$\Tau^1$ has a unique fixed point $u$ in $\S$.

Moreover the analyticity in $\th \in D_\rho$ is inherited by the fixed point of the contraction.
Hence $u$ depends analytically in $\th$. 

Furthermore, we have the standard fixed point estimate 
\begin{equation}\label{pfEQ}
\| u\|_{\mathcal{S}} \leq \frac{1}{1-\alpha}((\mathcal T^1)^N(0)-0) \leq C \eta,
\end{equation}
where $\alpha=C'\delta$ and the constant $C'$ depends on $N, C_h,\mu_1,\mu_3$. This estimate gives
that $d_\rho(\E^s_{K(\th)}, \tilde \E^s_{K(\th)}) \le C\eta$. Then, 
since the spaces $\E^s$, $\tilde \E^s$ are $C\eta $-close we also have 
$|\tilde \mu_1- \mu_1 | \le C\delta$.

The proof so far, gives us the existence of invariant spaces 
as in  \eqref{firstsplit}. This clearly gives us the existence of 
the invariant bundle  $\E^s_{K(\th)}$ and the invariant bundle 
$\E^{cu}_{K(\th)} $. 

We remark that exactly the same proof works if 
we take the splitting
\begin{equation}\label{firstsplit2}
\begin{split}
&\E^1_{K(\th)} = \tE^s_{K(\th)}  \oplus \tE^c_{K(\th)}\\
&\E^2_{K(\th)} = \tE^u_{K(\th)}. 
\end{split}
\end{equation} 
Hence, we also obtain the existence of 
the bundles $\E^{sc}_{K(\th)}$ and $\E^u_{K(\th)}$. 
The invariant bundle $\E^c_{K(\th)}$ is obtained as
$\E^{cu}_{K(\th)} \cap \E^{cu}_{K(\th)} $.

This concludes the proof of Proposition~\ref{PropDEG}. 

\end{proof}

\begin{proof}[Proof of Proposition \ref{PropDEG2}]
We
just observe that we can take the invariant
splittings for $DF \circ K$ as approximately 
invariant for $DF \circ {\tilde K}$. Using 
Cauchy estimates, we see that we can take 
$\delta = C \|\tilde K  - K\|_\rho$. 
Therefore, \eqref{proj1} follows from estimating the change of 
the spaces. The conclusions \eqref{mus}, \eqref{Cs} follow 
from the observations in Remark~\ref{musCs}. 

\end{proof}

The next Lemma~\ref{change-twist}
provides the perturbation for the remaining non-degenerate conditions. 
The idea is very simple.  The twist condition is just the norm of 
a matrix obtained by restricting the derivative to the tangent 
and projecting it on the symplectic conjugate directions to the tangent. 
Cauchy estimates allows us to estimate easily the changes of these spaces. 
The estimate of the change of the derivative when we change the embedding
is just the mean value theorem.

\begin{lemma}\label{change-twist}
Assume that the hypotheses of Proposition \ref{improvement} hold. If $\|E_{m-1}\|_{\rho_{m-1}}$ is small enough, then
\begin{itemize}
\item If $DK_{m-1}^\top DK_{m-1}$ is invertible with inverse $N_{m-1}$
  then $DK_{m}^\top DK_{m}$ is invertible with inverse $N_m$ and we have    
\begin{equation*}
\|N_m\|_{\rho_{m}} \leq \|N_{m-1}\|_{\rho_{m-1}}+C_{m-1} \kappa^2 \delta_{m-1}^{-(2\nu+1)} \|E_{m-1}\|_{\rho_{m-1}}.
\end{equation*}
\item If $\avg (A_{m-1})$ is non-singular then also $\avg  (A_{m})$ is and we have the estimate
 \begin{equation*}
|(\avg(A_m))^{-1}|\leq |(\avg(A_{m-1}))^{-1}|+C'_{m-1} \kappa^2 \delta_{m-1}^{-(2\nu+1)} \|E_{m-1}\|_{\rho_{m-1}}.
\end{equation*}
\item If $\avg (Q_{m-1 })$ is non-singular then also $\avg (Q_{m })$ is and we have the estimate
 \begin{equation*}
|(\avg(Q_m ))^{-1}|\leq |(\avg(Q_{m-1} ))^{-1}|+C''_{m-1} \kappa^2 \delta_{m-1}^{-(2\nu+1)} \|E_{m-1}\|_{\rho_{m-1}}.
\end{equation*}
\end{itemize}
\end{lemma}
\begin{proof}  For the first, we refer the reader to (\cite{LGJV05}, Section 5) since
the proof is identical. We turn to the second and third points. The
estimates just come from writing $K_m=K_{m-1}+\Delta_{m-1}$, using
estimates \eqref{improve1}-\eqref{improve2} and neglecting quadratic
error terms at the price of changing the constants. 
\end{proof} 

\subsection{Convergence of the scheme} \label{convergence}

It is by now classical that, under sufficiently strong 
smallness assumptions, the iterative scheme can be 
iterated indefinitely and that it converges. 
Similar arguments can be found in almost 
any paper in KAM theory, in particular 
 \cite{Zehnder75},
\cite{Zehnder76}, \cite{Moser66a}, \cite{Moser66b}, \cite{Llave01b}. 
The notation in this paper  matches closely that in 
\cite{LGJV05} so that the modifications, at this stage are 
rather minimal. 

Recall that we have identified a set of embeddings in which 
we can obtain uniform constants in the Newton step, 
see Proposition~\ref{improvement}. 

In the following Lemma~\ref{conv} we show that, with 
the choice of domain losses given in \eqref{deltachoices}
 if the 
initial error is small enough, the iterations do not leave
the neighborhood where we have uniform estimates and converge to 
a solution of the problem, which also has  hyperbolic splittings.

\begin{lemma}\label{conv}
Using the previous notations, let $C_m$ be the sequence of positive numbers defined above. For a fixed
$0<\delta_0\le\min(1,\rho_0/12)$ define for $m \geq 0$, 
\begin{equation}\label{deltachoices}
\delta_m=\delta_0 2^{-m},
\end{equation}
Denote 
$\rho_m = \rho_{m-1} - 6 \delta_{m-1}$ and
$\epsilon_m = \| E_m\|_{\rho_m}$.

There exists a constant $C$ depending on $l$, $\nu$, 
$|F|_{C^2(B_r)}$, $|J|_{C^1(B_r)}$, $\|DK_0\|_{\rho_0}$,
$\|N_0\|_{\rho_0}$, $|(\avg( Q_0 ))^{-1}|$, $|(\avg( A_0))^{-1}|$,
$\|\Pi^{s,c,u}_{K_0(\th)}\|_{\rho_0}$, $\|G\|_{\rho_0}$ such that if the error
$\epsilon_0$ satisfies the following inequalities 
\begin{equation*}
C 2^{4\nu}\kappa^4 \delta_{0}^{-4\nu} \epsilon_0 <1/2
\end{equation*}
and
\begin{equation*}
C (1+\frac{2^{4\nu}}{2^{2\nu}-1})\kappa^2 \delta_{0}^{-2\nu} \epsilon_0 <r,
\end{equation*}
then the modified Newton step can be iterated indefinitely and we obtain that $K_m$
converges to a map $K_{\infty} \in
\mathcal{A}_{\rho_0-6\delta_0}$ which satisfies the non-degeneracy
conditions, in particular, it is hyperbolic, and 
\begin{equation*}
F_{}\circ K_{\infty}=K_{\infty} \circ T_{\omega}.
\end{equation*} 
Moreover, there exists  a constant $D>0$ depending on $l$, $\nu$, 
$|F_{}|_{C^2(B_r)}$, $|J|_{C^1(B_r)}$, $\|DK_0\|_{\rho_0}$,
$\|N_0\|_{\rho_0}$, $|(\avg( Q_0))^{-1}|$, $|(\avg(A_0))^{-1}|$,
$\|\Pi^{s,c,u}_{K_0(\th)}\|_{\rho_0}$ and $\|G\|_{\rho_0}$ such that 
\begin{equation*}
\|K_{\infty}-K_0\|_{\rho_0-6\delta_0} \leq D \kappa^2 \delta_0^{-2\nu}
\|E_0\|_{\rho_0} .
\end{equation*}

\end{lemma}

\begin{remark}
We note again that by the vanishing lemma \ref{vanishing}, the sequence $\left \{ \lambda_n \right \}_{n \geq 0}$ converges to
$0$ as $n$ goes to $+\infty$. 
\end{remark}

\begin{proof} 
As mentioned in the introduction of the section, the argument is 
quite standard. 

To ensure that we can perform steps with the estimates in 
Proposition~\ref{improvement}, we just need to verify that 
we do not leave the neighborhood of $K_0$ given
by \eqref{closetoK0} and that we satisfy the bounds
\eqref{compositiondefined}. 

We note that, in a concise notation, Proposition~\ref{improvement} 
leads to the bounds 
\[ 
\epsilon_m \le C \kappa^4 \delta_{m - 1}^{-4 \nu} \epsilon_{m-1}^2
\]
With the choice $\delta_m=\delta_0 2^{-m}$, we see  that if we
can perform $m$ steps, we have:
\[
\begin{split} 
\epsilon_m & 
\le  C \kappa^4 \delta_0^{-4 \nu}2^{ 4 \nu (m-1) } \epsilon_{m-1}^2
\le  (C \kappa^4 \delta_0^{-4 \nu})^{1 + 2}  2^{4 \nu [ (m-1) + 2 (m-2) ] } \epsilon_{m-2}^{2^2}
\\ 
& \le 
 (C \kappa^4 \delta_0^{-4 \nu})^{1 + 2 + \cdots +2^{m-1}} 
 2^{4 \nu [ (m-1) + 2 (m-2) + \cdots +2^{m-2}  ] } \epsilon_{0}^{2^m}
\\
& \le
( C 2^{4\nu} \kappa^4 \delta_0^{-4\nu} \epsilon_0)^{2^m} ,
\end{split} 
\]
for $m\ge 1$,
where we have used that 
\[
(m-1) + 2 (m-2) + \cdots +2^{m-2} 
= 2^{m -2}[( m-1) 2^{-(m-2)} + (m -2) 2^{-(m-3)} + \cdots +1] 
\le 2^{m}.
\]

We see that if $\epsilon_0$ is small enough, then, 
$\epsilon_m \delta_m^{- 4 \nu}$ is so small than the conditions 
\eqref{compositiondefined} are true for the next step.  Indeed, we 
note that the smallness conditions that we need to impose in 
$\epsilon_0$ are independent of $m$.  

Furthermore, we also observe that we also have 
$K_m - K_0 = \sum_{i = 0}^{m -1} \Delta_i$. 
Hence 
\[
\| K_m - K_0\|_{\rho_m} \le   
 \sum_{i = 0}^{m -1} \| \Delta_i \|_{\rho_i}  
\le 
 \sum_{i = 0}^{m -1} 
C \kappa^2 ( C \kappa^4 \delta_0^{-4\nu} \epsilon_0)^{2^i} \delta_0^{-2 \nu} 
2^{ 2 i\nu} .
\]
We note that,  by taking $\epsilon_0$ small 
enough we can make the right-hand side of the last formula
as small as desired uniformly in $m$. 
In particular, by taking $\epsilon_0$  small enough we can 
ensure the assumption \eqref{closetoK0} for all $m$. 

Therefore, if we assume that  $\epsilon_0$ small enough, we can 
ensure that we can repeat the iterative step infinitely 
often and that the iteration never leaves the neighborhood identified 
in \eqref{closetoK0}. 

We also note that we have 
\[
\begin{split} 
 \sum_{i = 0}^{\infty} \| K_{i+1} - K_i \|_{\rho_\infty}   &= 
 \sum_{i = 0}^{\infty} \| \Delta_i \|_{\rho_\infty}   \le
 \sum_{i = 0}^{\infty} \| \Delta_i \|_{\rho_i}  \\
& \le 
C \kappa^2 \delta_0^{-2 \nu} \epsilon_0 \Big(1+
 \sum_{i = 1}^{\infty} 
C \kappa^2 ( C 2^{4\nu}\kappa^4 \delta_0^{-4\nu} \epsilon_0)^{2^i} \delta_0^{-2 \nu} 
2^{ 2 i \nu} \epsilon_0^{-1} \Big)\\
& \le
D \kappa^2 \delta_0^{-2 \nu} \epsilon_0 .
\end{split}
\]

The absolute convergence of the above series shows that $K_m$
converge to a limit and the last bound establishes the 
conclusion \eqref{estimation}. 
We note that since we had assumed \eqref{closetoK0}, we 
have that $K_\infty$ admits a hyperbolic splitting. 
Since the change in the hyperbolic splittings is bounded 
by the change of the embedding (see Proposition~\ref{PropDEG2}), 
we see that the hyperbolic splittings also converge to the limiting one. 
\end{proof}

\section{Proof of the local uniqueness theorem}\label{sec:uniqueness}

In this section, we prove Theorem~\ref{uniqueness}. We closely follow
the proof in \cite{LGJV05}. Similar results are more or less implicit 
in the treatment of whiskered tori in \cite{Zehnder76}. For
fully dimensional tori local  uniqueness 
results appear in \cite{Moser66a,SalamonZ89, Salamon04}.
As we have argued before, local uniqueness results allow us
to deduce results for flows from results for maps. 

The proof of Theorem~\ref{uniqueness} is  based on showing that the operator
$D\mathcal{F}_\omega(K)$ has an approximate left inverse (as in
\cite{Zehnder75}). Notice first that the composition on the right by every
translation of a solution of \eqref{embed} is also a
solution. Therefore, one cannot expect a general uniqueness
result. Moreover, the second statement in Lemma \ref{main} and the calculation on the hyperbolic directions
show that,
roughly speaking, two solutions of the linearized equation differ by
their average. Moreover this difference is in the direction of the tangent space of the torus.


The idea behind the local uniqueness result is to
prove that one can transfer the difference of the averages of two
solutions to a difference of phase between the two solutions.   

Now 
we assume that the embeddings $K_1 $ and $K_2 $
satisfy the hypotheses in Theorem \ref{uniqueness}, in particular $K_1$ and $K_2$ are solutions
of \eqref{embed}, or \eqref{translated} with $\lambda = 0$. 
If $\tau\ne 0$ we write $K_1$ for $K_1\circ T_\tau$ 
which is also a solution.
Therefore
$\mathcal{F}_{\omega}(0,K_1)=\mathcal{F}_{\omega}(0,K_2)=0$.
By Taylor's theorem we can write 
\begin{equation}\label{unique}
\begin{split}
0=\mathcal{F}_{\omega}(0,K_1)-\mathcal{F}_{\omega}(0,K_2)=&D_{\lambda,K}\mathcal{F}_{\omega}(0,K_2)(0,K_1-K_2)\\
&+\mathcal{R}(0,0,K_1,K_2),
\end{split}
\end{equation} 
where 
$$\mathcal R(0,0,K_1,K_2)=\frac{1}{2}\int_0^1 D^2F(K_2+t(K_1-K_2))(K_1-K_2)^2\,dt.$$
Then, there exists $C>0$ such that 
\begin{equation*}
\|\mathcal{R}(0,0,K_1,K_2)\|_{\rho} \leq C \|K_1-K_2\|_\rho^2. 
\end{equation*}
Hence we end up with the following linearized
equation 
\begin{equation} \label{linearized-5}
D_{\lambda,K}\mathcal{F}_{\omega}(0,K_2)(0,K_1-K_2)=-\mathcal{R}(0,0,K_1,K_2).
\end{equation}
We denote $\Delta=K_1-K_2$ 

Projecting \eqref{linearized-5} on the center subspace with $\Pi^c_{K_2(\th+\omega)}$, writing 
$\Delta^c(\th) = \Pi^c_{K_2(\th)}\Delta (\th)  $ and
making the change of function $\Delta^c(\th)= \tilde M(\th) W(\th) $, 
where $\tilde M$ is defined in \eqref{definicioMtilde} with $K=K_2$, we obtain
\begin{align}\label{change-uniq}
DF(K_2(\th))&\tilde{M}(\th)W(\th)-\tilde{M}(\th+\omega)W(\th+\omega) \nonumber \\
& = - \Pi^c_{K_2(\th+\omega )} \mathcal{R}(0,0,K_1,K_2)(\th).
\end{align}
We note that since $K_2$ is an exact solution, $\range \tilde M(\th)$ 
coincides with $\E^c_{K(\th)}$. See Subsection \ref{sec:basis}.

Applying the property $ DF(K_2(\th))\tilde{M}(\th) = \tilde{M}(\th+\omega) \mathcal{S}(\th)$
for solutions of \eqref{embed}, multiplying both sides by  $[\tilde{M}^\top J(K_2)](\th+\omega)$
and using that $\tilde{M}^\top J(K_2) \tilde{M} $ is invertible we get 
\begin{align*}
\mathcal{S}(\th) W(\th)& - W(\th+\omega) \\& = 
- [(\tilde{M}^\top J(K_2)\tilde{M})^{-1} \tilde{M}^\top J(K_2)](\th+\omega) \Pi^c_{K_2(\th+\omega )} \mathcal{R}(0,0,K_1,K_2)(\th).
\end{align*}
Since $W$ solves the previous equation,
we get bounds for it
using the methods in Section \ref{sec:temp}.
We write $W=(W_1, W_2) $. Since $\mathcal{S}$ is triangular we begin by looking for $W_2$. We search it 
in the form $W_2= W_2^\bot + \avg(W_2) $. We have
$\|W_2^\bot\|_{\rho-\delta}  \leq  C \kappa \delta ^{-\nu}  \|K_1-K_2\|_\rho^2 $.
For $W_1$ we have
\begin{align}\label{condW1}
W_1(\th) - W_1(\th+\omega) 
= & T_2(\th) (\Pi^c_{K_2(\th+\omega )} \mathcal{R}(0,0,K_1,K_2))_1(\th) \nonumber \\ 
 & - A(\th)W^\bot_2(\th) - A(\th) \avg(W_2),
\end{align} 
where $T_2= N_2^\top DK_2^\top J(K_2)^{-\top}[DK_2N_2DK_2^\top -\Id] J(K_2)$ 
and $N_2= DK_2^\top DK_2 $.

The condition the right-hand side of \eqref{condW1} to have zero average gives
$|\avg(W_2)  | \le  C\kappa \delta ^{-\nu} \|K_1-K_2\|_\rho^2$. Then 
$$
\|W_1 -\avg(W_1) \|_{\rho-2\delta} \le C\kappa^2 \delta ^{-2\nu} \|K_1-K_2\|_\rho^2
$$
but $\avg(W_1) $ is free. Then 

\begin{equation*}
\|\Delta^c-(\avg(\Delta^c)_1,0)^\top \|_{\rho-2\delta} \leq C \kappa
^2 \delta ^{-2\nu } \|K_1-K_2\|_\rho^2.
\end{equation*}
The next step  is done in 
the same way as in \cite{LGJV05}. We quote Lemma 14 of that reference
using our notation.   
It is basically an application of the standard implicit function theorem. 

\begin{lemma} \label{lem:tau}
There exists a constant $C$ such that if $C \|K_1-K_2\|_{\rho} \leq 1$
then there exists an initial phase $\tau_1 \in \left \{ \tau   \in
  \mathbb{R}^l \mid \; |\tau|< \|K_1-K_2\|_{\rho} \right \}$ such that 
 \begin{equation*}
\avg(T_2(\th) \Pi^c_{K_2(\th)}(K_1\circ T_{\tau_1}-K_2)(\th))=0. 
\end{equation*}
\end{lemma}
The proof is based on an  application of
implicit function theorem in $\mathbb{R}^l$. 

As a
consequence of Lemma \ref{lem:tau}, if $\tau_1$ is as in the statement, then $K \circ
T_{\tau_1}$ is a solution of \eqref{embed} such that if 
\begin{equation*}
W=[\tilde{M}^\top J(K_2)\tilde{M}](\th+\omega)^{-1}[\tilde{M}^\top J(K_2)](\th+\omega) \Pi^c_{K_2(\th)}(K_1\circ
T_{\tau_1}-K_2), 
\end{equation*}
for all $\delta \in (0,\rho/2)$ and 
we have the estimate 
\begin{equation*}
\|W\|_{\rho-2\delta} <C \kappa^2
\delta^{-2\nu} \|\mathcal{R}\|^2_{\rho} \leq C \kappa^2 \delta^{-2\nu}
\|K_1-K_2\|^2_\rho.  
\end{equation*}
This leads to on the center subspace 
\begin{equation*}
\|\Pi^c_{K_2(\th)}(K_1\circ T_{\tau_1}-K_2)\|_{\rho-2\delta} \leq C \kappa^2
\delta^{-2\nu} \|K_1-K_2\|^2_{\rho}. 
\end{equation*}
Furthermore, as in Section \ref{sec:hyper}, taking 
projections on the  hyperbolic subspace, we have that
$\Delta^h=\Pi^h_{K_2(\th)} (K_1-K_2)$ satisfies the estimate 
\begin{equation*}
\|\Delta^h\|_{\rho-2\delta} <C \|\mathcal{R}\|_\rho.
\end{equation*}
All in all, we have proven the estimate for $K_1 \circ T_{\tau_1}
-K_2$ (up to a change in the original constants)
\begin{equation*}
\|K_1\circ T_{\tau_1}-K_2\|_{\rho-2\delta} \leq C \kappa^2
\delta^{-2\nu} \|K_1-K_2\|^2_{\rho}. 
\end{equation*}
We are now in position to carry out an argument very 
similar to the one  used in Section \ref{convergence}. 
We can take a sequence $\left \{ \tau_m \right
\}_{m \geq 1}$ such that $|\tau_1| \le \|K_1  -K_2 \|_{\rho}$ and 
\begin{equation*}
|\tau_m -\tau_{m-1}| \leq \|K_1 \circ T_{\tau_{m-1}} -K_2 \|_{\rho_{m-1}} , \qquad m\ge 2,
\end{equation*}
and 
\begin{equation*}
\|K_1 \circ T_{\tau_{m}} -K_2 \|_{\rho_{m}} \leq C \kappa^2
\delta_m^{-2\nu} \|K_1 \circ T_{\tau_{m-1}} -K_2 \|^2_{\rho_{m-1}},
\end{equation*}
where $\delta_1=\rho/4$, $\delta_{m+1}=\delta_m/2$ for $m\ge 1$ and 
$\rho_0=\rho $, $\rho_m=\rho_0-\sum_{k=1}^m \delta_k$ for $m\ge 1$.
By an induction argument we end up with 
\begin{equation*}
\|K_1 \circ T_{\tau_{m}} -K_2 \|_{\rho_{m}} \leq
(C\kappa^2\delta^{-2\nu}_1 2^{2\nu} \|K_1-K_2\|_{\rho_0})^{2^m}
2^{-2\nu m}. 
\end{equation*} 
Therefore, under the smallness assumptions on $\|K_1-K_2\|_{\rho_0}$,
the sequence $\left \{\tau_m \right \}_{m \geq 1}$ converges and one gets 
\begin{equation*}
\|K_1 \circ T_{\tau_\infty} -K_2\|_{\rho /2}=0.
\end{equation*} 
Since both $K_1 \circ T_{\tau_\infty} $ and $ K_2$ are analytic in $D_\rho$ and coincide in 
$D_{\rho/2}$ we obtain 
the result.

\section{Applications} \label{sec:applications}

In this section, we collect several consequences of 
our main theorem. We note that these consequences
follow mainly from the fact that we have formulated the 
theorem in {\sl a posteriori} style without reference to 
an integrable system. 
\subsection{Lipschitz dependence with respect to 
the frequency. Estimates of the measure occupied 
by the tori}

The basic idea is that if we have an embedding $K$ that solves
the equation for one frequency, then it solves approximately 
the equation for a nearby frequency. Then, applying 
Theorem \ref{existence}, there should be a solution for 
a new frequency which is close to the the original one.
Performing the argument with care, we see that this implies
Lipschitz dependence of the solution on the frequency. Similar ideas were indicated in 
\cite{Zehnder75}. We remark that this Lipschitz dependence 
leads to estimates on the measure occupied by the tori 
in the perturbative case. We concentrate on the case of maps since, 
as we have shown, it implies the corresponding result for 
flows. 

We assume that $F$ is defined and analytic in a 
sufficiently big complex domain
of an Euclidean manifold $\M$.  
We consider $\omega \in D(\kappa, \nu)$ with $\kappa$ and $\nu$ fixed 
and we suppose that $K_\omega \in \A_\rho$ satisfies 
\eqref{embed} and is non-degenerate in the sense of
Definition~\ref{ND}. For all 
$\omega' \in D(\kappa, \nu)$ we have
\begin{equation*}
F\circ K_\omega - K_\omega \circ T_{\omega'}  = 
K_\omega \circ T_\omega - K_\omega \circ T_{\omega'}
\end{equation*}
and therefore, applying the mean value theorem and Cauchy 
estimates, we have
\begin{equation} \label{approximate} 
\| F\circ K_\omega - K_\omega \circ T_{\omega'} \|_{\rho - \delta} 
\le C \delta^{-1} \|K_\omega\|_{\rho}  | \omega - \omega'|
\end{equation} 
for all $ \delta \in (0,\rho)$.
If $| \omega - \omega'|$ is small enough, namely 
$C\kappa ^4 \delta^{-4\nu-1} \|K_\omega\|_\rho |\omega-\omega'| <1$,
applying 
Theorem~\ref{existence} we obtain that there is an embedding
$K_{\omega'}$ satisfying 
\eqref{embed} with the frequency $\omega'$. Furthermore, taking the value of $\delta $ in Theorem ~\ref{existence} appropriately, one gets 
\begin{equation*}
\| K_\omega - K_{\omega'} \|_{(\rho - \delta)/2} 
\le C \kappa^2 (\rho-\delta) ^{-2\nu}\delta^{-1}\|K_\omega\|_\rho |\omega - \omega'|.
\end{equation*}

We note that, in the domain of applicability of
the previous argument, the Lipschitz constant is
uniform since we are assuming that $\kappa$ and $ \nu$ are fixed.  
Since the set of uniformly Diophantine vectors $D(\kappa, \nu)$ is
locally 
compact, we can cover a bounded subset of $D(\kappa, \nu)$ 
with a finite 
number of balls in which the previous argument applies
and therefore in this set we get Lipschitz dependence 
on $\om$.

Moreover, the frequency of $K_\omega$ is given by the formula (up to
some multiplicative constant)
\begin{equation}\label{rotNumber}
\omega= \Pi_\varphi \int_{\mathbb{T}^l}(\tilde{F}\circ
\tilde{K}_\omega(\th)-\tilde{K}_\omega(\th))\,d\th, 
\end{equation}
where $\tilde{F}$ and $\tilde{K}$ are the lifts of $F$ and $K$ to the
universal cover of $\M$ and $\Pi_\varphi$ is the projection over the 
lift of the 
angle variables.

Thanks to formula \eqref{rotNumber}, it is straightforward to see that
the map $K_\omega \mapsto \omega$ is Lipschitz.  
Hence, we conclude that the map 
$\omega \mapsto K_\omega$ is bi-Lipschitz from the set of 
Diophantine vectors with fixed Diophantine constants.
Since the set $D(\kappa, \nu)$ has positive $l$-dimensional 
measure, we conclude that the set of tori also has $2l$-dimensional
measure, i.e.
\begin{equation*}
\mathcal{H}^{2l} (\bigcup_{\omega \in D(\kappa,\nu)} K_\omega(\mathbb{T}^l))>0,
\end{equation*}
where $\mathcal{H}^{2l}$ stands for the Hausdorff measure.

\subsection{Analyticity with respect to parameters} 

The proof of the existence of tori associated to a fixed frequency $\omega$ presented here leads to analyticity in the dependence 
with respect to parameters.  Later, we will see 
that this leads to analyticity properties of some series expansions,
such as Lindstedt series. 
The argument is already contained in \cite{Moser67}. However 
the argument presented here is somewhat simpler 
than the one presented in that reference.

Given a family of functions $F_\eta$ and a family of 
approximate solutions $K_\eta$ both depending analytically
on parameters $\eta \in U \subset \complex^p$ ($p \geq 1$) and continuous
in $\overline{U}$, we see that the assumptions  
of Theorem \ref{existence} are satisfied uniformly in $\eta  \in \overline{U}$. Consequently, there is a true solution nearby which also depends analytically on the parameters $\eta$. 

The proof is very simple; we just observe that the iterative 
step is analytic for $\eta \in U$ (resp. continuous for 
$\eta \in \overline{U}$) if 
the family and the error are. 

The procedure and the 
estimates for one step of the iterative procedure
are stated in Lemma \ref{main} and in a more detailed way 
in the statements and proofs of Proposition~\ref{solCenter} and Proposition \ref{hyperb}. We just call
attention to the fact that  the correction applied 
at each step of the application of Proposition~\ref{solCenter} relies on  some explicit algebraic
formulas --- involving derivatives with respect to $\th$ --- 
and to use the solution of some small 
divisor equations. Note that the solution of the small divisor 
equations is obtained applying a linear operator which is 
independent of $\eta$.
Also the method of obtaining the projection on the hyperbolic directions done in Section \ref{sec:hyper}
and summarized in Proposition \ref{hyperb} preserves the analyticity on parameters
since $\Delta ^{s,u}$ are obtained as sums of uniformly convergent series.

Clearly, the analyticity properties 
with respect  to $\eta$  are preserved by all these 
steps. Hence, 
the corrections applied in one 
step of the iterative Lemma~\ref{main} depend analytically on parameters
when the error does. 
Of course, the error depends analytically on  $\eta \in U$ 
(resp. continuously on $\eta \in \overline{U}$) if 
the approximate  solution at the start of 
Lemma~\ref{repres} depends analytically  on the parameters,
since to compute the error from the approximate solution, we
just have to compose with the function $F_\eta$ and translate. 
Therefore, we conclude that the application of 
Proposition~\ref{solCenter} preserves the analyticity 
properties with respect to parameters.

Hence, in all  steps of the  iterative 
process used in the proof of Theorem~\ref{existence}, 
the functions $\left \{  K_m \right \}_{m \in \mathbb{N}}$ depend  analytically on parameters ranging on 
the open set $U$ (and continuously on the boundary). Of course, 
in the iterative step, we decrease the analyticity domain 
in the variables $\th$, but not the analyticity
domain  in $\eta$. 

We also observe that in the
proof of Proposition~\ref{solCenter}, the estimates on the correction
applied at each step
depend only on the sizes of the error and the non-degeneracy 
conditions. Also, we observe that the estimates on the change of 
the non-degeneracy conditions are uniform
on the size of the corrections. 
In particular, if we assume that the smallness and non-degeneracy 
conditions hold uniformly for $\eta \in \overline{U}$, we 
can apply Lemma~\ref{conv} to obtain that there is 
a sequence of analytic functions in $\th, \eta$ converging uniformly
for $ \th\in D_{\rho_\infty}$ and $ \eta \in \overline{U}$.

In the paper \cite{LlaveO00}  there is an alternative point of 
view for results with parameters. One can apply an 
abstract KAM implicit function theorem as in \cite{Zehnder75} to spaces of analytic functions in 
other Banach spaces. These kind of arguments were used to 
deal with rather degenerate problems. A more detailed study of 
KAM theorems with parameters appears in \cite{Vano02}. 

\subsection{Small twist theorems and small 
hyperbolicity theorems} \label{sec:smalltwist}

Small twist theorems have been introduced in \cite{Moser62,Kyner68} to deal
with problems in celestial mechanics. The idea of small twist (and
small hyperbolicity) theorems is to give conditions that ensure the
convergence of the Newton-type method even if the twist is close to be
degenerate. It goes through a more precise analysis of the constants
involved in the Newton scheme. 

We
refer to \cite{Moser62,Kyner68,Ortega97,Ortega99} for applications 
of small twist theorems  to celestial
mechanics and to the stability of oscillators. 

The goal of this section is to
provide, as a corollary of the proof of our main Theorem
\ref{existence}, a small twist and small hyperbolicity result.

By examining carefully the proof involved in the iterative 
step in the KAM  method (see Section  \ref{sec:temp}), 
we get the following proposition.

\begin{pro}\label{pro:constants}
\begin{itemize}
\item  There exist two positive numbers $\alpha, \beta$ such that 
the constant $C$ in equation \eqref{estimApprox} 
(which here will be denoted $C^c$) 
depending 
on $l$, $\kappa$, $\nu$, 
$|F|_{C^2(B_r)}$, $\|DK\|_{\rho}$, $\|N\|_{\rho}$,
$|(\avg(A))^{-1}| $, $|(\avg(Q))^{-1}| $,
$\|\Pi^c_{K(\th)}\|_{\rho}$ and $\|G\|_{\rho}$ can be estimated by 
\begin{equation}\label{ST1}  
\begin{split}
C^c \leq &\|\Pi^c_{K_0(\th)}\|_{\rho} \max
(1,\|DK\|_\rho)^\alpha \max(1,\|N\|_\rho)^\beta
(|(\avg(A))^{-1}| \\
&+|(\avg(Q ))^{-1}| ),
\end{split} 
\end{equation}
where $A$, $Q $ are defined in 
\eqref{A}, \eqref{Q} respectively. 

\item The constant $C$ in equation \eqref{estimHyperb}
(which here will be denoted $C^h$) depending on the hyperbolicity constant $\mu_1$
(resp. $\mu_2$), the norm of the projection
$\|\Pi^s_{K(\th)}\|_{\rho}$ (resp. $\|\Pi^u_{K(\th)}\|_{\rho}$)
and $\|G\|_{\rho}$ and
the constant $C_h$ involved in (\ref{ndeg1}) and (\ref{ndeg2}) can be estimated by 
\begin{equation}\label{ST2}
C^h \leq C_h (1+C^c) \max(\|\Pi^s_{K(\th)}\|_{\rho} \frac{1}{1-\mu_1},
\|\Pi^u_{K(\th)}\|_{\rho} \frac{1}{1-\mu_2}).
\end{equation} 

\item As a consequence of the two above items, 
the constant $C$ 
appearing in Theorem \ref{existence} (the one which enters in equation \eqref{constantIter})
can be bounded by 

\begin{equation}\label{ST3}
\begin{split}
&C_h^2 \max(\|\Pi^s_{K(\th)}\|_{\rho}
\frac{1}{1-\mu_1},\|\Pi^u_{K(\th)}\|_{\rho} \frac{1}{1-\mu_2})^2 \\
&+\|\Pi^c_{K_0(\th)}\|^2_{\rho} \max
(1,\|DK\|_\rho)^\alpha \max(1,\|N\|_\rho)^\beta
(|(\avg(A))^{-1}|  +|(\avg(Q))^{-1}| )^2.
\end{split}
\end{equation} 
\end{itemize}
\end{pro}

The  argument presented in this paper gives 
$\alpha = 4, \beta = 2$, but there are other variants of the 
argument which give better values. We have not optimized 
the bounds. 

To prove Proposition \ref{pro:constants} we note that
to find $\Delta^{s,u}$  we 
just use
formula \eqref{hyperbolicformula} from 
which the claim follows by estimating the sum using the 
triangle inequality and using the sum of the geometric 
series. 

The estimates claimed for the constants related to $\Delta^c$
follow by observing that the solution is obtained by applying
the following operations: 
multiplying by the matrices $\tilde{M}$, $\tilde{M}^\top $, multiplying by the 
matrices $N$, modifying the constants $\Lambda$ 
and choosing the average of $W_2$. The latter steps are 
estimated by  multiplying by
$|(\avg(Q ))^{-1}|$ and $|(\avg(A))^{-1}| $. 

We also recall that the remainder of 
the Newton method is estimated by the 
remainder of the Taylor expansion. Hence, it is 
estimated by the square of $\|\Delta\|_{\rho-\delta}$.

Let $m \geq 0$ be an index for the Newton step and denote $\tilde{C}$
the constant involved in the third item of Proposition
\ref{pro:constants}. We have for some $\upsilon >0$
\begin{equation*}
\|E_m\|_{\rho_m} \leq \tilde{C} \delta_m^{-\upsilon}
\|E_{m-1}\|^2_{\rho_{m-1}} ,
\end{equation*}
where $\rho_m=\rho_{m-1}-\delta_{m-1}$.

It is standard in KAM theory (see Theorem~\ref{existence}) that 
if 
$$\tilde{C} \delta_0^{-2\upsilon}\|E_0\|_{\rho_0}< C(\upsilon) \ll 1, $$  
then the Newton method with $\delta_m = 2^{-m} \delta_0$
converges to a solution. 

Therefore, even if the twist and the hyperbolicity are close to
degenerate so that $\tilde{C}$ is large, if the initial error is small enough, one
gets convergence of the scheme. 

Small hyperbolicity arises naturally in perturbations of 
integrable systems. The integrable system, of course, has no 
hyperbolic behaviour, but an averaged system has some 
small hyperbolicity. Indeed, similar considerations for 
periodic orbit happen already in \cite[Ch. 74, 79]{Poincare99}.

The papers  \cite{JorbaLZ00, LlaveW04, ChengW99, Cheng99, Eliasson94, Treschev94a}
 consider perturbations of 
integrable systems at resonances,  where the hyperbolicity is 
small and the present result 
can be applied.  These papers differ in several important 
aspects such as the dimension and the topology of the tori 
considered. The methods are also different. 

All of the above papers consider 
perturbations of quasi-integrable systems $H_0 + \varepsilon H_1$.

The paper \cite{JorbaLZ00} shows that, 
given some appropriate non-degeneracy conditions on the perturbation, 
  it is possible to 
construct  formal series of approximate solutions in 
powers of $\ep$. 
Truncating the series up to order $N$, it is shown that the
error of the power series can be bounded 
by $C N^{a N} \ep^N$. Similarly, for some of the solutions,
$\Pi^{s,u,c}_{K(\th)}$ are of order
$(\Re \ep) ^{-1/2}$, the hyperbolicity constants $(1-\mu_1)^{-1},(1-\mu_2)^{-1}$
are of order $(\Re \ep)^{-1/2}$ but $(\avg(A))^{-1}$ and
$(\avg(Q ))^{-1}$ are of order $1$.
The Diophantine constants can be assumed to be 
fixed. 

If we fix $r > 0 $ sufficiently small 
and consider the set $ r < |\ep| < 2 r$, we can choose the order  of 
truncation so  that the error is less that 
$C \exp( - b \ep^{-c})$. Then, the small hyperbolicity result 
 applies to show that there are invariant tori, which depend
analytically in $\ep$ for 
$\ep$ such that $ r < |\ep| < 2 r $ and 
$\Im \ep \ge  C \exp( - b (\Re \ep)^{c})$. 
Since $r$ is arbitrary, we obtain that there 
are hyperbolic invariant tori in a ball except, at most in 
a wedge around the positive real axis which is exponentially 
thin. 

We refer to \cite{JorbaLZ00} for precise conditions on the series 
so that we can get the perturbation series as above.
For the case of two degrees of freedom, the paper 
\cite{JorbaLZ00} considers the existence of 
elliptic tori. We note that the method of 
\cite{JorbaLZ00} is based on \emph{reducibility}.  This paper
shows that we do not need to use reducibility for 
the hyperbolic directions. The study of elliptic directions 
has experienced very significant progress in the last 
years, but we will not mention it here. 

The paper \cite{LlaveW04} considers also 
weakly hyperbolic tori around periodic orbits
generated by resonances.  Note that the tori considered
in \cite{LlaveW04} are \emph{secondary tori}, that is 
tori that cannot be deformed to tori with the same frequency 
in the  unperturbed system. Indeed, the tori 
can be deformed into tori with less angles. 
Since \cite{LlaveW04} involves reduction to a center manifold
it only concludes  that the tori are finite differentiable 
even if the system is analytic. A result which is improved 
by using the results provided in this paper in Section \ref{sec:bootstrap}.

\subsection{Secondary tori and whiskered  tori close to rank-$1$ resonances}

The method described in this paper can accommodate to study  secondary tori. 
Indeed, the development of algorithms which could deal with
secondary tori was an important motivation to modify   
the invariance equation by adding a term containing $\lambda$. 

We recall that secondary KAM tori are invariant tori, such that the motion on 
them are conjugate to an irrational rotation but in contrast to 
the usual KAM tori which are homotopic to $\torus^l \times \{0 \in 
\real^{2 d - l} \}$,
the secondary tori are homotopic to $\torus^{l - k} \times 
\{ 0 \in \real^{2 d - l +k } \}$.

The existence of secondary KAM tori is very apparent in numerical 
explorations. For example, in two-dimensional maps, they are known as
\emph{islands}. In two-dimensional maps, islands are quite visible
and they may occupy a large measure of the phase space. 

Note that secondary tori are not present in the integrable system
and their existence is not guaranteed by the standard pertubative
KAM theory, which is concerned with the persistence of the 
invariant tori already present in the integrable system. 
In contrast, they are generated by the perturbation. 
The perturbation theory is somewhat unconventional since 
the unperturbed system does not present the phenomenon. 
Perturbative proofs of existence of secondary tori 
are done in \cite{LlaveW04} and in \cite{DelshamsLS06a}.

In the recent papers \cite{DelshamsLS03,DelshamsLS06a} it is 
shown that these secondary tori can be used as effective tools to 
generate diffusion and, in particular, to overcome the 
\emph{large gap problem} in the study of diffusion. 
The paper \cite{HaroL00} argues heuristically and verifies numerically 
that, in multiparticle systems that will be considered in 
a follow-up of this paper, in particular 
in the celebrated Fermi-Pasta-Ulam 
\cite{FermiPU},  the secondary tori occupy a much larger 
volume of phase space than the primary tori.

The method to construct whiskered tori in 
\cite{LlaveW04} was to show that, under explicit conditions on the 
perturbation, the rational frequencies give rise to periodic orbits, some of 
which admit center manifolds. Under appropriate non-degeneracy 
conditions, these center manifolds contain tori which are 
invariant under the restriction. These invariant tori, are 
whiskered tori for the full system. They are secondary since the 
directions corresponding to the center directions can be contracted to 
a point. 

The paper \cite{DelshamsLS06a} shows that secondary 
tori are generated by resonances in systems such that the unperturbed 
system has a two-dimensional normally hyperbolic manifold. The method of 
proof is to show that, near the resonances, one can approximate the 
system by a system which is pendulum like. This pendulum has orbits that 
are rotating. In \cite{DelshamsLS06a}, it is shown that one can consider
the real system as a perturbation of the pendulum and, therefore that some of 
the tori present in the pendulum are also present in the real system. 
See also \cite{DelshamsH06}. 

One of the difficulties of the method in \cite{DelshamsLS06a} is 
that the action variables near the separatrix are singular. This 
difficulty is, of course, not a problem for the method developed in the present paper. The method used in \cite{DelshamsLS06a} to 
overcome the singularity of the action variables was to perform 
more averaging steps, which required assuming more regularity of 
the perturbation. Applying the methods of this paper allows us
to reduce the number of derivatives assumed in \cite{DelshamsLS06a}. 

As we will discuss in more detail in Section~\ref{sec:bootstrap}, 
by using reduction to center or normally hyperbolic 
invariant manifolds, one can only prove that the obtained tori are finitely differentiable even if 
the mapping is analytic. 
Using the results of this paper, we will show 
that these tori are actually analytic if the map is. 

\subsection{Bootstrap of regularity of invariant tori}
\label{sec:bootstrap}

In this section, we show that if an analytic exact symplectic map 
$F$ admits an invariant torus, with the maximal number of hyperbolic directions
permitted by the symplectic structure, of class  $C^r$ with $r$ large enough, then 
the torus is actually analytic.  
Similar results for Lagrangian tori have been proved in 
\cite{SalamonZ89}. 

\begin{pro}  \label{bootstrap}
Let $F:\M \rightarrow \M$ be an analytic exact symplectic map. 
Let $\omega \in D(\kappa, \nu) $ for some $\kappa>0$ and $ \nu \geq l$,
and $K: \torus^l \rightarrow  \M$ satisfy 
\begin{enumerate}
\item $K$ is a solution of the equation $F\circ K - K\circ T_\omega=0$. 
\item 
$K$ is non-degenerate in the sense of
Definition \ref{ND}. 
\item $K$ is $C^r$ with  
\begin{equation}\label{rlarge}
r > 4 \nu .
\end{equation}
\end{enumerate}

Then $K$ is analytic. 
\end{pro}

\begin{remark} One case when Proposition~\ref{bootstrap}  is useful is 
when the tori are produced  by a reduction 
to a center manifold  or a normally hyperbolic manifold. 
These invariant manifolds are only finitely differentiable. 
Applying the above result, Proposition~\ref{bootstrap}, 
we can conclude that the tori are analytic. 

The paper 
\cite{LlaveW04} constructs  whiskered tori
by applying the KAM theorem 
for Lagrangian tori to the restriction of the system 
to a center manifold.   The papers 
\cite{DelshamsLS00, DelshamsLS03, DelshamsLS06a} 
consider tori in a normally hyperbolic manifold. 
In particular, \cite{DelshamsLS06a} constructs 
secondary tori. We conclude that, in the case that the
considered system is analytic, Proposition~\ref{bootstrap}
shows that the tori are analytic. 
\end{remark}

The idea of the proof is very simple. We approximate 
the function $K$ by an analytic one which 
will be an approximate solution of equation \eqref{embed}. 
Applying our main Theorem \ref{existence}
we will obtain that there is an analytic invariant torus nearby. 
The smoothness in the assumption enters because 
the hypotheses of Theorem~\ref{existence} involve the
size of the error in a complex strip. 
The uniqueness result 
Theorem~\ref{uniqueness}  will give that the
analytic torus obtained this way coincides with the original one up to a translation in the ``angles". 

The construction of the analytic approximations 
could be done in many different ways. For example, 
truncating the Fourier series of $K$ would do, 
if one assumes a condition stronger than 
\eqref{rlarge}. 

As it is well-known since \cite{Moser66a}, a very efficient 
way of approximating smooth functions by analytic ones is 
performing a 
convolution with a suitable kernel. 

Following \cite{Zehnder75, Moser66a}, we introduce 
smoothing operators that provide natural ways of 
approximating smooth functions by analytic ones. 

\begin{defi}\label{smoothing}
Let $u :\real^{l} \to \real$ be  a $C^\infty$ even function identically
$1$ in a neighborhood of the origin and with
support contained in the unit ball. Let $\hat{u}: \real^{l} \to \complex$
be the Fourier transform of $u$
and denote by $v$ the holomorphic continuation of $\hat{u}$.
For $ f\in C^{0}(\mathbb{R}^l)$ and $t>0$ we define
\begin{equation}\label{eq:smoothing}
S_t[f](z) := t^l \int_{\real^{l}} v(t(y-z))f(y) dy.
\end{equation}
\end{defi}
The map $S_t $ defines a  linear operator from  $C^0(\mathbb{T}^l)$ to the space of analytic maps
from $D_\rho$ to $\CC$, $\rho>0$.
Moreover,  
$S_{t}$ is an analytic smoothing operator in the sense
of \cite{Moser66a, Zehnder75} since it satisfies the following
proposition (see \cite[Lemma 2.1]{Zehnder75} for 
a proof).
We recall that if $g\in \A_\rho$, $\|g\|_\rho = \sup _{z\in D_\rho} |g(z)|$. 

\begin{pro}\label{prop-smoothing}
Let $r\in \real^+ \backslash \nat$. 
There exists a
constant $\kappa_1=\kappa_1(l,r)$ such that
for all $t\ge 1$ and all $f \in C^{r}( \torus^l)$
we have 
\begin{enumerate}
\item\label{part1}
$ |(S_{t}- \Id)\, [f] |_{C^{0}}\le \kappa_1\,
 |f |_{C^{r}}\,\, t^{- r}\, $,
\item\label{part2}
$\| S_{t}[f] \|_{t^{-1}} \le \, \kappa_1  |f |_{C^{0}}$,
\item\label{part3}
$\|( S_{\tau} - S_{t} ) [f] \|_{\,\tau^{-1}}
\le \kappa_1\,  |f |_{C^r}\,\, t^{-r},\qquad$ 
for all $\tau\ge t$.
\end{enumerate}
\end{pro}

We note that, since the smoothing operator commutes with derivatives, 
we also have the following extensions of \eqref{part1} and
\eqref{part2} for $s\le r$ 
\begin{eqnarray} 
&  |( S_{t}- \Id)\, [f] |_{C^{s}}\le \kappa_1\,
 |f |_{C^{r}}\,\, t^{- r+s} \label{part4} ,\\
& \|D^ s S_{t}[f] \|_{t^{-1}} \le \, \kappa_1  |f |_{C^{s}}.
\label{part5}
\end{eqnarray}

\begin{proof}[Proof of Proposition \ref{bootstrap}]
We consider $S_t [K]$, the smoothed version of $K$ with $t\ge 1$.
Our first goal is to estimate the error in a domain of size $t^{-1}\xi$ with
$\xi \in (0,1)$. 

We note that, by \eqref{part2} in Proposition \ref{prop-smoothing}
and \eqref{part5}, $\| S_t [K]\|_{t^{-1}}\le \kappa_1  |K | _{C^0}$ 
and 
$\| D  S_t [K]\|_{t^{-1}} \le \kappa_1  |  K  |_{C^1} $ remain bounded 
uniformly in $t$.

\begin{lem}\label{comparacio}
For $t\ge 1$ and $f\in C^r(\TT^l)$ we have
$$
\big|\, \| D^s S_t[f] \|_{t^{-1}} -  |D^s f |_{C^0} \,\big| \le 2 \kappa_1  |f |_{C^r} t^{-1}, 
\qquad 0 \le s\le r-1.
$$
\end{lem}
\begin{proof}
Since $f\in C^r(\TT^l)$, $ S_t[f]$ is analytic and $\TT^l$ and $D_{t^{-1}}$ are compact, there
exists  $x_0\in \TT^l$ and $z_0\in D_{t^{-1}}$
such that $|D^s f|_{C^0} = |D^s f(x_0)|$ and $
\| D^s S_t[f] \|_{t^{-1}}=|D^s S_t[f](z_0)| $. Assume that 
$ \| D^s S_t[f] \|_{t^{-1}} \ge  |D^s f |_{C^0}$. 
Then applying \eqref{part4} and \eqref{part5} we have
\begin{align*}
0 & \le | D^s S_t[f] (z_0) | -  |D^s f (x_0)| \\
& \le | D^s S_t[f] (z_0) - D^s S_t[f] (\Re z_0)| + |D^s S_t[f] (\Re z_0)- D^s f (\Re z_0)|  \\
& \quad
+ |D^s f (\Re z_0)| - |D^s f (x_0)| \\
&\le \kappa_1   |  f |_{C^{s+1}} t^{-1} + \kappa_1  |   f |_{C^r} t^{-r+s} .
\end{align*}
If $ \| D^s S_t[f] \|_{t^{-1}} <   |D^s f |_{C^0}$ we argue in a symmetric way and we obtain the result.
\end{proof}

Since $K(\torus^l)$ is real
we have that for $t$ large enough, $S_t[K]( D_{ t^{-1}} ) $
is contained in the domain of $F$. Therefore, if $t$ is large enough, we have that 
$\| F\circ S_t [K]\|_{t^{-1}}$ 
and 
$\|F\circ S_t [K]  - S_t [K]\circ T_\omega \|_{t^{-1}}$
remain uniformly bounded.

On the other hand, by \eqref{part1} in Proposition \ref{prop-smoothing} and the 
fact that $K$ satisfies the functional equation \eqref{embed},  we have 
that
\[
\begin{split}
 | F\circ S_t [K]  - S_t [K] \circ T_\omega  |_{C^0} 
& \le  | F\circ S_t [K]  - F\circ K   |_{C^0} +
  | S_t [K] \circ T_\omega  - K \circ T_\omega  |_{C^0} \\
& \le \kappa_1  |K |_{C^r} ( | F |_{C^1}+1) t ^{-r}.
\end{split}
\]

Therefore, using the interpolation inequality \eqref{interpolation}
in Proposition \ref{prop:interpolation} with $\rho_1=t^{-1}$ and $\rho_2= 0$, 
we obtain that
\begin{equation}\label{errorbound} 
\begin{split}
\| F\circ S_t [K]  & - S_t [K] \circ T_\omega \|_{t^{-1} \xi} \\
& \le
 | F\circ S_t [K]  - S_t [K] \circ T_\omega  |_{C^0}^{1 - \xi} 
\| F\circ S_t [K]  - S_t [K] \circ T_\omega \|_{ t^{-1}}^{\xi} \\
& \le C t^{ -r( 1 - \xi)}.
\end{split} 
\end{equation}

Since all the non-degeneracy constants involve the 
first derivatives, by Lemma~\ref{comparacio} 
we can perform the perturbative arguments in Section~\ref{sec:iteration}. 

The  
constants in the non-degeneracy assumptions remain uniformly bounded 
for $S_t [K]$ in a neighborhood of size $t^{-1}$ 
and, {\sl a fortiori}, in a neighborhood of size $t^{-1}\xi$. 

Therefore, we can apply Theorem~\ref{existence} 
with $\rho_0 = t^{-1}\xi$ and $\delta = t^{-1}\xi/12$ 
provided 
that we can find $t\ge 1$ such that 
\[
C (t^{-1} \xi)^{- 4 \nu}t^{ -r( 1 - \xi)} < 1
\]
for some constant $C>0$, which depends on $l$, $\nu$, $\|DS_t[K] \|_{t^{-1}\xi }$,
$\|N\|_{t^{-1}\xi } $, $\|A\|_{t^{-1}\xi } $, $|(\avg(A))^{-1}|,|(\avg(Q))^{-1}|$. 
By Lemma \ref{comparacio}, if $t$ is big enough,  the constant $C$
can  be chosen independently on $t$.

The condition $r> 4\nu$ implies that there exists $\xi$ close to 0 and $t$ sufficiently large such that the
previous inequality holds. Applying Theorem \ref{existence} 
with initial approximation  $K_0= S_t [K]$ we conclude that there
exists an analytic solution $K^\infty_t$  of equation \eqref{embed}
defined on $D_{t^{-1}\xi/2}$ which satisfies 
\[
 \| K_t^\infty   - S_t [K] \|_{t^{-1}\xi/2}  \le C_1 (t^{-1} \xi)^{ -2 \nu}  
t^{-r(1-\xi)} ,
\]
where $C_1$ depends on $l$, $\nu$, $\|DS_t[K] \|_{t^{-1}\xi }$,
$\|N\|_{t^{-1}\xi } $, $\|A\|_{t^{-1}\xi } $, $|(\avg(A))^{-1}|,|(\avg(Q))^{-1}|$. 
As before $C_1 $ can be taken independent on $t$.
{From} (3) in Proposition \ref{prop-smoothing} we have $(\tau \geq t)$
$$
\|S_\tau[K] - S_t[K]\| _{\tau^{-1}\xi/2} \le \|S_\tau[K] - S_t[K]\| _{\tau^{-1}} \le C_2 t^{-r}
$$
with $C_2$ independent on $t$. 

We will apply Theorem \ref{uniqueness} with $K_1 $ and $K_2$ being $K_t^\infty$ and $K_\tau^\infty$
respectively,
with $t, \tau \ge 1$. 
The application of this result requires Condition \eqref{cond-unicitat} 
which in our case reads 
\begin{equation}\label{cotaunicitat-t} 
\tilde C_3 \kappa^2 (\frac{\tau^{-1} \xi/2}{4})^{-2\nu} \|K_t^\infty   - K_\tau^\infty  \|_{\tau^{-1}\xi/2}\le 1.
\end{equation}

The constant $\tilde C_ 3$ depends on $l$, $\nu$, $\| K^\infty_t \|_{\tau^{-1}\xi/2}  \le 
\| K^\infty_t \|_{t^{-1}\xi/2} $, $\|N_t\|_{t^{-1}\xi } $ $\|A_t\|_{t^{-1}\xi }$, $|(\avg(A_t))^{-1}|,|(\avg(Q_t))^{-1}|$,
where $N_t$, $A_t$ and $Q_t$ are the expressions introduced in Definition \ref{ND} 
corresponding to $ K^\infty_t$.
As before $\tilde C_3 $ can be chosen independently 
on $t,\tau \in [1,\infty)$, if $t$ is big enough. We write 
$C_3 = 8^{2\nu} \kappa^2 \tilde C_3$.
\begin{lem}
There exists $t\ge 1$ such that if $\tau \ge t$ there exists $\vp _{t,\tau} \in \TT^l$ such that 
\begin{equation} \label{condicio-shift}
 K^\infty_t \circ T_{\vp _{t,\tau}} = K^\infty_\tau .
\end{equation}
\end{lem}  
\begin{proof}
Using the previous notation we take $t$ 
big enough such that the constants $C_1$, $C_2$ and $C_3$ are independent on $t$ and  such that  
\begin{equation} \label{condicio-boot}
C_3 2^{2\nu} \xi^{-4\nu} (2C_1 t^{4\nu-r(1-\xi)} + C_2 t^{2\nu -r}) < 1.
\end{equation}
We define $t_m= 2^mt $, $m\ge 0$, and we claim that for 
$t_m \le \tau\le 2t_m = t_{m+1} $ there exists $\vp _{t_m,\tau} \in \TT^l$ such that 
$$
 K^\infty_{t_m} \circ T_{\vp _{t_m,\tau}} = K^\infty_\tau .
$$
Indeed, we apply Theorem \ref{uniqueness} with $K_1= K^\infty_{t_m}$ and $K_2 =  K^\infty_\tau$.
We have 
\begin{align*}
\|K_{t_m}^\infty -  K_\tau^\infty  \|_{\tau^{-1}\xi/2}  \le& \|K_{t_m}^\infty -S_{t_m}[K]\|_{ t_m^{-1}\xi/2} 
+ \|S_{t_m}[K] -S_\tau[K] \| _{\tau^{-1}\xi/2} \\
& + \|S_\tau[K] -K_\tau^\infty \|_{\tau^{-1}\xi/2} \\
 \le & C_1(t_m^{-1} \xi)^{-2\nu} t_m^{-r(1-\xi)} + 
C_2 t_m^{-r}+ C_1(\tau^{-1} \xi)^{-2\nu} \tau^{-r(1-\xi)}.
\end{align*} 
Using that $\tau \le 2t_m$,  Condition \eqref{cotaunicitat-t} is implied by
$$
C_3 2^{4\nu}  \xi^{-4\nu} [2C_1 t_m^{4\nu-r(1-\xi)} + C_2 t_m ^{2\nu-r}] <1
$$
which holds true by \eqref{condicio-boot} since $t_m\ge t$. 

If $\tau >t $ there exists $k\ge 0$ such that $t_k \le \tau < t_{k+1}$. 
{From} the claim 
we can  define $\vp_{t,\tau} = \sum_{m=0}^{k-1}\vp_{t_m,t_{m+1}} + \vp_{t_k,\tau}$.
Clearly $\vp_{t,\tau} $ satisfies 
\eqref{condicio-shift}. 
\end{proof}


Now consider $\tau_j\ge t $ going to $\infty$. Since $\vp_{t,\tau_j}\in \TT^l $ there exists a convergent subsequence,
which we denote again $\vp_{t,\tau_j}$, with limit $\vp_\infty\in \TT^l$.

Then
\[
\begin{split}
 | K_{\tau_j}^\infty -  K  |_{C^{0}} 
& \le 
 |  K_{\tau_j}^\infty - S_{\tau_j}  [K]  |_{C^{0}}  +  | S_{\tau_j}  [K] -  K  |_{C^{0}}  \\
& \le C_1 {\tau_j}^{ 2 \nu-r} + \kappa_1  |K |_{C^r} {\tau_j}^{ - r}.
\end{split}
\]
Also, using that $K_{\tau_j}^\infty =K_t ^\infty\circ T_{\vp_{t,\tau_j}} $ we       get
$$
 | K_t^\infty \circ T_{\vp_\infty}-  K  |_{C^{0}} \le 
 | K_t^\infty \circ T_{\vp_\infty}-  K_t ^\infty\circ T_{\vp_{t,\tau_j}} |_{C^{0}} + 
 | K_{\tau_j}^\infty -  K  |_{C^{0}} .
$$
Finally, taking limit as $j$ goes to $\infty$ we get 
$K = K_t^\infty \circ T_{\vp_{\infty}}$ and hence $K$ is analytic.
\end{proof}

\subsection{Nontrivial stable and unstable bundles} 
\label{sec:nontrivial} 

\subsubsection{General comments and classification of bundles}

In this section we describe some examples of whiskered  invariant tori
with non-trivial stable/unstable bundles.
Theorem~\ref{existence} applies to these tori while other methods in 
the literature 
do not seem to apply. 
We note that, for some systems (see \cite{HLlex}), non-trivial 
bundles appear naturally when the systems experience resonances. 
We think that the study of bifurcations of the bundles of 
invariant tori deserves further exploration. 

We are very grateful to Prof. R. Gompf for very enlightening discussions
and, in particular, for constructing Example~\ref{ex:Euler}
and for providing us with a complete classification of rank $2$ 
bundles over the torus, which we hope will be useful for 
future research.

We start from 
a non-trivial bundle $E \xrightarrow{\Pi}  \torus^l$ 
whose fibers are $\real^{d-l}$. 
Such examples are well-known, but we detail a special one in 
Example~\ref{ex:Euler}. 

As it is well-known, when $l = 1$, the only obstruction to 
triviality is the orientation but when $l \ge 2$, there are
other obstructions to triviality. We just mention the 
Euler characteristic or characteristic classes (Whitney-Stiefel or 
Pontryagin for real bundles or Chern classes for complex bundles). 
See \cite{Steenrod51, MilnorS74, Husmoller94}. The following 
construction is very similar to
constructions in \cite[Section 1.4]{GompfS99}. 

We now consider a manifold $\M$ as a bundle given by  
\begin{equation} \label{manifold}
\begin{split}
\M &= \mathcal{E}^s \oplus \mathcal{E}^u \oplus T \torus^l \\
&= \mathcal{E}^s \oplus \mathcal{E}^u \oplus ( \real^l\times \torus^l),
\end{split} 
\end{equation}
where $\mathcal{E}^s = E $, $\mathcal{E}^u = E^*$ --- the notation $E^*$ indicates the 
dual bundle of linear functions on the fibers --- and $\oplus$ is 
the Whitney sum of bundles.  We use the index $s,u$ to 
give an indication of future constructions.

We will also introduce the notation $T \torus^l = \mathcal{E}^c$ so
that we can write 
\begin{equation} \label{manifold2}
\M = \mathcal{E}^s \oplus \mathcal{E}^u \oplus \mathcal{E}^c. 
\end{equation}
We denote the projections associated to each of the bundles
$\mathcal{E}^s, \mathcal{E}^u, \mathcal{E}^c$ 
by $\Pi^s$, $\Pi^u$, $\Pi^c$ respectively. 

The manifold $\M$ is a bundle over $\torus^l$ whose fibers 
are $\real^{d -l} \times \real^{d-l} \times \real^l$. 
We can denote points in $\M$ as $(e^s, e^u, e^c, \th)$, 
where $e^\sigma \in (\Pi^\sigma)^{-1}(\th)$, $\sigma = s, u, c$. 

We also recall that if $E$ is a linear bundle over a manifold $\N$, 
$TE$ can be canonically identified as a bundle over $T\N$ with fibers 
isomorphic to those of $E$. The basic idea is that the tangent 
directions along the fibers of $E$  can be identified with elements of 
the fibers since the space is linear. 

Hence, we will write points in 
$T_{(e^s, e^u, e^c, \th)}\M$ as $(v^s, v^u, v^c, v^t)$
where $v^\sigma \in \mathcal{E}^\sigma_\th$, $\sigma = s,u,c$
and $v^t \in T_\th \torus^l$.  Of course, we have the fact that the 
tangent bundle over the torus is trivial.  

In a coordinate patch which trivializes the bundle, we can 
introduce the form $\alpha^{su}   = \sum_{i = 1}^{d -l} e^u_i d e^s_i$. 
The key observation is that, even if the definition is in a
coordinate patch, a change of coordinates in the patch leaves
the form invariant. This is completely analogous to the 
coordinate construction of the canonical form in a cotangent bundle
\cite{Arnold-MathMethods,GuilleminS77}. 

We also construct the canonical one-form in $\mathcal{E}^c$
by $\alpha^c = \sum_{i = 1}^l e^c_i d\th_i$ and consider the form 
$\alpha = \alpha^{su} + \alpha^c$. 

The form $\Omega = d  \alpha= d (\alpha^{su} + \alpha^c)$ is symplectic
on $\M$. Indeed, it is clearly closed by definition. The fact that it is non-degenerate can 
be seen directly since, in the coordinate patch which trivializes the bundle, 
it has the standard form. As a consequence, $\M$ can be considered as
an exact symplectic manifold.  

We now relate the previous construction to our problem. We consider a
linear bundle isomorphism on $\mathcal{E}^s$ over a rotation
$T_\omega$, i.e. a family of 
invertible linear maps 
$A_\th : \mathcal{E}^s_\th \rightarrow \mathcal{E}^s_{\th +
  \omega}$.  We 
can then form a bundle isomorphism on $\mathcal{E}^s \oplus
\mathcal{E}^u$ over the same rotation which preserves
the form $\alpha^{su}$ by setting 
\[
A^{su}_\th(e^s, e^u) = (A_\th e^s, (A^{-1}_\th)^\top  e^u) .
\]

Then, the mapping  
\[
F(e^s, e^u, e^c, \th) = ( 
A^{su}_\th (e^s, e^u), e^c, \th + \omega) 
\]
is exact symplectic. The embedding $K: \torus^l \rightarrow \M$ 
given by $K(\th) = (0,0,0, \th)$ clearly
satisfies \eqref{embed}. If we compute the non-degeneracy conditions
for this trivial solution, we obtain that $A(\th)= \Id$ and $Q(\th)= \Id$, which is the derivative of 
the frequency
on the center direction. 

The hyperbolicity condition is verified if $$\|A^s\| < \mu_1 <1$$
and $$\|(A^{u})^{-1} \| = \| (A^s)^\top  \|  < \mu_2 < 1.$$ 
This can be arranged by multiplying $A^s$ by a constant 
if necessary. Note that in this case,
we can take $\mu_3$ to be as close to $1$ as desired. 

 Furthermore, if $G$ is analytically
close to $F$ (i.e. $\|F-G\|_{\B} < \eps$, where $\B$ is a suitable complex subset of $\M$) and exact symplectic, then we have 
$$\|G \circ K - K \circ T_\omega\|_{\rho_0} = \|F\circ K -G \circ
K\|_{\rho_0} <\eps$$
so that if $\eps$ is small enough the hypotheses of
Theorem~\ref{existence} are met.

\subsubsection{An explicit example}\label{ex:Euler}

To make the whole construction more concrete, we 
just end with an explicit example  of a non-trivial $\real^2$-bundle   over $\torus^2$
with positive Euler characteristic explained to 
us by Prof. Gompf.
Many more examples can be found in 
\cite{MilnorS74}. Applying the
construction in this section to 
these examples gives us symplectic manifolds and whiskered tori 
with non-trivial stable and unstable bundles. 
\def\sphere{ {\mathbb S}}
We construct a $\real^2$ bundle over $\sphere^2$ with 
non-zero Euler characteristic. If we identify $\real^2$ with 
$\complex$ using the standard identification and 
$\sphere^2$ with the Riemann sphere,  we can construct a 
non-trivial bundle in the semi-sphere, whose 
boundary is the circle $ \sphere^1 \equiv  \{  |z| = 1\}$, by identifying the 
product bundle. We just give a gluing map on the unit sphere bundle, 
and extend it homogeneously. Hence, it suffices to 
give an identification mapping $i$ from 
$\sphere^1 \times \sphere^1$ to itself. The first factor 
is the boundary of the disk and the other factor is the unit 
bundle. We take $i(z, w) = (1/z, z^n w)$.  Using partitions of 
identity, one can extend this bundle on a disk to a bundle of 
the torus. 

\section{Finite-dimensional flows}

This section is devoted to the application of our method to find invariant tori
for symplectic (locally Hamiltonian)  vector-fields. 
Although we have already presented a result --- in a rather abstract way --- on existence of invariant 
tori for vector-fields in Theorem \ref{flots}, 
we now present a direct proof of the results, following 
similar methods as in the case for maps. One motivation for writing this section is 
that the proof leads immediately to algorithms, which may 
be useful for applications involving vector-fields 
rather than maps. It may be of interest for  practitioners
to have algorithms for flows. 

The proof for flows can also serve
as a starting point for a proof for PDE's. 
We also note that the methods developed 
here apply to some ill-posed partial differential equations, which 
do not admit time-$1$ maps. Of course, 
the adaptation of the strategy of proof to PDE's involves
several technical considerations (the generators of the evolutions  are
unbounded operators rather than differentiable ones).
We postpone these considerations on PDEs to a forthcoming paper (see \cite{LlaveS07}).

We will study first the case of locally Hamiltonian flows. 
The case of globally Hamiltonian flows will be discussed  in
Section \ref{secHamilFD}. 

\subsection{Some preliminaries on symplectic geometry}

In this section we recall several well-known facts on symplectic geometry of vector-fields. 

We will consider vector-fields on an exact symplectic manifold $\mathcal{M}$ with symplectic structure $\Omega=d\alpha$. We have the following definitions.

\begin{defi}\label{exactVF}
We say that a vector-field $X$ on $\M$ is symplectic when 
\begin{equation*}
\mathcal{L}_X \Omega=0, 
\end{equation*}
where $\mathcal{L}_X$ stands for the Lie derivative with respect to $X$. 
\end{defi}

\begin{defi}
We say that a vector-field $X$ is exact symplectic when there exists a smooth function $W$ on $\M$ 
such that  
\begin{equation*}
\mathcal{L}_X \alpha =dW.
\end{equation*}
\end{defi}
An easy calculation  checks that exact symplectic vector-fields are symplectic: 
\begin{equation*}
\mathcal{L}_X \Omega= \mathcal{L}_X d\alpha = d(\mathcal{L}_X \alpha)=d(dW) =0.
\end{equation*}
However, the converse is not true. A well-known example is the following: consider the manifold $\mathcal{M}=\mathbb{T} \times \mathbb{R}$. We denote the corresponding coordinates $(q,p)$ and we 
set $\alpha=pdq$ and $\Omega=dp \wedge dq$. Consider now the vector-field $X=\partial_p$. It is symplectic but not exact symplectic. 

Using Cartan's formula and the fact that $d\Omega=0$, we obtain that $X$ is symplectic if and only if 
\begin{equation}\label{local}
0=d i_X \Omega +  i_X d\Omega = d  i_X \Omega.
\end{equation}
This means, by Poincar\'e lemma,  that locally we can write 
\begin{equation*}
i_X \Omega=dH.
\end{equation*}
Of course, \eqref{local} does not imply that $H$ is a global function since in general it is 
only locally defined.

As a matter of fact, $H$ will be a global function if and only if the vector-field $X$ is exact symplectic. 
Indeed, since 
\begin{equation*}
dW=\mathcal{L}_X \alpha =d(i_X \alpha)+i_X \Omega,
\end{equation*}
we see that, if $X$ is exact symplectic, we can take $H=W-i_X \alpha$. 

The above discussion shows that the only difference between symplectic
and exact symplectic is the (de Rham) cohomology class of $i_X \Omega$. We
introduce the following definition.
\begin{defi}\label{complete}
Let $K$ be an embedding from $\mathbb{T}^l$ into $\mathcal{M}$. We say
that a family of vector-fields
$X_{\lambda}$ with $\lambda \in \RR^l$ spans the cohomology of $K(\TT^l)$ at $\lambda = \overline \lambda $ if the map
\begin{equation*}
\begin{array}{ccc}
\mathbb{R}^l & 
\longrightarrow & H^1(\mathbb{T}^l)\\
v & \mapsto & \frac{d}{d\lambda} [K^*i_{X_{\lambda}} \Omega]_{\mid \lambda = \overline \lambda} v
\end{array} 
\end{equation*} 
is an isomorphism. Here we denote $H^1(\mathbb{T}^l)$ the first de Rham cohomology group of $\mathbb{T}^l$,
which is well known to be $\RR^l $ (see \cite{hat02}).
\end{defi}

In $\TT^l\times \RR^l$ with the standard symplectic form, we have that, denoting by
$p_i$ the coordinates along $\mathbb{R}^l$,  the family
\begin{equation*}
X_\lambda=\sum_{i=1}^l \lambda_i \partial_{p_i}
\end{equation*} 
spans the cohomology at every $\lambda$. Of course, in this case, the cohomology classes have a
very simple characterization as the averages along each of the elementary cycles of $\mathbb{T}^l$.

\subsection{Setting of the equations}

The result for flows is based on the study of the equation 
\begin{equation}\label{flowNH}
\partial_{\omega} K(\th)=X (K (\th)),
\end{equation}
for $K: D_\rho \supset \TT^l \to \M$, where 
the operator $\partial_{\omega}$ (derivative in the direction
  $\omega$) is defined by 
$$
\partial_{\omega}K(\th)=\sum_{i=1}^l \omega_i
  \frac{\partial K(\th)}{\partial \th_i}
  $$
and 
the vector-field $X:\mathcal{M} \rightarrow T\mathcal{M}$ is symplectic and real analytic.     

Let $S_t$ be the flow of $X$.
If $K:\TT^l\rightarrow \M$ is a solution of \eqref{flowNH} then 
\begin{equation}\label{flowproperty}
S_t(K(\th) ) = K(\th+\omega t), \qquad \th\in \TT^l, \;\; t\in \RR,
\end{equation}
and therefore the range of $K$ is invariant by $S_t$.
Indeed, considering $\th\in \TT^l$ fixed, both sides of 
 \eqref{flowproperty} satisfy the same Cauchy problem. 
 
We first deal with a family of vector-fields $X_\lambda$ and 
we prove a 
version of the translated torus theorem.
For an exact symplectic vector-field we will embed it 
into a family, then prove a vanishing lemma
and finally prove the existence of an invariant torus. 
For families $X_\lambda$, the equation under consideration is 
\begin{equation}\label{flowNHtrans}
\partial_{\omega} K(\th)=X_{\lambda}(K(\th)),
\end{equation}
where $\lambda \in \mathbb{R}^l$,  the dependence of $X_{\lambda}$
in $\lambda$ is  at least $C^1$ and we assume that the
vector-field $X_{\lambda}$ spans the cohomology 
of $K_0(\TT^l)$ in the sense of Definition
\ref{complete}, where $K_0$ is an approximate solution of 
\eqref{flowNHtrans}.  

A very important role will be played by the linearized equation
\begin{equation}\label{varNH}
\frac{d\Delta}{dt}=A_{\lambda}(\th+ \omega t)\Delta,
\end{equation}
where $A_{\lambda}(\th) = DX_{\lambda} (K(\th))$.
Since $A_{\lambda}$ is a bounded operator, equation \eqref{varNH} admits an evolution operator, which is defined for all $t\in \RR$, and we will denote it $U_{\th}(t)$. 
It is characterized by
\begin{equation} \label{equacioperaU}
\frac{d}{dt}U_{\th}(t)=A_{\lambda}(\th+\omega t)U_{\th}(t), \qquad U_{\th}(0)=\Id.
\end{equation} 

\subsection{Non-degeneracy conditions}

To establish the existence of tori, we will require
non-degeneracy conditions  similar to the ones considered
in the case of maps: namely, a spectral condition and a twist condition. 

\newtheorem{condition}[thm]{Condition}

\begin{condition}\label{ND1VFNH}(Spectral non-degeneracy condition)
Given $\lambda \in \RR^l$ and an embedding $K: D_\rho\supset \TT^l \rightarrow \M$ 
we say that the pair $(\lambda,K)$ is
hyperbolic
 non-degenerate for the functional equation \eqref{flowNHtrans} if 
there is an analytic splitting 
\begin{equation*}
T_{K(\th)} \mathcal{M}=\mathcal{E}_{K(\th)}^s\oplus
\mathcal{E}_{K(\th)}^c  \oplus \mathcal{E}_{K(\th)}^u 
\end{equation*} 
invariant under the linearized equation \eqref{varNH} in the sense that 
\begin{equation*}
U_{\th}(t)\mathcal{E}^{s,c,u}_{K(\th)}=\mathcal{E}^{s,c,u}_{K(\th+\omega t)}. 
\end{equation*}
Moreover the center subspace $\mathcal{E}^c_{K(\th)}$ has dimension $2l$.
We denote $\Pi_{K(\th)}^s$, $\Pi_{K(\th)}^c$ and
$\Pi_{K(\th)}^u$ the projections associated to this splitting and we denote
\begin{eqnarray*}
U^{s,c,u}_{\th}(t)=U_{\th}(t)|_{\mathcal{E}^{s,c,u}_{K(\th)}}.
\end{eqnarray*}

Furthermore, we assume that 
there exist $\beta_1 ,\,\beta_2 ,\, \beta_3>0$ and $C_h>0$ independent of $\th$ satisfying $\beta_3<\beta_1  $, $\beta_3<\beta_2 $ and such that the splitting is characterized by the following rate conditions:
\begin{eqnarray}
\|U^s_{\th}(t)U^s_{\th}(\tau)^{-1}\|_{\rho} &\leq &C_h e^{-\beta_1 (t-\tau)}, \qquad t \geq \tau\ge 0,\nonumber \\
\|U^u_{\th}(t)U^u_{\th}(\tau)^{-1} \|_{\rho} &\leq& C_h e^{\beta_2 (t-\tau)},  \qquad  t\leq \tau \le 0,\label{cotes-scu}\\
\|U^c_{\th}(t)U^c_{\th}(\tau)^{-1}\|_{\rho} &\leq &C_h e^{\beta_3 |t-\tau|},  \qquad t,\tau \in \mathbb{R}\nonumber . 
\end{eqnarray}
\end{condition}
\begin{remark}
As in the case if maps, if we have an approximately invariant splitting and
\begin{eqnarray*}
\|U^s_{\th}(t)U^s_{\th}(\tau)^{-1}\|_{\rho} &\leq &  e^{-\tilde \beta_1 (t-\tau)}, \qquad T/2\le t - \tau\le T,\\
\|U^u_{\th}(t)U^u_{\th}(\tau)^{-1} \|_{\rho} &\leq &  e^{\tilde \beta_2 (t-\tau)},  \qquad  T/2\le \tau-t  \le T,\label{cotes-scu-2}\\
\|U^c_{\th}(t)U^c_{\th}(\tau)^{-1}\|_{\rho} &\leq & e^{\tilde \beta_3 |t-\tau|},  \qquad T/2\le |t-\tau| \le T, 
\end{eqnarray*}
for some $T$ large enough, then there exists a true invariant splitting,  
close to the approximately invariant one, and the bounds \eqref{cotes-scu} 
with respect to this new splitting hold.
This can be checked by using the time $T$ map.
\end{remark}
\begin{remark}
The previous non-degeneracy condition just expresses that we can
associate semi-groups in positive and negative times to the operator
$A_{\lambda}(\th+\omega t)$. More precisely, since the systems under consideration are non-autonomous, we should write 
\begin{equation*}
\left \{
\begin{array}{l}
\frac{d V}{dt}=  A_{\lambda}(\tilde{\th})V,\\
\frac{d \tilde{\th}}{dt} = \omega, \qquad \qquad  \tilde{\th}(0) =\th.
\end{array}\right .
\end{equation*} 
Note that if the systems were autonomous, the exponential bounds would
follow from the spectral properties of $A_{\lambda}$. 
\end{remark}

The linear operators $U^{s,c,u}_{\th}(t)$ enjoy the following co-cycle
property. 
\begin{lemma}\label{propSG}
For all $\th$ and $\omega$ and all times $t,\tau$ we have 
\begin{equation*}
U^{s,c,u}_{\th}(t+\tau)=U^{s,c,u}_{\th+\omega t}(\tau)U^{s,c,u}_{\th}(t), \qquad t,\tau\in \RR.
\end{equation*} 
\end{lemma} 
\begin{proof} It follows from  the classical argument of uniqueness for
Cauchy ODE problems. Dropping the
indexes $s$, $c$ and $u$, for $\th $, $\omega$  and $t$ fixed, we define the functions 
\begin{equation*}
\psi_{1,t}(\tau)=U_{\th}(t+\tau) \psi_0,\qquad \psi_{2,t}(\tau)=U_{\th+\omega t}(\tau)U_{\th}(t) \psi_0
\end{equation*}
for an arbitrary $\psi_0$. Since $U_{\th}(0)$ is the identity operator,
these two functions satisfy the same Cauchy problem and hence are equal. 
\end{proof}
 
\begin{condition}\label{ND2VFNH} (Twist condition)
Let $ A_{\lambda}(\th)=DX_{\lambda}(K(\th))$
and 
$$N(\th)=  [DK(\th)^\top DK(\th)]^{-1}.$$ 
We say that the pair $(\lambda, K)$ satisfies the twist condition if   
the
average on $\torus^l$ of the matrix   
\begin{equation*}
S_\lambda(\th)=N(\th)DK(\th)^\top [\partial_{\omega}( J(K)^{-1} DK \,N)-A_{\lambda} J(K)^{-1}DK\,N](\th).  
\end{equation*}
is non-singular. 
\end{condition}  

If a pair $(\lambda, K)$ with $K: D_\rho \supset \TT^l \to \M $ satisfy both Conditions 
\ref{ND1VFNH} and \ref{ND2VFNH} we write 
$(\lambda, K)\in ND(\rho)$. If $X$ does not depend on $\lambda $ we simply write
$ K \in ND(\rho)$.

We note that Conditions~\ref{ND1VFNH} and~\ref{ND2VFNH} hold in 
open sets of $K$. The fact that Condition~\ref{ND2VFNH} holds for 
an open set (in the $C^1$ topology) is obvious since it is the non-degeracy of
a matrix that is just an explicit algebraic expression involving derivatives. 
The fact that Condition~\ref{ND1VFNH} is stable under perturbations 
will be the content of Section~\ref{change-nondeg}.
\subsection{Statement of the results}

The first result below provides an existence result in the case of a family of symplectic vector-fields.
{From} a sufficiently good approximate torus for a vector-field in the family it provides an invariant torus 
for a translated (with respect to the parameter) vector-field in the family.

\begin{thm}\label{existencetranslatedNH}

Let $\omega\in D_h(\kappa, \nu)$ for some $\kappa >0$ and $\nu\ge l-1$. Assume the following hypotheses
\begin{enumerate}
\item The vector-fields $X_{\lambda}$ are symplectic for every $\lambda \in \RR^l$.
\item The family $X_{\lambda}$ spans the cohomology of $K_0(\TT^l)$ at $\lambda=\lambda_0$ in the sense of Definition \ref{complete}.
\item The pair $(\lambda_0,K_0)$ satisfies the non-degeneracy Conditions
  \ref{ND1VFNH} and \ref{ND2VFNH}.   
\item The vector-fields $X_{\lambda}$ are real analytic and they can be extended holomorphically to a complex neighborhood of the image under $K_0$ of $D_{\rho_0}$: 
\begin{equation*}
B_r=\left \{ z \in \complex^{2d} | \; \exists \th \in \{ |\Im \th| < \rho_0 \} \; s.t. \; 
|z-K_0(\th)| <r \right \}, 
\end{equation*}   
for some $r>0$, and are $C^1$ with respect to $\lambda$.
\end{enumerate}  
Define the error $E_0$ by
\begin{equation*}
E_0(\th)=\partial_{\omega} K_0 (\th)-X_{\lambda_0} (K_0(\th)).  
\end{equation*}
Then there exists a constant $C>0$ depending on $l$, $\nu$,
 $|X_{\lambda}|_{C^2(B_r)}$, $\|DK_0\|_{\rho_0}$,
$\|N_0\|_{\rho_0}$, $\|\frac{\partial X_{\lambda}(K)}{\partial
  \lambda}\|_{\rho_0}$, 
$\|S_0\|_{\rho_0}$, $|(\avg (S_0))^{-1}|$, 
(where $S_0$ and $N_0$ are as
in Condition 
 \ref{ND2VFNH} replacing $\lambda$ by $\lambda_0$ and $K$ by $K_0$) and the norms of the
projections $\|\Pi^{s,c,u}_{K_0(\th)}\|_{\rho_0}$ such that, if $E_0$ satisfies the estimates
\begin{equation*}
C\kappa^4 \delta^{-4\nu} \|E_0\|_{\rho_0} <1
\end{equation*}
and 
\begin{equation*}
C\kappa^2 \delta^{-2\nu} \|E_0\|_{\rho_0} <r,
\end{equation*}
where $0 < \delta \le \min(1,\rho_0/12)$ is fixed, there exists an
embedding $K_{\infty} $ and a
vector $\lambda_{\infty} \in \mathbb{R}^l$ such that $(\la_\infty,K_{\infty})  
\in ND(\rho_{\infty}:=\rho_0-6\delta)$ and 
\begin{equation}\label{transSolHamil}
\partial_{\omega} K_{\infty}(\th)=X_{\lambda_{\infty}}(K_{\infty}(\th)).
\end{equation}
Furthermore, we have the estimates   
\begin{equation*}
\|K_{\infty}-K_0\|_{\rho_{\infty}} \leq C \kappa^{2} \delta^{-2\nu} \|E_0\|_{\rho_0}
\end{equation*} 
and
\begin{equation*}
|\lambda_{\infty}-\lambda_0| < C \kappa^{2} \delta^{-2\nu} \|E_0\|_{\rho_0}. 
\end{equation*} 
\end{thm}

The following theorem deals with the existence of invariant tori for exact symplectic
vector-fields. It follows from the translated torus version Theorem \ref{existencetranslatedNH}
applied to a suitably chosen perturbation of the exact symplectic 
vector-field $X$ and a vanishing theorem  whose proof is
postponed to Section \ref{vanishing-edo}. 
   
\begin{thm}\label{existenceNH}
Let $\omega \in D_h(\kappa,\nu)$ for some $\kappa>0$ and $\nu\ge l-1$. Assume that
\begin{enumerate}
\item The vector-field $X$ is exact symplectic.
\item $K_0$ satisfies the non-degeneracy Conditions
  \ref{ND1VFNH} and \ref{ND2VFNH}.   

\item The vector-field $X$ is real analytic and it can be extended holomorphically to a complex neighborhood of the image under $K_0$ of $D_{\rho_0}$: 
\begin{equation*}
B_r=\left \{ z \in \complex^{2d} | \; \exists \th \in \{ | \Im \th| < \rho_0\}\; s.t.\; |z-K_0(\th)| <r \right \}, 
\end{equation*}   
for some $r>0$.
\end{enumerate}   
Denoting $E_0$ the initial error, there exists a constant $C>0$ depending on $l$, $\nu$,
$|X|_{C^2(B_r)}$, $\|DK_0\|_{\rho_0}$,
$\|N_0\|_{\rho_0}$, 
$\|S_0\|_{\rho_0}$, $|(\avg (S_0))^{-1}|$, 
(where $S_0$ and $N_0$ are as
in Condition \ref{ND2VFNH} replacing $K$ by $K_0$) and the norms of the
projections $\|\Pi^{s,c,u}_{K_0(\th)}\|_{\rho_0}$ such that, if $E_0$ satisfies the estimates
\begin{equation*}
C\kappa^4 \delta^{-4\nu} \|E_0\|_{\rho_0} <1
\end{equation*}
and 
\begin{equation*}
C\kappa^2 \delta^{-2\nu} \|E_0\|_{\rho_0} <r,
\end{equation*}
where $0 < \delta \le \min(1,\rho_0/12)$ is fixed, then there exists an
embedding $K_{\infty} \in ND(\rho_{\infty}:=\rho_0-6\delta)$ such that  
\begin{equation}\label{transSolNH2}
\partial_{\omega} K_{\infty}(\th)=X (K_{\infty}(\th)).
\end{equation}
Furthermore, we have the estimate   
\begin{equation*}
\|K_{\infty}-K_0\|_{\rho_{\infty}} \leq C \kappa^{2} \delta^{-2\nu} \|E_0\|_{\rho_0}. 
\end{equation*}  
\end{thm}

\begin{remark}  
One could also formulate a local uniqueness result in the case of vector-fields. This can be done by a reduction to a time-one map (see \cite{douady82}).
\end{remark} 

\subsection{Linearized equation}

In this context we define the operator
$$
\G_\omega(\lambda, K ) = 
\partial _\omega K - X_\lambda \circ K
$$
and we want to solve the equation $\G_\omega(\lambda, K ) = 0$. 
As in the case of maps this will be done through  a KAM iterative procedure, starting with 
$(\lambda_0, K_0 )$ such that $E=\G_\omega(\lambda_0, K_0 ) $ is sufficiently small. Therefore we 
are lead to consider the linearized equation
\begin{equation}\label{totalVFNH}
\partial_{\omega} \Delta(\th)-A_{\lambda}(\th)\Delta(\th) 
- \frac{\partial X_{\lambda}(K(\th))}{\partial \lambda}\Lambda=-E(\th), 
\end{equation}
where $A_{\lambda}(\th)=DX_{\lambda}(K(\th))$. 

Let $\xi:\TT^l\rightarrow \M$ be a function. {From} the spectral non-degeneracy condition 
we have 
\begin{equation} \label{invarianciaU}
\Pi_{K(\th + \omega t)} U_\th  (t) \xi (\th) = U_\th  (t) \Pi_{K(\th )} \xi (\th), 
\end{equation}
where $\Pi $ stands for any of the projections $\Pi^s$, $\Pi^c$ 
and $\Pi^u$. Differentiating with respect to $t$ both sides of 
\eqref{invarianciaU} and using \eqref{equacioperaU}
we obtain 
\begin{align*} 
\frac{d\,}{d\th}
[\Pi_{K(\th + \omega t)}]\omega U_\th  (t) \xi (\th) 
+ \Pi_{K(\th + \omega t)} & A_\lambda(\th + \omega t) U_\th  (t) \xi (\th)  \\
&= A_\lambda(\th + \omega t) U_\th  (t) \Pi_{K(\th )} \xi (\th).
\end{align*}
Evaluating this expression at $t=0$ and using the definition of $\partial_\omega$ we have
\begin{align*} 
\partial_\omega
[\Pi_{K(\th )} \xi (\th) ] - \Pi_{K(\th )}\partial_\omega \xi (\th) 
+ \Pi_{K(\th )} A_\lambda(\th )  \xi (\th)  
= A_\lambda(\th )  \Pi_{K(\th )} \xi (\th)
\end{align*}
which implies
\begin{equation} \label{bonaprojeccio}
\Pi_{K(\th )} [\partial_\omega - A_\lambda(\th ) ]\xi (\th) = 
[\partial_\omega  - A_\lambda(\th )]\Pi_{K(\th )} \xi (\th).
\end{equation}

\subsubsection{Linearized equation on the center subspace}\label{secCenterVF}

We first project  equation \eqref{totalVFNH} on the center subspace. Using 
\eqref{bonaprojeccio} we immediately obtain
\begin{equation}\label{centerVFNH}
\partial_{\omega} \Delta^c(\th)-A_{\lambda}(\th)\Delta^c(\th)
-\Pi^c_{K(\th)}\frac{\partial X_{\lambda}(K(\th))}{\partial \lambda}\Lambda=-E^c(\th),
\end{equation}
where $\Delta^c(\th) = \Pi^c_{K(\th)}\Delta(\th)$ and $E^c(\th) = \Pi^c_{K(\th)}E(\th)$.

\subsubsection{Small divisors equations and isotropic character of the
torus}

The following result, which is completely analogous to Proposition 
\ref{sdrussman}, deals with the resolution of small divisors
equations along characteristics (see
\cite{Russmann76}, \cite{Russmann76a},
\cite{Russmann75}, \cite{Llave01c}).  
\begin{pro}\label{russVF}
Assume that $\omega\in D_h(\kappa,\nu)$ with 
$\kappa>0$ and $\nu \ge l-1$. Let $h:D_\rho \supset\TT^l \rightarrow \M$ be a real analytic function with zero
average. Then, for any  $0<\delta <\rho$ there exists a unique analytic solution 
$v:D_{\rho-\delta} \supset\TT^l \rightarrow \M$ 
of the linear equation 
\begin{equation*}
\sum_{j=1}^l \omega_j \frac{\partial v}{ \partial \th_j}=h
\end{equation*} 
having zero
average.
Moreover, if $h\in \mathcal{A}_{\rho}$ then $v$ satisfies the
following estimate 
\begin{equation*}
\|v\|_{\rho-\delta} \leq C \kappa\delta^{-\nu} \|h\|_{\rho}, \qquad 0<\delta <\rho.
\end{equation*}
The constant $C$ depends on $\nu$ and the dimension of the torus $l$.
\end{pro}

The following result provides the approximate isotropic character of the
torus. This proposition is similar to the one in \cite{LGJV05} and we
do not reproduce its proof here.  We note that it also follows
by taking time-$1$ maps 
from the corresponding result for maps, which we have
established in Section~\ref{sec:isotropic}. 
\begin{pro}\label{isotVFNH}
Let $K: D_\rho\supset\mathbb{T}^l \rightarrow \mathcal{M}$, $\rho>0$, be a real analytic mapping.
Define the error
\begin{equation*}
E(\th):= \partial_{\omega} K(\th)-X_{\lambda}(K(\th)).
\end{equation*} 
Let $L(\th)=DK(\th)^\top J(K(\th)) DK(\th)$. There exists
  a constant $C$ depending on $l$, $\nu$ and $\|DK\|_{\rho}$ such that 
\begin{equation*}
\|L\|_{\rho -2\delta}   \leq C \kappa \delta^{-(\nu+1)} \|E\|_{\rho}, \qquad 0<\delta <\rho/2. 
\end{equation*}
\end{pro}

Once again, we use
a normalization argument which allows us to write equation \eqref{centerVFNH}
in a suitable form. To do so, we need a result which allows to approximate the center subspace 
with the range of the $2d \times 2l$-matrix  

\begin{equation} \label{defmtilde}
\tilde{M}(\th)=[DK(\th),\; J(K(\th))^{-1}DK(\th)N(\th)],
\end{equation}
where $N(\th)$ is the normalization $l \times l$-matrix given by
$N(\th)=[DK(\th)^\top DK(\th)]^{-1}$,
as in Proposition \ref{prop:distance}. One can prove the following result. 
\begin{pro}\label{prop:distanceVF}
Denote by $\Gamma_{K(\th)}$ the 
range of $\tilde{M}(\th)$ and by 
$\Pi^\Gamma_{K(\th)}$ the projection onto
$\Gamma_{K(\th)}$ according to the splitting 
$\E^s_{K(\th)} \oplus \Gamma_{K(\th)} \oplus \E^u_{K(\th)}$.

Then there exists a constant $C>0$ such that if  
$$
\delta^{-1} \| E\|_\rho \leq C
$$
then we have the estimates (here dist$_\rho$ has to be understood as the distance 
of subspaces
in the  Grassmanian sense)

\begin{equation} \label{eq:distanceboundVF} 
\begin{split}
& \dist_{\rho-2\delta}( \Gamma_{K(\th)}, \E^c_{K(\th)} ) \le C\delta^{-1} \| E \|_\rho,  \\
& \| \Pi_{K(\th)}^c - \Pi_{K(\th)}^\Gamma \|_{\rho - 2 \delta}
\le  C\delta^{-1} \| E \|_\rho 
\end{split}
\end{equation}
for every $\delta \in (0,\rho/2)$ and where $C$, as usual, depends on the non-degeneracy constants of 
the problem. 
\end{pro} 

The proof of the previous proposition follows the same lines as the one of Proposition \ref{prop:distance}. We refer the reader to Corollary \ref{cor2:iterNH} where we construct exact invariant splittings from approximate ones.

We introduce the change of function  $\Delta^c=\tilde{M} \xi+\hat e \xi$, where $\xi:
 \mathbb{T}^l \rightarrow T \mathcal{M}$, with $\xi(\th) \in T_{K(\th)}\mathcal{M}$ and $\hat e =\Pi^c_{K(\th)}-\Pi^\Gamma_{K(\th)}$. We then get 
\begin{equation}\label{tempcVFNH}
[\partial_{\omega}\tilde{M}(\th)-A_{\lambda}(\th)\tilde{M}(\th)]\xi(\th)+\tilde{M}(\th)\partial_{\omega}
\xi(\th)-
\Pi ^c_{K(\th)}\frac{\partial X_{\lambda}(K(\th))}{\partial \lambda}\Lambda=-E^c(\th),
\end{equation}
where we have dropped the terms depending on $\hat e \xi $, which are quadratic in the error. 
As in the case of maps, the matrix $\tilde{M}(\th)$ is not invertible but the matrix 
$\tilde{M}(\th)^\top  J(K(\th))\tilde{M}(\th)$ is. Multiplying equation (\ref{tempcVFNH}) by
$\tilde{M}(\th)^\top  J(K(\th))$ and then by $(\tilde{M}^\top J(K)\tilde{M})^{-1}$, we get the following
equation
\begin{eqnarray*}
(\tilde{M}(\th)^\top  J(K(\th))\tilde{M}(\th))^{-1}\tilde{M}(\th)^\top J(K(\th)) [\partial_{\omega}\tilde{M}(\th)-A_{\lambda}(\th)\tilde{M}(\th)]\xi(\th)+\partial_{\omega}
\xi(\th)\\=
(\tilde{M}(\th)^\top  J(K(\th))\tilde{M}(\th))^{-1}\tilde{M}(\th)^\top 
 J(K(\th))
[\Pi ^c_{K(\th)} \frac{\partial
  X_{\lambda}(K(\th))}{\partial \lambda}\Lambda-E^c(\th)] . 
\end{eqnarray*} 
We are going to normalize the matrix 
$
\partial_{\omega}\tilde{M}(\th)-A_{\lambda}(\th)\tilde{M}(\th).
$
To avoid some computational technicalities, we perform this
normalization only when $K$ is a solution of 
\eqref{flowNHtrans}. 
We refer the reader to the case of maps
on how to handle the computations in the approximate
case.
\begin{lemma}\label{normalizationNH}
Let $(\lambda,K)$ be a solution of 
\begin{equation} \label{mainlambda}
\partial_{\omega} K(\th)=X_{\lambda}(K(\th))
\end{equation} 
and $\tilde{M}$ be the matrix defined by \eqref{defmtilde}.
Then there exists a $l \times l$-matrix $S_\lambda(\th)$ such that 
\begin{equation}\label{parblocNH}
\partial_{\omega}\tilde{M}(\th)-A_{\lambda}( \th)\tilde{M}(\th)=\tilde{M}(\th)
\begin{pmatrix} 
0_l & S_\lambda(\th)\\ 
0_l & 0_l 
\end{pmatrix}.
\end{equation}
The  matrix $S_\lambda(\th)$ has the form
\begin{align*}
S_\lambda(\th)=N(\th)DK(\th)^\top [\partial_{\omega}( J(K)^{-1} DK \,N)-A_{\lambda} J(K)^{-1}DK\,N](\th).  
\end{align*}
\end{lemma}
\begin{proof} 
Exactly in the same way as in the case of maps, if $K$ is a solution
of \eqref{mainlambda} the columns of $\tilde M$ generate the center subspace.
Since 
$\partial_{\omega} -A_{\lambda}(\th)$ commute with $\Pi^c_{K(\th)}$ 
we have that
\begin{equation} \label{MAM}
\partial_{\omega}\tilde{M}(\th)-A_{\lambda}( \th)\tilde{M}(\th)=\tilde{M}(\th) C(\th)
\end{equation}
for some $2l\times 2l$ matrix $C(\th)$. 
Differentiating equation \eqref{mainlambda} with respect to $\th$
we obtain 
\begin{equation} \label{mainlambdader}
\partial_{\omega} DK(\th)=A_{\lambda}(\th) DK(\th).
\end{equation}
This implies that 
$$
C(\th) = \begin{pmatrix} 
0_l & S_\lambda(\th)\\ 
0_l & R_\lambda(\th) 
\end{pmatrix}.
$$
Identifying blocks in \eqref{MAM}
we end up with 
\begin{equation}\label{formulaidetificacio}
\partial_{\omega}( J(K)^{-1} DK\, N)-A_{\lambda} J(K)^{-1}DK\,N
= 
DK \,S_\lambda + J(K)^{-1} DK\, N \,R_\lambda.  
\end{equation}
Multiplying \eqref{formulaidetificacio} by $DK^\top J(K)$ 
and using the isotropic character of the invariant torus, i.e.
$$
L(\th)=DK(\th)^\top J(K(\th))DK(\th)=0,
$$ 
it follows that 
\begin{align}\label{formulaR}
R_\lambda=&DK^\top J(K) [
\partial_\omega (J(K)^{-1} DK\,N ) -A_{\lambda}J(K)^{-1} DK \,N ].
\end{align}
Expanding $\partial_{\omega}( J(K)^{-1} DK\, N)$ and using equation \eqref{mainlambdader}, we get
$$
\partial_{\omega}( J(K)^{-1} DK N)= \partial_{\omega} (J(K)^{-1}) DK\, N 
+ J(K)^{-1}A_{\lambda} DK\, N + J(K)^{-1} DK \partial_\omega N. 
$$
By differentiation of $N N^{-1}=\Id$, using \eqref{mainlambdader} we easily obtain
\begin{equation*}
\partial_\omega N= -N DK^\top  [ A_\lambda^\top +A_\lambda] DK N.
\end{equation*} 
Also $\partial_{\omega}( J(K)^{-1}) =  
-J(K)^{-1} DJ(K) A_\la DK \,J(K)^{-1} $.

Moreover the symplectic character of the vector-fields $X_\la$, i.e. 
$\mathcal L_{X_\la} \Omega= 0$ can be expressed by (recalling the definition of the Lie derivative)  
\begin{equation}\label{symplecticpermatrius}
\frac{d}{dt} [D\Phi^\top _t J(\Phi _t) D\Phi _t]_{\mid t=0}=0,
\end{equation} 
where $\Phi _t $ is the flow solution of $X_\la$ and \eqref{symplecticpermatrius} implies  
$$
A_\lambda^\top J(K) + J(K) A_\lambda + DJ(K) X(K) = 0.
$$
Using the previous calculations we obtain that the right-hand side of 
\eqref{formulaR} vanishes, i.e. $R_\la=0$. 

Now multiplying \eqref{formulaidetificacio} by $N\, DK^\top$ and using the definition of $N$ 
we have 
\begin{eqnarray} \label{expressS}
S_\lambda(\th)=N(\th)DK(\th)^\top [\partial_{\omega}( J(K)^{-1} DK \,N)-A_{\lambda} J(K)^{-1}DK\,N](\th).  
\end{eqnarray}
Using again the previous calculations we can express $S_\la $ as 
$$
S_\lambda=N\,DK^\top J(K)^{-1}[\Id_{2d}-  DK \,N\, DK^\top]
(A_{\lambda} +A_{\lambda}^\top) DK\,N. 
$$
We emphasize that this last formula coincides with \eqref{expressS}
only when $K$ is an exact solution. If $K$ is only an approximate solution
then both expressions are approximately equal.

\end{proof}

We now turn to the case of {\sl approximate} solutions. The procedure
is similar to the one of the case of maps. 

When $K$ is just an approximate solution, 
we define 
$$
(e_1, e_2) = \partial_{\omega}\tilde{M}(\th)-A_{\lambda}( \th)\tilde{M}(\th)-\tilde{M}(\th)
\begin{pmatrix} 
0_l & S_\lambda(\th)\\ 
0_l & 0_l 
\end{pmatrix} .
$$
Some computations, using 
that $\partial_{\omega} DK(\th)-A_{\lambda}(\th) DK(\th) = E(\th)$
and the defintion of $S_\la$
give $e_1= DE$ and $e_2 = O(\|E \|_\rho, \|DE\|_\rho)$.

Next we just state the result
without proof, but we indicate that it is quite analogous 
to the proof in the map case. We first identify -- up to 
a small error -- the center space with the span of 
the tangent and its symplectically conjugate and then compute the matrix
of the derivative in these coordinates.
\begin{lemma}\label{represVFNH}
Assume $\omega\in D_h(\kappa,\nu)$ with $\kappa>0$ and $\nu\ge l-1$
and $\|E\|_{\rho}$ is small enough. Then there exist a matrix
$B(\th)$ and vectors $p_1$ and $p_2$ such that equation \eqref{tempcVFNH} can be written as
\begin{align}\label{eqSDVFNH}
\Big[ \begin{pmatrix} 0_l & S(\th)\\ 0_l & 0_l
\end{pmatrix}   &+  B(\th) \Big]  \xi(\th)+\partial_{\omega}\xi(\th)\nonumber = 
p_1(\th)+p_2(\th) \\
 &-(\tilde{M}(\th)^\top J(K(\th))
\tilde{M}(\th))^{-1}\tilde{M}(\th)^\top J(K(\th))
\Pi^c_{K(\th)}\frac{\partial X_{\lambda}(K(\th))}{\partial \lambda}\Lambda.
\end{align}
Moreover, the following estimates hold:
\begin{equation}\label{estimp1VFNH}
\|p_1\|_{\rho} \leq C \|E\|_{\rho},
\end{equation}
where $C$ just depends on $\|J(K)\|_{\rho}$, $\|N\|_{\rho}$, $\|DK\|_{\rho}$ and
$\|\Pi^c_{K(\th)}\|_{\rho}$. For   $p_2$ and $B$  we have 
\begin{equation}\label{estimp2VFNH}
\|p_2\|_{\rho-2\delta} \leq C \kappa \delta^{-(\nu+1)} \|E\|^2_{\rho}
 \end{equation}
 and
\begin{equation}\label{estimbVFNH}
\|B\|_{\rho-2\delta} \leq C \kappa \delta^{-(\nu+1)} \|E\|_{\rho}
 \end{equation}
for
$\delta \in (0,\rho/2)$, 
where $C$ depends on $l$, $\nu$, $\|N\|_{\rho}$, 
$\|DK\|_{\rho}$, $|X_{\lambda}|_{C^2(B_r)}$, $|J|_{C^1(B_r)}$ and
$\|\Pi^c_{K(\th)}\|_{\rho}$. 
\end{lemma}

\subsubsection{Solution of the reduced equations}

The solution of the reduced equations works in the same way as in
the case of maps. We sketch the procedure in this section and we emphasize on the cohomology obstructions on the equations.   

We write $\xi=(\xi_1,\xi_2)$. We introduce the operator 
\begin{equation}\label{sdApproxNH}
\mathcal{L} \xi= \begin{pmatrix} 0_l & S(\th)\\ 0_l & 0_l
\end{pmatrix}\xi+\partial_{\omega} \xi=p_1(\th)+Q(\th)\Lambda,
\end{equation}
where $p_1=(p_{11},p_{12})$ and $Q=(Q_1,Q_2)$. 
Using this decomposition of $\E^c_{K(\th)}$ we can write equation \eqref{sdApproxNH}
in the form
\begin{eqnarray*}
S(\th)\xi_2(\th)+\partial_\omega \xi_1(\th)=p_{11}(\th)+ Q_1(\th)\Lambda,\\
\partial_\omega \xi_2(\th)=p_{12}(\th)+ Q_2(\th)\Lambda.
\end{eqnarray*}
We furthermore have 
\begin{align*}
Q_1(\th)&= (N^\top DK^\top J(K) ^{-\top})(\th) [(DK \,N \,DK^\top)(\th)  - \Id_{2d} ] J(K(\th))
\Pi^c_{K(\th)}\frac{\partial X_{\lambda}(K(\th))}{\partial \lambda}\Lambda ,\\
Q_2(\th)&= DK(\th)^\top J(K(\th)) 
\Pi^c_{K(\th)}\frac{\partial X_{\lambda}(K(\th))}{\partial \lambda}\Lambda .
\end{align*}
The assumption of spanning the cohomology of $K(\TT^l)$ for the vector-field $X_\lambda$ ensures
that we can choose $\Lambda$ such that the second equation is solvable
in the sense of Proposition \ref{russVF}. Indeed, notice first that the cohomology along the hyperbolic bundle of the form $K^*i_{X_\lambda } \Omega$ is trivial and we have then 
$$ [\frac{d}{d\lambda} K^* i_{X_\lambda} \Omega]=[\frac{d}{d\lambda} K^* i_{\Pi^c X_\lambda} \Omega].$$  
Identifying the cohomology class of a form in $H^1(\mathbb{T}^l)$ to
its integral on the torus $\mathbb{T}^l$ and
using
the fact that the family $X_\lambda$ spans the cohomology of $K(\TT^l)$ at $\lambda$ gives the result
(since we can choose $\Lambda$ such that the average of $p_{12}(\th)+ Q_2(\th)\Lambda$ 
vanishes).

The degree of freedom we get on the average of
$\xi_2$ then allows us to solve the equation on $\xi_1$. Recall that we
use the non-degeneracy conditions as stated in Condition \ref{ND2VFNH}. We obtain the following proposition. 
\begin{pro}\label{approximateSolNH}
Assume $\omega \in D_h(\kappa,\nu)$ with $\kappa >0$  and $ \nu\ge l-1$, and $(\lambda,K)$ is a
non-degenerate pair. If the error
$\|E\|_{\rho}$ is small enough, there exists a mapping $\xi$, analytic on $D_{\rho-2\delta}$ and a vector $\Lambda \in \mathbb{R}^{l}$ solving equation \eqref{sdApproxNH}. 

Moreover there exists a constant $C>0$ depending on $\nu,  l,
\|K\|_{\rho}$, $|(\avg(A))^{-1}|$, $\|N\|_{\rho}$ and $\|\Pi^c_{K(\th)}\|_{\rho}$ such that
\begin{equation*}
\|\xi\|_{\rho-2\delta} <C \kappa^2 \delta^{-2\nu} \|E\|_{\rho} 
\end{equation*}
and
\begin{equation*}
|\Lambda| <C \|E\|_{\rho}. 
\end{equation*}
\end{pro}

\subsection{Linearized equation on the hyperbolic space} \label{secHyperbVF}

We project the linearized equation \eqref{totalVFNH} 
\begin{equation*}
\partial_{\omega} \Delta -A_{\lambda}(\th) \Delta -\frac{\partial
X_{\lambda}(K(\th))}{\partial \lambda}\Lambda=-E(\th)
\end{equation*}
on the stable and unstable subspaces
by using the projections $\Pi_{K(\th)}^s$ and $\Pi_{K(\th)}^u$ respectively.  
We denote 
$\Delta^s(\th)=\Pi_{K(\th)}^s \Delta(\th)$, $\Delta^u(\th)=\Pi_{K(\th)}^u \Delta(\th)$ 
and
$\tilde{E}(\th,\lambda,\Lambda)
=\frac{\partial X_{\lambda}(K(\th))}{\partial \lambda}\Lambda-E(\th)$.

Using the previous notation and \eqref{bonaprojeccio} we obtain
\begin{equation}\label{VFstNH}
\partial_{\omega} \Delta^s(\th) -A_{\lambda}(\th) \Delta^s(\th) 
= \Pi_{K(\th)}^s\tilde{E}(\th,\lambda,\Lambda)  
\end{equation}  
for the stable part and
\begin{equation}\label{VFinstNH}
\partial_{\omega} \Delta^u(\th) -A_{\lambda}(\th) \Delta^u(\th) 
= \Pi_{K(\th)}^u\tilde{E}(\th,\lambda,\Lambda)   
\end{equation}  
for the unstable one. 

The following result provides the solution of the
previous equations.  
\begin{pro}
Given  $\rho>0$, equations \eqref{VFstNH} and 
\eqref{VFinstNH}  admit      unique analytic solutions
$\Delta^s:D_{\rho} \rightarrow \mathcal{E}^s$ and
$\Delta^u:D_{\rho} \rightarrow
\mathcal{E}^u $ respectively,
such that $\Delta^{s,u}(\th) \in  \mathcal{E}^{s,u}_{K(\th)}$. 
Furthermore there exist constants $C^{s,u}$
such that 
\begin{equation}\label{estimHyperbVFNH}
\|\Delta^{s,u}\|_{\rho} \leq C^{s,u} (\|E\|_{\rho} +|\Lambda|),
\end{equation} 
where  $C^{s,u}$ depend on $\beta_1$,  
$\|\Pi^s_{K(\th)}\|_{\rho}$ (resp.  $\beta_2$, $\|\Pi^u_{K(\th)}\|_{\rho}$)
and $C_h$, $\|\frac{\partial X_{\lambda}(K)}{\partial \lambda}\|_{\rho}$. 
\end{pro}  
\begin{proof} The proof is based on the integration of the equation
along the characteristics $\th+ \omega t$ and the use of the  spectral non-degeneracy
Condition \ref{ND1VFNH}. We give the proof for
the stable case, the unstable case being symmetric (for negative times). 

We introduce the function $\tilde{\Delta}(t)=\Delta^s(\th+\omega
t)$. If $\Delta^s$ has to satisfy \eqref{VFstNH} then $\tilde{\Delta}(t)$ has to satisfy the equation 
\begin{equation}\label{VFintermNH}
\frac{d}{dt} \tilde{\Delta} (t) -A_{\lambda}(\th+ \omega t)
\tilde{\Delta}(t)
=\Pi_{K(\th+ \omega t)}^s \tilde{E} (\th+ \omega t,\lambda,\Lambda).
\end{equation}  

We first derive heuristically  a formula \eqref{solucioalest} for 
$\Delta^s$. Then, examining the formula, it will be easy to justify the 
derivation.

Let $U_\th(t) $ be the evolution operator characterized by
\begin{equation} \label{equacioperaUUU}
\frac{d}{dt}U_{\th}(t)=A_{\lambda}(\th+\omega t)U_{\th}(t), \qquad U_{\th}(0)=\Id.
\end{equation}

Using the formula of the variation of parameters we have
\begin{equation} \label{formulaldeltas}
\tilde{\Delta}(t)=U_{\th}(t)\Big[\tilde \Delta(0)+\int_0^tU^{-1}_{\th}(s) \Pi_{K(\th+\omega s)}^s \tilde{E}(\th+\omega s,\lambda,\Lambda)\,ds \Big].
\end{equation}
Using the co-cycle property given by Lemma \ref{propSG} we have
$U^{-1}_{\th}(s)  = U_{\th+\omega s}(-s)$.
 
Since formula \eqref{formulaldeltas} is valid for all $\th\in D_\rho\supset \TT^l$ we can use it substituting $\th$ by 
$\th -\omega t$ and recovering the notation $\Delta ^s$:
\begin{equation*}
\Delta ^s(\th)=U_{\th-\omega t}(t) \Big[\Delta ^s(\th-\omega t)+\int_0^tU _{\th-\omega (t-s)}(-s) \Pi_{K(\th-\omega (t- s))}^s \tilde{E}(\th-\omega(t- s),\lambda,\Lambda)\,ds \Big].
\end{equation*}
We assume that $\Delta^s$, the solution we are looking for, stays in $\E^s$
and it is bounded; then $U_{\th-\omega t}(t) \Delta ^s(\th-\omega t)$
goes to 0 when $t$ goes to $\infty$. 
Using again the co-cycle property we have
$$
U_{\th-\omega t}(t) U_{\th-\omega (t-s) }(-s)  = U_{\th-\omega (t-s)}(t-s).
$$
Then we write
$$
\Delta ^s(\th)=U_{\th-\omega t}(t) \Delta ^s(\th-\omega t)+\int_0^t
U_{\th-\omega (t-s)}(t-s) \Pi_{K(\th-\omega (t- s))}^s \tilde{E}(\th-\omega(t- s),\lambda,\Lambda)\,ds .
$$
Performing the change of variable $\tau = t-s$ and letting $t$ go to $\infty$ we finally obtain 
\begin{equation}\label{solucioalest}
\Delta^s(\th)=\int_0^{ \infty}U_{\th-\omega
\tau}(\tau) \Pi_{K(\th-\omega\tau)}^s \tilde{E}(\th-\omega \tau,\lambda, \Lambda)\,d\tau.
\end{equation}  
Using the spectral non-degeneracy hypothesis the subintegral function is bounded 
by $C_h e^{-\beta_1 \tau}  \|\Pi^s _{K(\th - \omega \tau)}\|_\rho \|\tilde E \|_\rho $.
The exponential bound assures the convergence of the integral and also permits to 
obtain the bound
\begin{equation*}
\|\Delta^s\|_\rho = 
(C_h /\beta_1)  \|\Pi^s _{K(\th - \omega \tau)}\|_\rho \|\tilde E \|_\rho. 
\end{equation*}    
Once we have formula \eqref{solucioalest} we check directly that $\Delta^s$ is indeed a solution
of \eqref{VFstNH}. The absolute convergence of 
\eqref{solucioalest} justifies the exchange of 
limits and rearrangements used in the derivation.

The uniqueness follows from the fact that if we start with any solution $\Delta^s_1$ of \eqref{VFstNH},
doing the previous manipulations we will end up with the same explicit formula  
\eqref{solucioalest}.
\end{proof}

\subsection{Change of non-degeneracy conditions in the iterative step}
\label{change-nondeg}
The next result deals with the measure of the change of the splitting 
when we perturb a linear system in an Euclidean
space $\M$.

Let $A_\la(\th)$ be a family of linear maps from an Euclidean space $\M$ into iteself,
depending on $\th\in D_\rho\supset \TT^l$ and $\la \in \RR^l $ 
and let $U_\th$ be its evolution operator, i.e.
\begin{equation*}
\frac{d}{dt}{U}_{\th}(t)={A}_{\lambda}(\th+\omega t){U}_{\th}(t), 
\qquad {U}_{\th}(0)=\Id.
\end{equation*} 
Assume that $\M$ has an analytic family of splittings
$$
\M = \E^s_\th \oplus \E^c_\th \oplus \E^u_\th
$$
invariant by $U_\th$ in the sense that 
${U}_{\th}(t){\mathcal{E}}^{s,c,u}_{\th}= {\mathcal{E}}^{s,c,u}_{\th+\omega t} $.
Let  $\Pi_{\th}^{s,c,u}$ 
the projections associated to this splitting and
$
{U}^{s,c,u}_{\th}(t)={U}_{\th}(t)|_{{\mathcal{E}}^{s,c,u}_{\th}}
$. 
Assume furthermore 
there exist ${\beta}_1 ,\,  {\beta}_2 ,\,  {\beta}_3>0$ and $ {C}_h>0$ independent of $\th$ satisfying ${\beta}_3<{\beta}_1  $, $ {\beta}_3< {\beta}_2 $ and such that the splitting is characterized by the following rate conditions:
\begin{eqnarray*}
\|{U}^s_{\th}(t)  {U}^s_{\th}(\tau)^{-1}\|_{\rho} \leq {C}_h e^{-{\beta}_1 (t-\tau)},\qquad t \geq \tau\ge 0,\\
\|{U}^u_{\th}(t) {U}^u_{\th}(\tau)^{-1} \|_{\rho} \leq {C}_h e^{{\beta}_2 (t-\tau)},
\qquad t \leq \tau \le 0,\\
\|{U}^c_{\th}(t){U}^c_{\th}(\tau)^{-1} \|_{\rho} \leq {C}_h e^{{\beta}_3 |t-\tau|},
\qquad t,\tau \in \mathbb{R}. 
\end{eqnarray*}

\begin{pro}\label{iterNH}
Assume that $A_\la(\th)$ is a family of linear maps as before. 
Let $\tilde A_\lambda(\th)$ be another family 
such that $\|\tilde A_\la-{A}_\la\|_{\rho}$ is small enough. Let $\tilde{U}_{\th}(t)$ 
denote the evolution operator corresponding to $\tilde A_\lambda$, i.e.
\begin{equation*}
\frac{d}{dt}\tilde{U}_{\th}(t)=\tilde{A}_{\lambda}(\th+\omega t)\tilde{U}_{\th}(t), 
\qquad \tilde{U}_{\th}(0)=\Id.
\end{equation*}

Then there exists a family of analytic splittings
\begin{equation*}
\mathcal{M}=\tilde {\mathcal{E}}^s_{\th}\oplus
\tilde{\mathcal{E}}^c_{\th }\oplus \tilde{\mathcal{E}}^u_{\th}
\end{equation*} 
which is
invariant under the linearized equation
\begin{equation*}
\frac{d}{dt}\Delta =\tilde{A}_{\lambda}(\th+\omega t)\Delta
\end{equation*} 
 in the sense that 
\begin{equation*}
\tilde{U}_{\th}(t)\tilde{\mathcal{E}}^{s,c,u}_{\th}=\tilde{\mathcal{E}}^{s,c,u}_{\th+\omega t}. 
\end{equation*}
We denote $\tilde\Pi_{\th}^{s,c,u}$  the projections associated to this splitting and denote
\begin{eqnarray*}
\tilde{U}^{s,c,u}_{\th}(t)=\tilde{U}_{\th}(t)|_{\tilde{\mathcal{E}}^{s,c,u}_{\th}}.
\end{eqnarray*} 
Then
there exist $\tilde{\beta}_1 ,\, \tilde{\beta}_2 ,\, \tilde{\beta}_3>0$ and $\tilde{C}_h>0$ independent of $\th$ satisfying $\tilde{\beta}_3<\tilde{\beta}_1  $, $\tilde{\beta}_3<\tilde{\beta}_2 $ and such that the splitting is characterized by the following rate conditions:
\begin{eqnarray*}
\|\tilde{U}^s_{\th}(t)  \tilde{U}^s_{\th}(\tau)^{-1}\|_{\rho} \leq \tilde{C}_h e^{-\tilde{\beta}_1 (t-\tau)},\qquad t \geq \tau\ge 0,\\
\|\tilde{U}^u_{\th}(t) \tilde{U}^u_{\th}(\tau)^{-1} \|_{\rho} \leq \tilde{C}_h e^{\tilde{\beta}_2 (t-\tau)},
\qquad t \leq \tau \le 0,\\
\|\tilde{U}^c_{\th}(t)\tilde{U}^c_{\th}(\tau)^{-1} \|_{\rho} \leq \tilde{C}_h e^{\tilde{\beta}_3 |t-\tau|},
\qquad t,\tau \in \mathbb{R}. 
\end{eqnarray*}  
Furthermore the following estimates hold 
\begin{align}
\label{proj1NH} \|\tilde \Pi_{\th}^{s,c,u}-\Pi_{\th}^{s,c,u}\|_{\rho
  } &\leq C\|\tilde{A}_\la-A_\la\|_{\rho},\\
\label{betas}|\tilde{\beta}_i- \beta_i  |  &\leq C  \|\tilde{A}_\la-A_\la\|_{\rho},\qquad i=1,2,3 ,\\
\tilde C_h & = C_h.
\end{align} 
\end{pro} 

\begin{proof} 
We provide the proof of the statements concerning the stable subspace. 
We divide it into several steps. We use the notation of Condition \ref{ND1VFNH}.

{\bf Step 1: Construction of the invariant splitting. } 
We look for the invariant splitting associated to the linearized equation
\begin{equation} \label{lin-eq-1}
\frac{d }{dt}W(t) = \tilde{A}_{\lambda}(\th+\omega t)W(t)
\end{equation} 
focusing on the stable bundle.
We write \eqref{lin-eq-1} as
\begin{equation} \label{lin-eq-2}
\frac{d }{dt}W(t) = {A}_{\lambda}(\th+\omega t)W(t) + B_\lambda(\th+\omega t)W(t)
\end{equation} 
with $B_\lambda = \tilde A_\lambda - A_\lambda $.
Since we are interested in solutions decreasing exponentially at $\infty$,  for $a>0$ we introduce
the space 
$$
C_a=\{f: [0,\infty) \to \CC^{2d} \mid \, f \mbox{ continuous},\, \sup_{t\ge 0} e^{at} |f(t) | < \infty \},
$$
with the norm $|f|_a = \sup_{t\ge 0} e^{at} |f(t) |$. 

Given $\alpha \in (\beta_3,\beta_1) $ 
we look for solutions of \eqref{lin-eq-2} in the space $C_\alpha$. We begin with the following auxiliary result.
\begin{lem}\label{lemauxsplit}
Let $\th \in D_\rho \supset \TT^l$,  
$\xi\in \E^s_{K(\th)} $ and $H\in C_\alpha$ with $\alpha \in (\beta_3,\beta_1) $. 
Consider the equation
\begin{equation} \label{lin-eq-aux}
w' = {A}_{\lambda}(\th+\omega t)w + H(t).
\end{equation} 
Then there exists a unique function $\K(\xi,H) \in C_\alpha$ 
such that 
\begin{itemize}
\item[(i)]  $\K(\xi,H) $ is solution of \eqref{lin-eq-aux}.
\item[(ii)]  $\Pi^s_{\th} \K(\xi,H) (0) = \xi$.
\end{itemize}
Moreover $\K(\xi,H) = \K_1(\xi) + \K_2(H)$, where 
$\K_1:\E^s_{\th} \to C_\alpha$ and 
$\K_2:C_\alpha \to C_\alpha$ are bounded linear operators
and
\begin{align}
|\K_1 | &\le C_h , \label{cotaK1}\\
|\K_2 | &\le C_h \Big( \frac{|\Pi^s|}{\beta_1 -\alpha} +\frac{|\Pi^{cu}|}{\alpha -\beta_3} \Big) ,
\label{cotaK2}
\end{align}
where
$|\Pi ^{s,cu}| =\sup_{\th\in D_\rho} |\Pi^{s,cu} _{\th} |$. 
\end{lem}
\begin{proof}
If $w\in C_\alpha$ is a solution of \eqref{lin-eq-aux} in $[0,\infty)$ and $t,\tau \ge 0$ we have
\begin{equation} \label{lin-eq-3}
w(t) = U_\th (t)U_\th (\tau)^{-1} w(\tau) 
+ \int_\tau^t U_\th (t)U_\th (s)^{-1} H(s) \,ds.
\end{equation} 
Projecting \eqref{lin-eq-3} to the center-unstable subspace and 
using the invariance 
of the splitting $\mathcal{E}^s_{{\th}}\oplus (
\mathcal{E}^c_{{\th}}\oplus \mathcal{E}^u_{{\th}})$
with respect to $U_\th $ 
and writing $\Pi_\th^{cu}$ the projection onto 
$\mathcal{E}^c_{{\th}}\oplus \mathcal{E}^u_{{\th}}$
\begin{equation} \label{lin-eq-4}
\Pi^{cu}_{\th+\omega t} w(t) = U_\th (t)U_\th (\tau)^{-1} \Pi^{cu}_{\th+\omega \tau}w(\tau) 
+ \int_\tau^t  U_\th (t)U_\th (s)^{-1} \Pi^{cu}_{\th+\omega s} H(s) \,ds.
\end{equation} 
If $\tau \ge t$ we have
$$
|U_\th (t)U_\th (\tau)^{-1} \Pi^{cu}_{\th+\omega \tau}w(\tau) |
\le C_h e^{\beta_3(\tau-t)} |\Pi^{cu}_{\th + \omega \tau}| \, e^{-\alpha\tau}
|w| _\alpha
$$
which goes to zero as $\tau $ tends to $\infty$. 
Also, if $s>t$ 
$$
| U_\th (t)U_\th (s)^{-1} \Pi^{cu}_{\th+\omega s} H(s) |
\le C_h e^{\beta_3(s-t)} |\Pi^{cu}_{\th + \omega s}| \,e^{-\alpha s}
|H| _\alpha
$$
guarantees that we can take limit  $\tau \to \infty$ in the integral in 
\eqref{lin-eq-4}. Then we have
$$
\Pi^{cu}_{\th+\omega t} w(t) =  \int_\infty ^t  U_\th (t)U_\th (s)^{-1} \Pi^{cu}_{\th+\omega s} H(s) \,ds.
$$
Using the projection to the stable subspace, we obtain
\begin{align}
w(t)  = & \Pi^{s}_{\th+\omega t}w(t) + \Pi^{cu}_{\th+\omega t}w(t) \nonumber \\
 = & U_\th (t) \Pi^{s}_{\th}w(0) 
+ \int_0^t  U_\th (t)U_\th (s)^{-1} \Pi^{s}_{\th+\omega s} H(s) \,ds \label{lin-eq-5} \\
& 
+  \int_\infty ^t  U_\th (t)U_\th (s)^{-1} \Pi^{cu}_{\th+\omega s} H(s) \,ds.\nonumber 
\end{align} 
Once we have the explicit expression of $w$,
 we easily check  that it actually belongs to $C_\alpha$.
We define $\K_1(\xi)(t) =  U_\th (t) \xi $ and $\K_2(H)(t) $
to be the sum of the two integrals in \eqref{lin-eq-5}. 
A simple calculation gives the bounds 
\eqref{cotaK1} and \eqref{cotaK2}.
\end{proof}
By Lemma \ref{lemauxsplit} the solutions of 
\eqref{lin-eq-2} belonging to $C_\alpha$ satisfy 
$$
w(t) = \K_1(\Pi^{s}_{\th }w(0) ) + \K_2 (B_\lambda(\th + \omega \cdot ) w)(t) .
$$
Note that $B_\lambda(\th + \omega t )$ is bounded in $t$ and moreover 
$|B_\lambda(\th + \omega t )| \le  \gamma$, where $\gamma = \|\tilde A_\la-A_\la\|_\rho$.
We introduce the linear map $\tilde \K_2 : C_\alpha \to C_\alpha$ defined by
$$
\tilde \K_2(w) = \K_2 (B_\lambda(\th + \omega \cdot ) w).
$$
Clearly $|\tilde \K_2(w)| \le |\K_2 | \,| B_\lambda(\th + \omega \cdot )|.
$
With the   above introduced notation, 
given $ \xi \in \E^s_{\th} $, there is a unique solution
$w\in C_\alpha$ such that $\Pi^s_{\th}w(0) = \xi$ which is given by
$$
w = \K_1 (\xi) +\tilde \K_2 (w).
$$
Since $|B_\lambda| \le \gamma < 1$ we can write 
$$
w= (\Id - \tilde \K_2)^{-1} \K_1 (\xi) .
$$
Therefore $\tilde\E^s_{\th} $ is the graph of
$$
\xi \mapsto \tilde M^s(\th) \xi := 
\Pi^{cu}_{\th}  (\Id - \tilde \K_2)^{-1} \K_1 (\xi) (0)
= \Pi^{cu}_{\th}  \sum_{k=1}^\infty \tilde \K_2^k \K_1 (\xi) (0),
$$
where the sum stars with $k=1$ because
$\Pi^{cu}_{\th} \K_1 =0$.
Note that the analyticity in $\th$ is preserved in all the previous 
manipulations, hence $\tilde M_\th $ depends analytically in $\th$.
Since $|\K_2| \le C\gamma$ then $\|\tilde M^s(\th) \|_\rho<C \gamma$.
In a completely analogous way we find  
$\tilde \E^{cu}_{\th} $, and integrating with negative times 
we get $\tilde \E^{ u}_{\th} $ and $\tilde \E^{sc }_{\th} $.
Finally $\tilde \E^{ c}_{\th} = \E^{sc }_{\th} \cap \E^{cu }_{\th} $.

{\bf Step 2. Estimates on the projections.} 
To get the bounds for the projections we 
follow  the same argument as in the case of maps. 
We only give the argument for the stable subspace.
Let $\tilde M^{cu}(\th) $ be the linear map whose graph gives 
$\tilde \E^{cu}_{\th} $.

We write 
\begin{align*}
\Pi^s_{\th } \xi = (\xi^s,0), & 
\qquad \tilde\Pi^s_{  \th}\xi =  (\tilde \xi^s,\tilde M^s(\th) \tilde \xi^s), \\
\Pi^{cu}_{\th} \xi = (0,\xi^{cu} ), & 
\qquad \tilde \Pi^{cu}_{ \th}\xi =  (\tilde M^{cu}(\th) \tilde \xi^{cu}, \tilde \xi^{cu}) ,
\end{align*}
and then 
\begin{align*}
\xi^s &=\tilde \xi^s +
\tilde M^{cu}(\th) \tilde \xi^{cu}, \\
\xi^{cu} &= \tilde M^s(\th) \tilde \xi^s +  \tilde \xi^{cu} .
\end{align*}
Since $\tilde M^s(\th)$ and $\tilde M^{cu}(\th)$ are $O(\gamma) $  we can write
$$
\left( \begin{array}{c}
\tilde \xi^s \\
\tilde \xi^{cu}
\end{array}\right) 
= 
\left( \begin{array}{cc}
\Id & \tilde M^{cu}(\th)  \\
 \tilde M^s(\th) & \Id
\end{array}\right) ^{-1}
\left( \begin{array}{c}
\xi^s \\
\xi^{cu}
\end{array}\right) 
$$
and then 
deduce that 
$$
|(\tilde\Pi^s_{  \th  } - \Pi^s_{  \th  } )\xi | 
\le |(\tilde \xi ^s -  \xi^{s}, \tilde M^s(\th) \tilde \xi^{s}) |
\le  C\gamma .
$$

{\bf Step 3. Estimates on the growth conditions.} 

To get the exponential bounds let
$$
\psi(t) = \tilde U_\th (t) \tilde  U_\th(\tau)^{-1} \psi (\tau)  
$$
with $\psi(\tau) = (\xi, \tilde M(\th +\omega \tau) \xi) \in \tilde \E^s_{ \th +\omega \tau } $.
The function $\psi$ satisfies equation \eqref{lin-eq-2} and hence
$$
|\psi(t)| \le 
|U_\th (t)U_\th (\tau)^{-1} \psi(\tau)| 
+ \int_\tau^t  | U_\th (t)U_\th (s)^{-1} (\tilde A_\lambda- A_\lambda) (\th+\omega s) \psi(s)| \,ds,
$$
for $t \ge \tau$.
Let $\chi$ be the auxiliary function defined by $\chi(t) = e^{\beta_1 t} |\psi(t)|$. 
Using the bounds of Condition \ref{ND1VFNH} we have
$$
\chi(t) \le C_h \chi(\tau) + C_h C \gamma \int_{\tau}^{t} \chi(s) \, ds, \qquad t\ge \tau.
$$ 
By Gronwall's lemma we have $\chi(t) \le C_h \chi(\tau) e^{C_hC\gamma (t-\tau)}$ and hence
$$\psi(t) \le e^{-\beta_1 t} C_h e^{\beta_1 \tau}\psi(\tau) e^{C_hC\gamma (t-\tau)}.$$
We conclude that
$$
|\tilde U_\th (t) \tilde  U_\th(\tau)^{-1} \psi (\tau)  | \le C_h e^{-(\beta_1-C_hC\gamma) (t-\tau)}
|\psi(\tau)|,
\qquad t\ge \tau.
$$
We take $\tilde C_h = C_h$ and $\tilde \beta_1=\beta_1-C_hC\gamma$, which proves \eqref{betas}. 

\end{proof}

The first consequence of Proposition \ref{iterNH}  is that in the iterative step the small
change of $K$ produces a small change in the invariant splitting and in the hyperbolicity constants.
\begin{corollary}\label{cor1:iterNH}

Assume that $(\lambda, K)$ satisfies the hyperbolic non-degeneracy Condition
\ref{ND1VFNH} and that 
$\|K-\tilde{K}\|_{\rho}$ is small enough. If we denote $\tilde{A}_{\lambda}(\th) = DX_{\lambda}(\tilde{K}(\th))$, we can define an evolution operator, denoted $\tilde{U}_{\th}(t)$ such that
\begin{equation*}
\frac{d}{dt}\tilde{U}_{\th}(t)=\tilde{A}_{\lambda}(\th+\omega t)\tilde{U}_{\th}(t), 
\qquad \tilde{U}_{\th}(0)=\Id.
\end{equation*}

Then there exists an analytic splitting for $\tilde{K}$, i.e. 
\begin{equation*}
T_{\tilde{K}(\th)}\mathcal{M}=\mathcal{E}^s_{{\tilde{K}(\th)}}\oplus
\mathcal{E}^c_{{\tilde{K}(\th)}}\oplus \mathcal{E}^u_{{\tilde{K}(\th)}}
\end{equation*} 
which is
invariant under the linearized equation \eqref{varNH} (replacing $K$ by $\tilde{K}$) in the sense that 
\begin{equation*}
\tilde{U}_{\th}(t)\mathcal{E}^{s,c,u}_{\tilde{K}(\th)}=\mathcal{E}^{s,c,u}_{\tilde{K}(\th+\omega t)}. 
\end{equation*}
We denote $\Pi_{\tilde{K}(\th)}^s$, $\Pi_{\tilde{K}(\th)}^c$ and
$\Pi_{\tilde{K}(\th)}^u$ the projections associated to this splitting. Denoting
\begin{eqnarray*}
\tilde{U}^{s,c,u}_{\th}(t)=\tilde{U}_{\th}(t)|_{\mathcal{E}^{s,c,u}_{\tilde{K}(\th)}},
\end{eqnarray*} 
there exist $\tilde{\beta}_1 ,\, \tilde{\beta}_2 ,\, \tilde{\beta}_3>0$ and $\tilde{C}_h>0$ independent of $\th$ satisfying $\tilde{\beta}_3<\tilde{\beta}_1  $, $\tilde{\beta}_3<\tilde{\beta}_2 $ and such that the splitting is characterized by the following rate conditions:
\begin{eqnarray*}
\|\tilde{U}^s_{\th}(t)  \tilde{U}^s_{\th}(\tau)^{-1}\|_{\rho} \leq \tilde{C}_h e^{-\tilde{\beta}_1 (t-\tau)},\qquad t \geq \tau\ge 0,\\
\|\tilde{U}^u_{\th}(t) \tilde{U}^u_{\th}(\tau)^{-1} \|_{\rho} \leq \tilde{C}_h e^{\tilde{\beta}_2 (t-\tau)},
\qquad t \leq \tau \le 0,\\
\|\tilde{U}^c_{\th}(t)\tilde{U}^c_{\th}(\tau)^{-1} \|_{\rho} \leq \tilde{C}_h e^{\tilde{\beta}_3 |t-\tau|},
\qquad t,\tau \in \mathbb{R}. 
\end{eqnarray*}  
Furthermore the following estimates hold 
\begin{align}
\label{cor1:proj1NH} \|\Pi_{{\tilde{K}(\th)}}^{s,c,u}-\Pi_{{K(\th)}}^{s,c,u}\|_{\rho
  } &\leq C\|\tilde K-{K}\|_{\rho},\\
\label{cor1:betas}|\tilde\beta_i - {\beta}_i|  &\leq C  \|\tilde K-{K}\|_{\rho},\qquad i=1,2,3 ,\\
\tilde C_h & = C_h.
\end{align} 
\end{corollary} 
\begin{proof} 
We just take $A_\la(\th) = DX_\la (K(\th)) $, $\tilde A_\la(\th) = DX_\la (\tilde K(\th)) $,
$\E^{s,c,u}_{ K(\th)} = \E^{s,c,u}_{\th} $,
$\E^{s,c,u}_{\tilde K(\th)} = \tilde \E^{s,c,u}_{\th} $, 
$\Pi^{s,c,u}_{ K(\th)} = \Pi^{s,c,u}_{\th} $ and
$\Pi^{s,c,u}_{\tilde K(\th)} = \tilde \Pi^{s,c,u}_{\th} $
in Proposition \ref{iterNH}
and we use that 
$\|\tilde A_\la(\th)-A_\la(\th) |_\rho \le  \|X\|_{C^2} \|\tilde K(\th)- K(\th)\|_\rho$.
\end{proof} 

The second consequence of Proposition \ref{iterNH} is that if we have a sufficiently 
good approximate splitting associated to equation \eqref{equacioperaU}
then there is a true invariant splitting nearby.

\begin{corollary} \label{cor2:iterNH}

Assume that 
$T_{K(\th)}\mathcal{M}= \E^{*s}_{K(\th)}\oplus
\E^{*c}_{{K}(\th)}\oplus  \E^{*u}_{{K}(\th)}
$ is a splitting approximately 
invariant under the linearized equation \eqref{varNH}
with evolution operator $U_\th(t) $, 
in the sense that $A_\la(\th) = DX_\la(K(\th))$ can be represented as
$$
A_\la(\th) =
\begin{pmatrix}
A_\la^{11}(\th)	& A_\la^{12}(\th) &A_\la^{13}(\th) \\
A_\la^{21}(\th) & A_\la^{22}(\th)& A_\la^{23}(\th) \\
A_\la^{31}(\th) & A_\la^{32}(\th) & A_\la^{33}(\th)
\end{pmatrix}
$$
with respect to this splitting with
$\|A_\la^{ij} (\th)\|_\rho \le C\delta^{-1} \|E\|_\rho$ if $i\ne j$. 
Let $\Pi^{*s,c,u}_{K(\th)} $ be the projections associated to to this splitting. 

Let $\tilde U^{s,c,u}_\th $ be the evolution operators 
of $\dot \Delta = A_\la^{11}(\th+\om t) \Delta $, $\dot \Delta = A_\la^{22} (\th+\om t)\Delta $ 
and  $\dot \Delta = A_\la^{33}(\th+\om t) \Delta $ respectively, and assume 
\begin{eqnarray*}
\|\tilde{U}^s_{\th}(t)  \tilde{U}^s_{\th}(\tau)^{-1}\|_{\rho} \leq {C}^*_h e^{-\beta^*_1 (t-\tau)},\qquad t \geq \tau\ge 0,\\
\|\tilde{U}^u_{\th}(t) \tilde{U}^u_{\th}(\tau)^{-1} \|_{\rho} \leq {C}^*_h e^{\beta^*_2 (t-\tau)},
\qquad t \leq \tau \le 0,\\
\|\tilde{U}^c_{\th}(t)\tilde{U}^c_{\th}(\tau)^{-1} \|_{\rho} \leq {C}^*_h e^{\beta^*_3 |t-\tau|},
\qquad t,\tau \in \mathbb{R},
\end{eqnarray*} 
for some $\beta^*_{1,2,3} , C^*_h >0$ such that $\beta^*_3 < \beta^*_1 $, $\beta^*_3 < \beta^*_2 $.
Then there exists an analytic splitting 
$T_{K(\th)}\mathcal{M}=\mathcal{E}^s_{{ {K}(\th)}}\oplus
\mathcal{E}^c_{{K}(\th)}\oplus \mathcal{E}^u_{{{K}(\th)}}
$ invariant under equation \eqref{varNH}. Let 
$\Pi_{{K}(\th)}^{s,c,u}$ be the projections associated to this splitting
and $
{U}^{s,c,u}_{\th}(t)={U}_{\th}(t)|_{\mathcal{E}^{s,c,u}_{K(\th)}}
$. Moreover 
there exist ${\beta}_{1,2,3} >0$ and ${C}_h>0$ independent of $\th$ satisfying ${\beta}_3< {\beta}_1  $, ${\beta}_3<{\beta}_2 $ and such that the splitting is characterized by the following rate conditions:
\begin{eqnarray*}
\|{U}^s_{\th}(t)  {U}^s_{\th}(\tau)^{-1}\|_{\rho} \leq {C}_h e^{-{\beta}_1 (t-\tau)},\qquad t \geq \tau\ge 0,\\
\|{U}^u_{\th}(t) {U}^u_{\th}(\tau)^{-1} \|_{\rho} \leq {C}_h e^{{\beta}_2 (t-\tau)},
\qquad t \leq \tau \le 0,\\
\|{U}^c_{\th}(t){U}^c_{\th}(\tau)^{-1} \|_{\rho} \leq {C}_h e^{{\beta}_3 |t-\tau|},
\qquad t,\tau \in \mathbb{R}
\end{eqnarray*}  
and
\begin{align}
\|\Pi_{{{K}(\th)}}^{*s,c,u}-\Pi_{{K(\th)}}^{s,c,u}\|_{\rho
  } &\leq C\delta^{-1}\|E\|_{\rho},\\
|\beta^*_i - {\beta}_i|  &\leq C\delta^{-1}\|E\|_{\rho},\qquad i=1,2,3 ,\\
C_h^* & = C_h.
\end{align}  
\end{corollary} 
\begin{proof} 
We make the same identifications as in the proof of Corollary \ref{cor1:iterNH}.
Consider the auxiliary linear equation 
\begin{equation} \label{auxlineq}
\dot \Delta(t) = A^*_\la(\th+\om t) \Delta (t)
\end{equation}
with 
$$
A^*_\la(\th) =
\begin{pmatrix}
A_\la^{11}(\th)	& 0 &0 \\
0 & A_\la^{22}(\th)& 0 \\
0 & 0 & A_\la^{33}(\th)
\end{pmatrix}
$$
Clearly $U_\th^*(t)= (\tilde U_\th^s(t),\tilde  U_\th^c(t),\tilde  U_\th^u(t))$
is a solution of \eqref{auxlineq}. By hypothesis $\|A^*_\la(\th) - A_\la(\th)\|_\rho$ is small.
Then the application of Proposition \ref{iterNH}
gives the results.
\end{proof} 

\begin{remark}
We can give an alternative proof to Proposition \ref{iterNH}, parallel to the one for maps. We just sketch it for the stable bundle in the following. Recall that we have the invariance condition for all times $t \geq 0$ 
\begin{equation*}
{U}^s_{\th}(t)\mathcal{E}^s_{K(\th)}= \mathcal{E}^s_{K(\th+\omega t)}.
\end{equation*}
The graph condition then writes for all times $t \geq 0$
$$
{U}^s_{\th}(t)\begin{pmatrix} \Id \\M(\th)\end{pmatrix}\in \mbox{Graph} (M(\th+\omega t)).$$

We now consider the time-one map $U_1={U}^s_{\th}(1)$. The graph condition leads to a functional equation which is solved by a fixed point argument. To propagate the result to any time and get the estimates, one just has to use the co-cycle property as stated in Proposition \ref{propSG}.    
\end{remark}
The other non-degeneracy conditions can be checked in exactly the
same way (as in the previous section) and we do not repeat the arguments. 
\begin{lemma}
If $\|E_{m-1}\|_{\rho_{m-1}}$ is small enough, then
\begin{itemize}
\item If $DK_{m-1}^\top DK_{m-1}$ is invertible with inverse $N_{m-1}$
  then $DK_{m}^\top DK_{m}$ is invertible with inverse $N_m$ and we have    
\begin{equation*}
\|N_m\|_{\rho_{m}} \leq \|N_{m-1}\|_{\rho_{m-1}}+C_{m-1} \kappa^2 \delta_{m-1}^{-(2\nu+1)} \|E_{m-1}\|_{\rho_{m-1}}.
\end{equation*}
\item If $\avg(S_{m-1})$ is non-singular then also $\avg (S_{m})$ is and we have the estimate
 \begin{equation*}
|(\avg(S_m))^{-1}|\leq |(\avg(S_{m-1}))^{-1}|+C'_{m-1} \kappa^2 \delta_{m-1}^{-(2\nu+1)} \|E_{m-1}\|_{\rho_{m-1}}.
\end{equation*}
\end{itemize}
\end{lemma}
The last lemma is devoted to the proof of the cohomology obstruction under the iterative step. 
\begin{lemma}
Assume $\|E_{m-1}\|_{\rho_{m-1}}$ is small enough. If $X_{\lambda_{m-1}}$ spans the cohomology of $K_{m-1}(\TT^l)$ at $\lambda_{m-1}$, then $X_{\lambda_m}$ 
spans the cohomology of $K_m(\TT^l)$ at $\lambda_{m}$.
\end{lemma} 
\begin{proof}
We have by assumption that the map
\begin{equation*}
\frac{d}{d\lambda} [K_{m-1}^*i_{X_{\lambda}} \Omega]_{\mid\lambda=\lambda_{m-1}}: \mathbb{R}^l
\rightarrow H^1(\mathbb{T}^l)
\end{equation*}
is an isomorphism. Thanks to the estimates on $\Delta_m$ and $|\lambda_m-\lambda_{m-1}|$ and the continuity of $X_{\lambda}$ and $DX_{\lambda}$ with respect to $\lambda$, we can write 
\begin{equation*}
\|\frac{d}{d\lambda} [K_{m}^*i_{X_{\lambda}}
\Omega]_{\mid\lambda=\lambda_{m}}-\frac{d}{d\lambda}
[K_{m-1}^*i_{X_{\lambda}}
\Omega]_{\mid\lambda=\lambda_{m-1}}\|_{\rho-\delta} \leq C \kappa
\delta^{-1} \|E_{m-1}\|_{\rho}. 
\end{equation*}
The previous estimate comes from the identification of the cohomology
with the integration over loops of $\mathbb{T}^l$ and the fact the quantity $\frac{d}{d\lambda} K^* i_{X_\lambda} \Omega$ is in matrix notation  
\begin{eqnarray*}
DK(\th)^\top
J(K(\th))\Big(\frac{\partial
X_{\lambda}(K(\th))}{\partial \lambda}\Big).
\end{eqnarray*} 
This shows the invertibility of the map
\begin{equation*}
\frac{d}{d\lambda} [K^*_m i_{X_{\lambda}}\Omega]_{\mid\lambda=\lambda_m}.
\end{equation*}
\end{proof}

\subsection{Vanishing lemma}\label{vanishing-edo}
This section is devoted to the proof of Theorem \ref{existenceNH}. First, recall that the Lie derivative of the 1-form $\alpha$ with
respect to a vector-field $L$ is given by (Cartan formula) 
\begin{equation*}
\mathcal{L}_L \alpha= d i_L \alpha+i_L d\alpha . 
\end{equation*}
The following result is of general interest and is a vanishing lemma.   
\begin{lemma} \label{vanishingflow}
Assume $\omega \in D_h(\kappa,\nu)$ with $\kappa>0$ and $\nu\ge l-1$, and $X_{\lambda}$ is a family 
of real analytic symplectic vector-fields. Let $K: D_\rho \supset \TT^l \to \M$ be a solution of 
\begin{equation}\label{temp}
\partial_{\omega} K= X_{\lambda}\circ K+E,
\end{equation} 
for $|\lambda-\lambda^*|$ small enough. 
Assume furthermore that 
\begin{enumerate}
\item $X_{\lambda^*}$ is exact symplectic and for all $\lambda \in \RR^l$ and $\lambda \neq \lambda^*$, the vector-fileds $X_\lambda$ are symplectic but not exact symplectic. 
\item For all $\lambda \in \RR^l$, $X_{\lambda}$ can be extended holomorphically to a complex neighborhood 
of $K(D_\rho)$.
\item 
The family $X_{\lambda}$ spans the cohomology of $K(\TT^l)$ at $\lambda =\lambda^* $, i.e. 
the map 
\begin{equation} \label{cohomol-map}
\begin{array}{ccc}
\mathbb{R}^l &
\longrightarrow & H^1(\mathbb{T}^l)\\
v & \mapsto & \frac{d}{d\lambda} [K^*i_{X_{\lambda}} \Omega]_{\mid \lambda=\lambda^*} v
\end{array} 
\end{equation}
is an isomorphism.
\end{enumerate}
Then there exists a constant $C$ such that 
$$|\lambda-\lambda^*| \leq C \|E\|_\rho. $$
\end{lemma}

\begin{proof} 
The proof is very similar to the proof of
Lemma~\ref{vanishing}. Indeed, if we consider 
vector-fields as \emph{``infinitesimal''} diffeomorphisms, 
the present proof can be considered as an infinitesimal
version of the proof of Lemma~\ref{vanishing}. We define $\sigma_{i, \hth_i}$ as in 
\eqref{thetahat}, \eqref{sigmai}. 

The proof  consists of computing 
\begin{equation}\label{tocompute}
\int_{K \circ \sigma_{i, \hth_i}}
\mathcal{L}_{X_{\lambda}} \alpha
\end{equation}
in two 
different ways. First, notice that by Cartan's fomula, we have 
\begin{equation*}
\int_{K \circ \sigma_{i, \hth_i}} \mathcal{L}_{X_\lambda}\alpha = \int_{K \circ \sigma_{i, \hth_i}} i_{X_\lambda} \Omega. 
\end{equation*}
Expanding this last expression in terms of $\lambda$ yields
 \begin{equation*}
\int_{K \circ \sigma_{i, \hth_i}} i_{X_\lambda} \Omega = \int_{K \circ \sigma_{i, \hth_i}} i_{X_\lambda^*} \Omega+\int_{K \circ \sigma_{i, \hth_i}} \frac{d}{d\lambda }\Big (i_{X_\lambda } \Omega \Big ) |_{\lambda=\lambda^*}  (\lambda-\lambda^*)+ O(|\lambda-\lambda^*|^2). 
\end{equation*}
Furthermore, since the vector-field $X_{\lambda^*}$ is exact symplectic, we have $\mathcal L_{X_\lambda^*} \alpha =dW$ and then the first term in the right-hand side vanishes. We are led to 
\begin{equation}\label{cartan-temp1}
\int_{K \circ \sigma_{i, \hth_i}} i_{X_\lambda} \Omega = \int_{K \circ \sigma_{i, \hth_i}} \frac{d}{d\lambda }\Big (i_{X_\lambda } \Omega \Big )|_{\lambda=\lambda^*} (\lambda-\lambda^*)+ O(|\lambda-\lambda^*|^2). 
\end{equation}

On the other hand, using the linearity of the Lie derivative w.r.t. the vector-field and equation \eqref{temp}, we have 
\begin{equation*}
\int_{K \circ \sigma_{i, \hth_i}} \mathcal{L}_{X_\lambda}\alpha = \int_{K \circ \sigma_{i, \hth_i}} i_{X_\lambda} \Omega= \int_{ K \circ \sigma_{i, \hth_i}} i_{\partial_\omega } \Omega +R,
\end{equation*}
where $R$ is such that $\|R\|_\rho \leq C \|E\|. $

Furthermore, by the change of variables formula and the exact symplecticness of the manifold we have  
\[
\begin{split}
\int_{K \circ \sigma_{i, \hth_i}} i_{\partial_\omega } \Omega  
& = 
\int_{\sigma_{i, \hth_i}} i_\omega K^* \Omega \\ 
&= 
\int_{\sigma_{i, \hth_i}} i_\omega d K^* \alpha .\\ 
\end{split}
\]

Since $\omega$ is constant, the exterior differentiation commutes with the contraction operator and one gets for all $1\leq i \leq l$
$$\int_{K \circ \sigma_{i, \hth_i}} i_{\partial_\omega } \Omega =\int_{\sigma_{i, \hth_i}} di_\omega  K^* \alpha =0,  $$
yielding 
\begin{equation}\label{cartan-temp2}
\int_{K \circ \sigma_{i, \hth_i}} \mathcal{L}_{X_\lambda}\alpha =R .
\end{equation}

We note that $i$-component of the map \eqref{cohomol-map} is the integral of
$$
\xi \mapsto \frac{d}{d\lambda}\Big ( K^*i_{X_{\lambda }} \Omega \Big )_{\mid \lambda=\lambda^*} \xi
$$
over the $i-th$ generator of the torus.
Then, from \eqref{cartan-temp1}-\eqref{cartan-temp2} and the implicit function theorem, we get the desired result  if $|\lambda-\lambda^*|$ is small.
\end{proof}
We are now in position to prove Theorem \ref{existenceNH}. 

\vspace{1cm}

\begin{proof}[Proof of Theorem \ref{existenceNH}]

We have to introduce a  family of vector-fields
$X_\lambda $ satisfying the non-degeneracy condition \eqref{cohomol-map} 
 in Lemma \ref{vanishingflow}. 
 Since $\Omega $ is non-degenerate, given a family of closed 1-forms $\sigma_\lambda$ such that 
$\sigma_0=0$  and an exact symplectic vector-field $X$ there exists a family
of symplectic vector-fields $X_\lambda$ such that 
\begin{itemize}
\item[1)] $X_0 = X$.
\item[2)] $i_{X_\lambda} \Omega = \sigma_\lambda$.
\end{itemize}
Condition 2) implies that $X_\lambda$ is indeed symplectic:
$$
\L_{X_\lambda} \Omega = d i_{X_\lambda} \Omega +  i_{X_\lambda} d\Omega = d\sigma_\lambda = 0.
$$
If we choose $\sigma_\lambda $ such that the cohomology class $[\sigma_\lambda] \ne 0$ for $\lambda \ne 0$
then $X_\lambda $ will not be exact symplectic for $\lambda \ne 0$. Indeed, this follows form the 
calculation
$$
\L_{X_\lambda} \alpha = d i_{X_\lambda}  \alpha +  i_{X_\lambda} d \alpha = 
d i_{X_\lambda}  \alpha +  i_{X_\lambda} \Omega = dW_\lambda + \sigma_\lambda.
$$
To choose $\sigma_\lambda$ consider the torus $K(\TT^l)$. We take a tubular neighborhood $N^\ep(K(\TT^l))$ of $K(\TT^l)$.
Since it is  contractible to  $K(\TT^l)$, then  
$H^1(K(\TT^l)) \sim   H^1(N^\ep(K(\TT^l))) $. Now
we consider 
a basis $\{\delta_j\}_{1\le j\le l} $ of $H^1(K(\TT^l))  $.
We define 
$$
\sigma_\lambda = \sum _{i=1}^l \lambda_i \delta_i. 
$$ 
Then we have
$$
\frac{d}{d\lambda_j} K^*i_{X_\lambda} \Omega = 
\frac{d}{d\lambda_j} K^* (\sum _{i=1}^l \lambda_i \delta_i) 
= K^* (\delta_j). 
$$
Since $K$ is an embedding, $\{K^* \delta_j\}_{1\le j\le l}$ is a basis of 
$H^1(\TT^l)$ and then the map $v\mapsto D_\lambda(K^*i_{X_\lambda} \Omega ) v$ is invertible.
\end{proof}

\section{Finite-dimensional Hamiltonian flows}\label{secHamilFD}

This section is devoted to the application of our method in
the Hamiltonian vector-field case. In the same spirit as the previous
section, one of the motivations is the study of Hamiltonian
PDE's.

The result for Hamiltonian flows is based on the study of the
equation 
\begin{equation}\label{hamil}
\partial_{\omega} K(\th)=J(K(\th))\nabla H(K(\th)),
\end{equation}
where the function $H:\mathcal{M} \rightarrow \mathbb{R}$ is the
  Hamiltonian which is supposed to be real analytic.

Equation (\ref{hamil}) expresses the invariance of the range of $K$ under
the Hamiltonian vector-field $X_H= J \nabla H$. We assume that $\mathcal{M}$ is
endowed with $\Omega= dx \wedge dy$ and $\alpha=-ydx$
and hence $J$ is constant. 
Note that the vector-field $J\nabla H$ is exact symplectic. Indeed, by definition, we have 
\begin{equation*}
i_{J\nabla H}\Omega=-dH. 
\end{equation*}
Then taking $W=-H+i_{J\nabla H} \alpha$, we have $\mathcal{L}_X
\alpha= dW$. More generally, if we consider an exact
symplectic vector-field $X$ in the sense of Definition \ref{exactVF}, then there exists
a function $H$ such that $X=J\nabla H$. 

Consequently, the Hamiltonian framework fits exactly in the exact case
as described in the previous section (due to the lack of cohomology
obstruction). However since equations of the type \eqref{hamil} occur
in a lot of physical contexts, our motivation to write
this section is to provide the formulas showing up for this type of
system.  

Again, the linearized equation 
\begin{equation}\label{var}
\frac{d\Delta}{dt}=JD\nabla H(K(\th+\omega t))\Delta
\end{equation}
plays a crucial role. 
Since $A(\th )\equiv JD\nabla H(K(\th))$ is bounded, equation \eqref{var} admits an evolution operator, denoted $U_{\th}(t)$. We have
\begin{equation*}
\frac{d}{dt}U_{\th}(t)=A(\th+ \omega t)U_{\th}(t),
\end{equation*} 
and $U_{\th}(0)=\Id$.
We now have the following definitions. 
\begin{condition}\label{NDVF}
\begin{itemize}
\item  {\bf Spectral conditions:} the evolution operator
  $U_{\th}(t)$ satisfies the non-degeneracy Conditions
  \ref{ND1VFNH}.
\item {\bf Twist condition:} let 
$N(\th)=[DK(\th)^\top DK(\th)]^{-1}$
  and $P(\th)=DK(\th)N(\th)$. The
  average on $\torus^l$ of the matrix 
\begin{equation*}
S(\th)=N(\th)DK(\th)^\top [A(\th)J-JA(\th)]DK(\th)N(\th).
\end{equation*}
is non-singular. Here $A(\th)=JD\nabla H(K(\th))$.
\end{itemize}    
\end{condition} 
For the sake of completeness, we state a theorem for equation
\eqref{hamil}. It provides the existence of
invariant tori. 

\begin{remark}
To obtain the expression for $S$, we used the fact that $DK^\top JDK=0$ and $J^{-1}=-J$. 
\end{remark}
   
\begin{thm}\label{existenceHamil}
Let $\omega$ satisfy the Diophantine Condition \ref{rotVF}. Assume the following hypotheses
\begin{itemize}
\item The embedding $K_0$ satisfies the non-degeneracy Condition  \ref{NDVF}.   

\item The map $H$ is real analytic and it can be extended holomorphically to some complex neighborhood of the image under $K_0$ of $D_{\rho_0}$: 
\begin{equation*}
B_r=\left \{ z \in \complex^{2d} |\; \exists \th \in \{|\Im \, \th| < \rho_0\} \; 
s.t.\; |z-K_0(\th)| <r \right \}, 
\end{equation*}   
for some $r>0$.
\end{itemize}   
Define the error $E_0$ by
\begin{equation*}
E_0=\partial_{\omega} K_0 (\th)-J\nabla H (K_0(\th)).  
\end{equation*}
There exists a constant $C>0$ depending on $l$, $\kappa$, $\nu$,
$|H|_{C^3(B_r)}$, $\|DK_0\|_{\rho_0}$,
$\|N_0\|_{\rho_0}$, 
$\|S_0\|_{\rho_0}$, $|(\avg (S_0))^{-1}|$
(where $S_0$ and $N_0$ are as
in Condition \ref{NDVF} replacing $K$ by $K_0$) and the norms of the
projections $\|\Pi^{c,s,u}_{K_0(\th)}\|_{\rho_0}$ such that, if $E_0$ satisfies the  estimates
\begin{equation*}
C\kappa^4 \delta^{-4\nu} \|E_0\|_{\rho_0} <1
\end{equation*}
and 
\begin{equation*}
C\kappa^2 \delta^{-2\nu} \|E_0\|_{\rho_0} <r,
\end{equation*}
where $0 < \delta \le \min(1,\rho_0/12)$ is fixed, then there exists an
embedding $K_{\infty} $
and $\lambda_\infty \in \RR^l$ such that 
$(\lambda_{\infty},K_{\infty})\in ND(\rho_{\infty}:=\rho_0-6\delta)$ and
\begin{equation}\label{transSolHamil2}
\partial_{\omega} K_{\infty}(\th)=J\nabla H(K_{\infty}(\th)).
\end{equation}
Furthermore, we have the estimate   
\begin{equation*}
\|K_{\infty}-K_0\|_{\rho_{\infty}} \leq C \kappa^{2} \delta^{-2\nu} \|E_0\|_{\rho_0}. 
\end{equation*} 
 
\end{thm}

\begin{remark} One could also formulate a local uniqueness result in
the case of vector-fields. 
\end{remark}

\section*{Acknowledgments}
E.F. acknowledges the support of the Spanish Grant MEC-FEDER
MTM2006-05849/Consolider and the Catalan grant CIRIT 2005 SGR01028 and the hospitality
of the University of Texas at Austin.
R.L. has been supported by NSF grants. Several visits
of R.L to Barcelona supported by MCyT-FEDER grant
MTN2006-00478 were useful for this project.
Y.S. would like to thank the hospitality of the
Department of Mathematics of
the  University of Texas at Austin and the Departament of Matematica 
Aplicada i Analisis of the Universitat de Barcelona. 

\bibliographystyle{alpha}
\bibliography{llave99,new}

\end{document}